\documentclass[12pt]{amsart}
\usepackage{geometry}     
\usepackage{graphicx}
\usepackage{color}
\usepackage{amsthm}
\usepackage{amsmath}
\usepackage{amsfonts}
\usepackage{amssymb}
\usepackage{epstopdf}
\usepackage{hyperref}
\usepackage{mathrsfs}
\usepackage{enumitem}

\newtheorem{theorem}{Theorem}[section]
\newtheorem{proposition}[theorem]{Proposition}
\newtheorem{lemma}[theorem]{Lemma}

\DeclareMathOperator{\Vol}{Vol}
\DeclareMathOperator{\interior}{int}

\DeclareMathOperator{\inj}{inj}
\DeclareMathOperator{\dVol}{dVol}
\DeclareMathOperator{\dL}{dL}
\DeclareMathOperator{\length}{L}
\DeclareMathOperator{\Image}{Im}

\DeclareMathOperator{\support}{spt}

\DeclareMathOperator{\mass}{\mathbf{M}}
\DeclareMathOperator{\diameter}{diam}

\DeclareMathOperator{\dist}{dist}

\DeclareMathOperator{\faces}{Faces}
\DeclareMathOperator{\radius}{rad}
\DeclareMathOperator{\proj}{proj}
\DeclareMathOperator{\Lip}{Lip}
\DeclareMathOperator{\Center}{Center}
\DeclareMathOperator{\Adj}{Adj}

\DeclareMathOperator{\Exp}{Exp}

\newcommand{\F}{\mathcal{F}}

\theoremstyle{definition}
\newtheorem{definition}[theorem]{Definition}

\newtheorem{remark}[theorem]{Remark}

\usepackage{booktabs,amsmath}

\def\XXint#1#2#3{{\setbox0=\hbox{$#1{#2#3}{\int}$}
    \vcenter{\hbox{$#2#3$}}\kern-.5\wd0}}

\def\YYint#1#2#3{{\setbox0=\hbox{$#1{#2#3}{\int}$}
    \lower1ex\hbox{$#2#3$}\kern-.46\wd0}}
\def\YYYint#1#2#3{{\setbox0=\hbox{$#1{#2#3}{\int}$}
    \lower0.35ex\hbox{$#2#3$}\kern-.48\wd0}}

\def\ZZint#1#2#3{{\setbox0=\hbox{$#1{#2#3}{\int}$}
    \raise1.15ex\hbox{$#2#3$}\kern-.57\wd0}}
\def\ZZZint#1#2#3{{\setbox0=\hbox{$#1{#2#3}{\int}$}
    \raise0.85ex\hbox{$#2#3$}\kern-.53\wd0}}

\title{Weyl law for 1-cycles}
\author{Bruno Staffa}

\begin{document}
\maketitle

\begin{abstract}
    We prove the Weyl law for the volume spectrum for $1$-cycles in $n$-dimensional manifolds which was conjectured by Gromov. We follow the strategy of Guth and Liokumovich of obtaining the Weyl law from parametric versions of the coarea inequality and the isoperimetric inequality. A version of the later for families of $0$-cycles is shown in this article. We also obtain approximation results by $\delta$-localized families, which are used to prove the parametric inequalities.

\end{abstract}


\tableofcontents

\section{Introduction}

Let $(M,g)$ be an $n$-dimensional compact Riemannian manifold, possibly with boundary. In the 1960’s Almgren initiated a program of developing a Morse theory on the space
$\mathcal{Z}_{k}(M;G)$ of flat cycles with coefficients in an abelian group $G$ ({see} \cite{FA62},\cite{FA65}). It follows from Almgren’s work that for each $0\leq k\leq n-1$ there is a non-trivial cohomology
class $\lambda_{k}\in H^{n-k}(\mathcal{Z}_{k}(M;\mathbb{Z}_{2});\mathbb{Z}_{2})$ such that the following is true: a continuous  $(n-k)$-parameter family of $k$-cycles $F:X\to\mathcal{Z}_{k}(M;\mathbb{Z}_{2})$ is a sweepout of $M$ if and only if $F^{*}(\lambda_{k})\neq 0$ in $H^{n-k}(X;\mathbb{Z}_{2})$ (see \cite{GuthMinMax}).
Moreover, all cup powers $\lambda_{k}^{p}$ are also non-trivial and we call the corresponding families of
cycles $p$-sweepouts (namely, $F:X\to\mathcal{Z}_{k}(M;\mathbb{Z}_{2})$ is a $p$-sweepout if and only if $F^{*}(\lambda_{k}^{p})\neq 0$). We denote $\mathcal{P}_{p}^{k}$ the set of all $p$-sweepouts of $M$ by families of $k$-cycles. Corresponding to each $p$ we can define the $p$-width

\begin{equation*}
    \omega_{p}^{k}(M,g)=\inf_{F\in\mathcal{P}_{p}^{k}}\sup_{x\in X}\{\Vol_{g}(F(x))\}.
\end{equation*}
The $p$-widths correspond to volumes of certain generalized minimal submanifolds
that arise via a min-max construction. In dimension $k=n-1$, $3\leq n\leq 7$, they are
smooth minimal hypersurfaces \cite{JP81},\cite{SchoenSimon} and for $k=1$ they are stationary geodesic nets \cite{JP73}. The study of $p$-widths led to many remarkable breakthroughs in recent years, including the resolution of Yau’s conjecture on existence of infinitely many minimal surfaces by Irie-Marques-Neves \cite{IMN} in the generic case (see also \cite{LS}  and \cite{LiSta} for similar results on stationary geodesic nets) and by Song \cite{Song} for arbitrary metrics. The related Multiplicity One Conjecture was proved by Chodosh-Mantoulidis \cite{ChoMan20} in the $3$-dimensional case and by Xin Zhou \cite{ZhouMult1} and Marques-Neves \cite{MNMult1} for $3\leq n\leq 7$. Gromov suggested to view these widths, or “volume spectrum”, as non-linear
analogs of the eigenvalues of the Laplacian (see \cite{Gromov86}, \cite{Gromov02}, \cite{Gromov09}). This framework is very
useful for obtaining new results about widths and some other non-linear min-max
geometric quantities. In line with this analogy Gromov conjectured that $\omega_{p}^{k}$'s satisfy an asymptotic Weyl law:
\begin{equation*}
    \lim_{p\to\infty}\omega_{p}^{k}(M,g)p^{-\frac{n-k}{n}}=\alpha(n,k)\Vol(M,g)^{\frac{k}{n}}
\end{equation*}
for a certain universal constant $\alpha(n,k)>0$ depending only on $n$ and $k$. In \cite{LMN}, the conjecture was resolved for domains of $\mathbb{R}^{n}$ for arbitrary $n$ and $k$, and for $k=n-1$ for smooth compact manifolds. Later Guth and Liokumovich solved the case $k=1$, $n=3$ on compact manifolds in their article \cite{GL22}. In the present work we solve the conjecture for $k=1$ and $n$ arbitrary.

\begin{theorem}[Weyl law for $1$-cycles]\label{Thm Weyl law}
    For each $n\in\mathbb{N}$, there exists a constant $\alpha(n)>0$ such that the following is true. Let $(M,g)$ be a compact $n$-dimensional Riemannian manifold, possibly with boundary. Then
    \begin{equation*}
        \lim_{p\to\infty}\omega_{p}^{1}(M,g)p^{-\frac{n-1}{n}}=\alpha(n)\Vol(M,g)^{\frac{1}{n}}.
    \end{equation*}
\end{theorem}

In \cite{MNS}, Marques, Neves and Song used the Weyl law for codimension-$1$ cycles proved in \cite{LMN} to show generic equidistribution of minimal hypersurfaces in ambient manifolds of dimension $3\leq n\leq 7$. Later, in a work with Xinze Li (\cite{LiSta}) we extended that result to the case $n=2$, proving generic equidistribution of immersed closed geodesics on surfaces. For that, we applied the following result of Chodosh and Mantoulidis which is proved in \cite{ChoMan}: on closed surfaces, the $p$-widths $\omega_{p}^{1}(M,g)$ are realized by unions of immersed closed geodesics. When $n\geq 3$, they are realized by stationary geodesic nets, which are embedded graphs in $M$ that are stationary with respect to the length functional. One could ask whether an equidistribution result like the one showed by Marques, Neves and Song holds for $k$-dimensional minimal surfaces (possibly with singularities) in $n$-dimensional manifolds, $k<n-1$. In \cite{LiSta}, we solved the case $k=1$ and $n=3$ using the Weyl law for $1$-cycles in $3$-manifolds from \cite{GL22}; proving generic equidistribution of stationary geodesic nets in $3$-manifolds. The only reason to restrict ourselves to $n=3$ was that the Weyl law for $1$-cycles was still open for $n\geq 4$. Using Theorem \ref{Thm Weyl law} and \cite{LiSta}[Theorem~1.5], we can now extend that result to $n\geq 4$, obtaining the following.

\begin{theorem}
Let $M^n$, $n\geq 3$ be a closed manifold. Then for a Baire-generic set of $C^{\infty}$ Riemannian metrics $g$ on $M$, there exists a set of connected embedded stationary geodesic nets that is equidistributed in $M$. Specifically, for every $g$ in the generic set, there exists a sequence $\{\gamma_i:\Gamma_i \rightarrow M\}$ of connected embedded stationary geodesic nets in $(M, g)$, such that for every $C^\infty$ function $f:M \rightarrow \mathbb{R}$ we have
\begin{equation*}
    \lim_{k \rightarrow \infty}\frac{\sum_{i = 1}^k\int_{\gamma_i}f \dL_g}{\sum_{i = 1}^k \length_{g}(\gamma_i)} = \frac{\int_M f \dVol_g}{\Vol(M, g)}.
\end{equation*}
\end{theorem}

The proof of Theorem \ref{Thm Weyl law} is discussed in Section \ref{Section Weyl law}. We follow the strategy proposed in \cite{GL22} which consists of obtaining the Weyl law from parametric versions of the Isoperimetric Inequality and of the Coarea Inequality. For that purpose, we prove the following theorem, which solves \cite{GL22}[Conjecture~1.3] for $k=0$.

\begin{theorem}[Parametric Isoperimetric Inequality for $0$-cycles]\label{Thm Parametric Isoperimetric}
    Let $(M,g)$ be a compact $n$-dimensional Riemannian manifold (possibly with boundary), $n\geq 2$. Then there exists a constant $C=C(M,g)$ such that the following holds. Let $X=X^{p}$ be a cubical complex of dimension $p$ and let $F:X^{p}\to\mathcal{Z}_{0}(M;\mathbb{Z}_{2})$ be a contractible family of $0$-cycles, $\dim(X)=p$. Denote
    \begin{equation*}
            \mass_{0}=\sup\{\mass(F(x)):x\in X\}.
        \end{equation*}
    Then there exists a family $G:X^{p}\to\mathcal{I}_{1}(M;\mathbb{Z}_{2})$ such that
    \begin{enumerate}
        \item $\partial G(x)=F(x)$ and
        \item $\mass(G(x))\leq C(M,g)(\mass_{0}p^{-\frac{1}{n}}+p^{\frac{n-1}{n}})$
    \end{enumerate}
    for every $x\in X$. The same result holds for families of relative cycles $F:X^{p}\to\mathcal{Z}_{0}(M,\partial M;\mathbb{Z}_{2})$.
\end{theorem}

The proof of Theorem \ref{Thm Parametric Isoperimetric} is discussed in Section \ref{Section isoperimetric inequality}. In Section \ref{Section Isoperimetric in the disk}, we prove the result for families on the Euclidean disk $\mathbb{D}^{n}$ and then in Section \ref{Section Isoperimetric in manifolds} we extend it to any manifold $M$ by considering a fine triangulation $\{Q_{i}\}_{1\leq i\leq N}$ of $M$ and applying the result for the disk to families $F_{i}(x)$ with $\support(F_{i}(x))\subseteq Q_{i}$ and $\sum_{i=1}^{N}F_{i}(x)=F(x)$. The previous reduces the problem to filling a family in the $(n-1)$-skeleton of $M$, allowing us to argue by induction. Breaking the family $F$ into continuous pieces $F_{i}(x)$ supported in the $Q_{i}$ yields the following obstacle: $\mass(F_{i}(x)\llcorner\partial Q_{i})$ might be very large, making the fillings on each $Q_{i}$ very long. In Section \ref{Section Two important propositions}, we prove Proposition \ref{Avoid ball D^n} which states that each $F_{i}$ can be slightly perturbed to a continuous family of absolute $0$-cycles $F_{i}'$ such that $\mass(F_{i}'(x)\llcorner B)$ is controlled for some small ball $B\subseteq\partial Q_{i}$, which turns out to be enough to resolve the previously mentioned obstacle.

To prove the Weyl law, we also use the Parametric Coarea Inequality for $1$-cycles in $n$-manifolds, which is proved in \cite{StaCoarea}. It resolves \cite{GL22}[Conjecture 1.6] in the case $k=1$ and $n$ arbitrary. We state it below.

\begin{theorem}\label{Thm Parametric Coarea}
    Let $M$ be a compact $n$-dimensional Riemannian manifold with boundary, $n\geq 4$. Let $\alpha\in(0,1)$. There exist $p_{0}=p_{0}(M,\alpha)$, a constant $c=c(n)$ and a sequence of numbers $\gamma_{p}=\gamma_{p}(n,\alpha)$ converging to $0$ such that the following is true. Let $p\geq p_{0}$ and let $F:X^{p}\to\mathcal{Z}_{1}(M,\partial M;\mathbb{Z}_{2})$ be a family of relative $1$-cycles which is continuous in the flat topology and has no concentration of mass. Then for every $\varepsilon>0$, there exists a continuous family $F':X^{p}\to\mathcal{I}_{1}(M;\mathbb{Z}_{2})$ of absolute chains in $M$ such that
    \begin{enumerate}
        \item $\support(\partial F'(x))\subseteq\partial M$. Therefore $F'$ induces a family of relative $1$-cycles $F':X^{p}\to\mathcal{Z}_{1}(M,\partial M;\mathbb{Z}_{2})$ which is denoted in the same way.
        \item $\mathcal{F}(F(x),F'(x))\leq\varepsilon$ as relative cycles in the space $\mathcal{Z}_{1}(M,\partial M;\mathbb{Z}_{2})$.
        \item $\mass(F'(x)) \leq(1+\gamma_{p})\mass(F(x))+c\Vol_{n-1}(\partial M)p^{\frac{n-2}{n-1}+\alpha}$.
        \item $\mass(\partial F'(x)) \leq c\Vol_{n-1}(\partial M)p^{1+\alpha}$.
    \end{enumerate}
\end{theorem}

In order to obtain the parametric inequalities, it is convenient to work with $\delta$-localized families, which were first defined in \cite{GL22}[Section~2]. These are families $F:X\to\mathcal{Z}_{k}(M,\partial M;\mathbb{G})$ such that for every cell $C$ of $X$ and for every $x,y\in C$, $F(x)-F(y)$ is supported in a collection of balls with sum of radii less than $\delta$ (see Definition \ref{Def delta localized}). $\mathbb{G}$ can be taken to be $\mathbb{Z}$ or $\mathbb{Z}_{p}$ for some prime $p$. Having this control over how the family $F$ behaves on the cells of $X$ allows us to produce new families with desired mass bounds and additional properties. For example, it permits us to construct families inductively skeleton by skeleton, starting from a discrete family defined in the $0$-skeleton of $X$, keeping track of quantities like the masses of the chains involved, of its boundaries, $\mathcal{F}(F(x)-F(y))$ for $x,y\in C$ and the supports of the $F(x)$. In order to use these families for our constructions, we need to show that every continuous family can be arbitrarily well approximated by a $\delta$-localized one. This is proved both for families of cycles and for families of chains, as stated below.

\begin{theorem}
    Let $\mathbb{G}=\mathbb{Z}$ or $\mathbb{Z}_{p}$ and let $F:X\to\mathcal{Z}_{k}(M,\partial M;\mathbb{G})$ be a continuous map without concentration of mass. Let $\varepsilon,\delta>0$. Then there exists a continuous map $F':X\to\mathcal{Z}_{k}(M,\partial M;\mathbb{G})$ without concentration of mass and a fine cubulation of $X$ such that
    \begin{enumerate}
        \item $F'$ is $\delta$-localized.
        \item $\mathcal{F}(F(x),F'(x))\leq\varepsilon$ for every $x\in X$.
        \item If $x\in C$ for some cell $C$ of $X$ then 
        \begin{equation*}
            \mass(F'(x))\leq\max\{\mass(F(v)):v\in V(C)\}+\varepsilon.
        \end{equation*}
    \end{enumerate}
\end{theorem}

\begin{theorem}
    Let $\mathbb{G}=\mathbb{Z}$ or $\mathbb{Z}_{p}$ and let $G:X\to\mathcal{I}_{k+1}(M,\partial M;\mathbb{G})$ be a continuous map without concentration of mass. Denote $F=\partial G$ and let $\varepsilon,\delta>0$. Then there exists a continuous map $G':X\to\mathcal{I}_{k+1}(M,\partial M;\mathbb{G})$ without concentration of mass and a fine cubulation of $X$ such that if $F'=\partial G'$ then
    \begin{enumerate}
        \item $G'$ and $F'$ are $\delta$-localized.
        \item $\mathcal{F}(G(x),G'(x))\leq\varepsilon$ and $\mathcal{F}(F'(x),F(x))\leq\varepsilon$ for every $x\in X$.
        \item If $x\in C$ for some cell $C$ of $X$ then 
        \begin{equation*}
            \mass(G'(x))\leq\max\{\mass(G(v)):v\in V(C)\}+\varepsilon.
        \end{equation*}
        and
        \begin{equation*}
            \mass(F'(x))\leq\max\{\mass(F(v)):v\in V(C)\}+\varepsilon.
        \end{equation*}
    \end{enumerate}
\end{theorem}

In order to prove the previous theorems, we first obtain versions of them for discrete families $F:X_{0}\to\mathcal{Z}_{k}(M,\partial M;\mathbb{G})$ (resp. $\mathcal{I}_{k+1}(M,\partial M;\mathbb{G})$) and then use the techniques to go from discrete to continuous families developed in \cite{GL22}[Section~2]. All this is discussed in Section \ref{Section delta localized families}.

\textbf{Acknowledgments.} I am very grateful to Yevgeny Liokumovich for suggesting this problem and for all the valuable conversations we have had while I was working on it. I would also like to express my deep gratitude to the Hausdorff Research Institute for Mathematics, where part of this work was completed during my participation in the Trimester Program: Metric Analysis, funded by the Deutsche Forschungsgemeinschaft (DFG, German Research Foundation) under Germany's Excellence Strategy – EXC-2047/1 – 390685813.

\section{Preliminaries}\label{Section Preliminaries}

\subsection{Piecewise smooth Riemannian manifolds}

We will need to work with a class of Riemannian manifolds with piecewise smooth boundary which admit a compatible PL structure. For that purpose, we introduce the following definitions.

\begin{definition}\label{Def Euclidean polyhedron}
    A Euclidean polyhedron $P$ is a linearly embedded finite simplicial complex in $\mathbb{R}^{N}$ for some $N\in\mathbb{N}$.
\end{definition}

\begin{definition}
    Let $P$ be a Euclidean polyhedron. A map $\phi:P\to\mathbb{R}^{N}$ is piecewise smooth if $\phi$ is continuous and for each face $F$ of a certain triangulation $T$ of $P$ the map $\phi|_{F}$ is smooth. $\phi$ is a piecewise smooth embedding if it is a topological embedding and $\phi|_{F}$ is a smooth embedding for each face $F$ of $P$.
\end{definition}

\begin{definition}
    An $n$-dimensional Riemannian polyhedron is a is a piecewise smoothly embedded $n$-dimensional polyhedron in $\mathbb{R}^{N}$ (i.e. the image of a Euclidean polyhedron under a piecewise smooth embedding into $\mathbb{R}^{n}$ for some value of $N$) equipped with the piecewise smooth Riemannian metric $g$ induced by the Euclidean metric in $\mathbb{R}^{N}$. 
\end{definition}

\begin{definition}\label{Def smooth triangulation}
    Let $P$ be a Riemannian polyhedron embedded in $\mathbb{R}^{N}$ and denote the corresponding embedding by $\iota:P\to\mathbb{R}^{N}$. A smooth triangulation of $P$ is a homeomorphism $\tau:P'\to P$ where $P'$ is a Euclidean polyhedron provided with a triangulation $T'$ and $\iota\circ\tau:P'\to\mathbb{R}^{n}$ is a piecewise smooth embedding with respect to the triangulation $T'$. A triangulated Riemannian polyhedron is a pair $(P,T)$ where $P$ is a Riemannian polyhedron and $T$ is the triangulation on $P$ induced by the $T'$ just defined. A refinement of $T$ is the triangulation induced by taking a refinement of the corresponding $T'$.
\end{definition}

\begin{remark}
    Riemannian polyhedra (and hence Euclidean ones) admit different triangulations. For example, we can obtain infinitely many of them from a single one by refinement. 
\end{remark}


\begin{definition}
    A piecewise smooth Riemannian manifold $(M^{n},g)$ is an $n$-dimensional Riemannian polyhedron provided with a topological manifold with boundary structure (i.e. with an atlas of class $C^{0}$). In other words, it is an $n$-dimensional topological submanifold with boundary of $\mathbb{R}^{N}$ (for some $N\in\mathbb{N}$) which admits a smooth triangulation and is equipped with a continuous Riemannian metric $g$ which is smooth on each simplex of that triangulation.
\end{definition}

\begin{definition}\label{Def almost 1-Lip triangulable}
    A piecewise smooth Riemannian manifold $(M,g)$ is almost $1$-Lipchitz triangulable if for every $\varepsilon>0$ there exists a smooth $(1+\varepsilon)$-bilipschitz triangulation $\tau_{\varepsilon}:P_{\varepsilon}\to M$ (here $P_{\varepsilon}$ is a Euclidean polyhedron according to Definition \ref{Def smooth triangulation}).
\end{definition}

\begin{remark}
    By the Nash Embedding Theorem and Theorems 1.1 and 1.2 of \cite{Bowditch}, every smooth compact Riemannian manifold with boundary is almost $1$-Lipschitz triangulable. So is every piecewise linear submanifold of $\mathbb{R}^{N}$. Thus the collection of all almost $1$-Lipschitz triangulable piecewise smooth Riemannian manifolds includes both the category of compact smooth manifolds with boundary and the category of compact PL manifolds with boundary. It will be the correct set up for our constructions because we are interested in proving results about compact Riemannian manifolds but our methods will involve considering $(1+\varepsilon)$-bilipschitz triangulations.
\end{remark}

\begin{remark}
    Because of the works of Akopyan \cite{Akopyan} and Minemyer \cite{Minemyer}, every Euclidean simplicial complex as defined by Bowditch in \cite{Bowditch} can be piecewise linearly embedded in $\mathbb{R}^{N}$ and hence can be regarded as a Euclidean polyhedron.
\end{remark}

\begin{remark}
    As far as the author knows, the question whether every piecewise smooth Riemannian manifold is almost $1$-Lipschitz triangulable has not been studied. Maybe the methods of Bowditch \cite{Bowditch} also work in this case.
\end{remark}

\begin{remark}
    Notice that in Definition \ref{Def almost 1-Lip triangulable}, $P_{\varepsilon}$ inherits a topological manifold structure from $M$ and hence it is a compact PL manifold with boundary. So the condition of being almost $1$-Lipschitz triangulable can be re-expressed as being $(1+\varepsilon)$-bilipschitz diffeomorphic to a compact PL manifold with boundary for every $\varepsilon>0$. 
\end{remark}

\begin{definition}\label{Def tubular neighborhood}
    Let $(M,g)$ be a piecewise smooth Riemannian manifold. We denote $\Sigma=\partial M$ (which is a subpolyhedron of $M$) and $d$ the metric on $M$ induced by the Riemannian metric $g$. Given two subsets $A,B\subseteq M$ we define
    \begin{equation*}
        \dist(A,B)=\inf\{d(x,y):x\in A,y\in B\}.
    \end{equation*}
    Let $A\subseteq M$ be a subset of $M$. For each $r>0$, we denote
    \begin{align*}
        N_{r}A &=\{x\in M:\dist(x,A)< r\}\\
        D_{r}A & =\{x\in M:\dist(x,A)=r\}.
    \end{align*}
\end{definition}

\begin{definition}\label{Def Exp}
    Given a smooth Riemannian manifold with boundary $(M,g)$, we denote its exponential map by $\Exp$. 
\end{definition}

\subsection{Almgren-Pitts Min-Max Theory}\label{Section Almgren-Pitts}

Let $(M,g)$ be an $n$-dimensional piecewise smooth Riemannian manifold and let $\mathbb{G}=\mathbb{Z}$ or $\mathbb{Z}_{p}$ for some prime $p$. Let $\mathcal{Z}_{k}(M,\partial M)$ be the space of flat $k$-cycles with coefficients in $\mathbb{G}$ on $M$ relative to $\partial M$ and let $\mathcal{I}_{k}(M,\partial M)$ the corresponding space of flat $k$-chains, $0\leq k\leq n$. We also consider the spaces $\mathcal{Z}_{k}(M)$ and $\mathcal{I}_{k}(M)$ of absolute flat cycles and chains on $M$ with coefficients in $\mathbb{G}$ as defined in \cite{Fleming66} and in \cite{FedererGMT}[Section~4.2.26]. The mass functional on the previous spaces with respect to the metric $g$ is denoted by $\mass$. Notice that for a relative chain $\tau\in\mathcal{I}_{k}(M,\partial M)$,
\begin{equation}\label{Mass of a relative chain}
    \mass(\tau)=\inf\{\mass(\tau'):\tau'\in\mathcal{I}_{k}(M),[\tau']=\tau\}
\end{equation}
where given an absolute $k$-chain $\tau'$, $[\tau']$ denotes its class as a relative $k$-chain in $(M,\partial M)$. In addition, $\support(\tau)=\support(\tau'-\tau'\llcorner\partial M)$ for $\tau'\in\mathcal{I}_{k}(M)$ such that $[\tau']=\tau$ (the definition is independent of the representative $\tau'$).
The flat norm associated to $\mass$ is denoted by $\mathcal{F}$, which for absolute $k$-chains is defined as
\begin{equation*}
    \mathcal{F}(\tau)=\inf\{\mass(\alpha)+\mass(\beta):\tau=\alpha+\partial\beta,\alpha\in\mathcal{I}_{k}(M),\beta\in\mathcal{I}_{k+1}(M)\}
\end{equation*}
and the analogous definition holds for relative chains.

\begin{definition}
    We denote $I=[0,1]$ and $I^{d}=[0,1]^{d}$ for each $d\in\mathbb{N}$. In addition, given $q\in\mathbb{N}$ we denote $I(q)$ the complex obtained by dividing $I$ into $q$ equal parts, and $I^{d}(q)=I(q)\times...\times I(q)$ ($d$ times) with the product structure. We say that $I^{d}(q)$ is obtained from $I^{d}$ by performing a $q$-refinement.
\end{definition}

\begin{definition}
    A cubical complex $X$ is a subcomplex of $I^{d}(q)$ for some $d,q\in\mathbb{N}$. Given $q'\in\mathbb{N},$ we denote $X(q')$ the complex obtained by performing a $q'$-refinement of each cell of $X$. Notice that $X(q')$ is a subcomplex of $I^{d}(qq')$.
\end{definition}

\begin{remark}
    The convention in the previous two definitions is different from that in 
    \cite{MN}, \cite{Willmore} and \cite{GL22}.
\end{remark}

\begin{definition}
    Given a $p$-dimensional cubical complex $X$ and $0\leq j\leq p$, we denote by $X_{j}$ the $j$-skeleton of $X$.
\end{definition}

\begin{definition}
    Given a cubical complex $X$ and a cell $C$ of $X$, we denote $V(C)$ the set of vertices of $C$.
\end{definition}

\begin{definition}
    Given a $p$-dimensional cubical complex $X$ and $0\leq j\leq p$, we denote by $\faces_{j}(X)$ the set of $j$-dimensional faces of $X$.
\end{definition}

\begin{definition}
    Let $X$ be a cubical complex, let $F:X_{0}\to\mathcal{Z}_{k}(M,\partial M)$ ($\mathcal{I}_{k}(M,\partial M)$, $\mathcal{Z}_{k}(M)$ or $\mathcal{I}_{k}(M)$ respectively) be a discrete family and let $q$ be an odd natural number. We define the $q$-refinement of $F$ to be the discrete family $R^{q}F:X(q)_{0}\to\mathcal{Z}_{k}(M,\partial M)$ given by $R^{q}F(x)=F(v)$ where $v$ is the point of $X_{0}$ closest to $x\in X(q)_{0}$.
\end{definition}

\begin{definition}
    Let $X$ be a $p$-dimensional cubical complex, let $0\leq j\leq p$ and  let $F:X_{j}\to\mathcal{Z}_{k}(M)$ be a continuous map in the flat topology. We say that $F$ is $\varepsilon$-fine if for every cell $C$ of $X$ and every $x,y\in C\cap X_{j}$
    \begin{equation*}
        \mathcal{F}(F(x),F(y))\leq\varepsilon.
    \end{equation*}
    When $j=p$, we say that $F$ has no concentration of mass if
    \begin{equation*}
        \limsup_{r\to 0}\sup_{x\in X}\sup_{p\in M}\{\mass(F(x)\llcorner B(p,r)\}=0.
    \end{equation*}
\end{definition}

\begin{definition}
    For $\mathbb{G}=\mathbb{Z}_{2}$, we denote by $\lambda_{k}\in H^{n-k}(\mathcal{Z}_{k}(M,\partial M;\mathbb{Z}_{2});\mathbb{Z}_{2})$ the fundamental cohomology class of $\mathcal{Z}_{k}(M,\partial M;\mathbb{Z}_{2})$ (see \cite{GuthMinMax} and \cite{LS}[Section~3]). It is defined by the property that a continuous family $F:X\to\mathcal{Z}_{k}(M,\partial M;\mathbb{Z}_{2})$ is a sweepout of $M$ (i.e. the gluing homomorphism described in \cite{GuthMinMax} is nontrivial) if and only if $F^{*}(\lambda_{k})\neq 0$.
\end{definition}

\begin{definition}
    Let $X$ be a cubical complex and $\mathbb{G}=\mathbb{Z}_{2}$. We say that a continuous map $F:X\to\mathcal{Z}_{k}(M,\partial M;\mathbb{Z}_{2})$ is a $p$-sweepout of $(M,\partial M)$ by $k$-cycles if $F^{*}(\lambda_{k}^{p})\neq 0$. We denote $\mathcal{P}_{p}^{k}=\mathcal{P}_{p}^{k}(M,\partial M)$ the collection of $p$-sweepouts on $(M,\partial M)$ by $k$-cycles with no concentration of mass.

\end{definition}

\begin{definition}
     We define the $p$-width $\omega_{p}^{k}(M,g)$ of $(M,g)$ as
    \begin{equation*}
        \omega_{p}^{k}(M,g)=\inf_{F\in\mathcal{P}_{p}^{k}}\sup_{x\in X}\mass(F(x)).
    \end{equation*}
\end{definition}

Notice that in the definition of $p$-sweepout and of $p$-widths, $X$ can be any cubical complex; this means any subcomplex of the $q$-refinement of $[0,1]^{d}$ for $q$ and $d$ arbitrary. Nevertheless, as $\lambda_{k}^{p}\in H^{p(n-k)}(\mathcal{Z}_{k}(M;\mathbb{Z}_{2});\mathbb{Z}_{2})$ is a $p(n-k)$-cohomology class, we can expect the simplicial complexes of dimension $p(n-k)$ to be enough to detect $\lambda_{k}^{p}$. As any $m$-dimensional polyhedron embeds in $\mathbb{R}^{2m+1}$, we should expect to obtain the same definition for the $p$-widths if we restrict $d=2(n-k)p+1$. This is indeed true, as the following proposition shows.

\begin{lemma}\label{Lemma equiv def of widths}
    Denote $\overline{\mathcal{P}}_{p}^{k}$ the set of $F\in\mathcal{P}_{p}^{k}$ such that the domain $X$ of $F$ is a $p(n-k)$-dimensional subcomplex of $I^{2p(n-k)+1}(q)$ for some $q\in\mathbb{N}$. Let
    \begin{equation*}
        \overline{\omega}_{p}^{k}(M,g)=\inf_{F\in\overline{\mathcal{P}}_{p}^{k}}\sup_{x\in X}\mass(F(x)).
    \end{equation*}
    Then $\overline{\omega}_{p}^{k}(M,g)=\omega_{p}^{k}(M,g)$ for every $p,k\in\mathbb{N}$.
\end{lemma}

\begin{proof}
    As $\overline{\mathcal{P}}_{p}^{k}\subseteq\mathcal{P}_{p}^{k}$, it is clear that $\omega_{p}^{k}(M,g)\leq\overline{\omega}_{p}^{k}(M,g)$. To show the other inequality, let $F:X\to\mathcal{Z}_{k}(M,\partial M)$ be an element of $\mathcal{P}_{p}^{k}$. As $F^{*}(\lambda_{k}^{p})\neq 0$, there exists some $p(n-k)$-cycle $z$ in $X$ such that $\langle\lambda_{k}^{p},F_{*}([z])\rangle=0$, where $\langle,\rangle$ denotes the pairing between cohomology and homology in $\mathcal{Z}_{k}(M,\partial M)$. By the equivalence between singular homology and simplicial homology, we can assume that $z$ is a polyhedral cycle in $X$. Denote $Z$ its underlying cubical subcomplex. As $\dim(Z)=p(n-k)$, by the PL embedding theorem we can see $Z$ embedded in $\mathbb{R}^{2p(n-k)+1}$ (as some piecewise linear polyhedron whose cells are arbitrary simplices, not necessarily equilateral). Notice that $F:Z\to\mathcal{Z}_{k}(M,\partial M)$ is a $p$-sweepout. We would like to replace the domain $Z$ by some polyhedral cycle $Z'$ supported in a cubical grid of $\mathbb{R}^{2p(n-k)+1}$ so that we can see $Z'$ as a subcomplex of $I^{2p(n-k)+1}(q)$ for some $q\in\mathbb{N}$. Let $B$ be some tubular neighborhood of $Z$ in $\mathbb{R}^{2p(n-k)+1}$ and $R:B\to Z$ the corresponding retraction. By the Federer-Fleming Deformation Theorem, we can find a polyhedral $k$-cycle $Z'$ and a $(k+1)$-chain $T$ both supported in $B$ such that $\partial T=Z'-Z$. Then $[Z']=[Z]$ in $B$ therefore $R_{*}([Z'])=R_{*}([Z])=[Z]$ in $Z$ which implies that the map $F':Z'\to\mathcal{Z}_{k}(M)$ given by $F'(x)=F(R(x))$ is a $p$-sweepout because
    \begin{equation*}
        \langle\lambda_{k}^{p},F'_{*}([Z'])\rangle=\langle\lambda_{k}^{p},F_{*}([Z])\rangle\neq 0.
    \end{equation*}
    Then by construction $F'\in\overline{\mathcal{P}}_{p}^{k}$ therefore
    \begin{equation*}
        \overline{\omega}_{p}^{k}(M,g)\leq\sup_{x\in Z'}\mass(F'(x))\leq\sup_{x\in X}\mass(F(x)).
    \end{equation*}
    As $F\in\mathcal{P}_{p}^{k}$ is arbitrary, the previous implies $\overline{\omega}_{p}^{k}(M,g)\leq\omega_{p}^{k}(M,g)$.
\end{proof}

\begin{remark}
    As a consequence of Lemma \ref{Lemma equiv def of widths}, in the rest of the paper a cubical complex will denote a $p$-dimensional subcomplex of $I^{2p+1}(q)$ for some $p,q\in\mathbb{N}$. This assumption will simplify some of the proofs, but they could still be carried out without it.
\end{remark}

\section{$\delta$-localized families}\label{Section delta localized families}

\subsection{Definitions}
Throughout this section, $(M,g)$ is a fixed $n$-dimensional smooth Riemannian manifold with boundary $\partial M=\Sigma$. We fix $\mathbb{G}=\mathbb{Z}$ or $\mathbb{Z}_{p}$ for some prime number $p$ and we denote $\mathcal{Z}_{k}(M,\partial M)$ the space of flat relative $k$-cycles (respectively relative chains, and absolute cycles and chains) with coefficients in $\mathbb{G}$ as introduced in the previous section (in the rest of the paper we take $\mathbb{G}=\mathbb{Z}_{2}$, but the results of this section hold for more general $\mathbb{G}$). In fact, all the arguments of this section also work by replacing the pair $(M,\partial M)$ by a pair $(P,P')$ where $P$ is an $n$-dimensional Riemannian polyhedron and $P'$ is a subpolyhedron of $P$. This will be used in Section \ref{Section Isoperimetric in manifolds} when proving the Parametric Isoperimetric Inequality for $0$-cycles. We will introduce the notion of $\delta$-localized family which was first defined in \cite{GL22}[Section~2]. Given an arbitrary continuous family $F:X\to\mathcal{Z}_{k}(M,\partial M)$, there might be arbitrarily close points $x,y\in X$ such that $\mathcal{F}(F(x)-F(y))$ is very small (because of continuity) but $\support(F(x)-F(y))$ is very large; in fact the support of $F(x)-F(y)$ could be almost dense in $M$ (and its mass might be very large too). The property of $\delta$-localization prevents from this to happen, as it requires that if $x$ and $y$ are close then the difference $F(x)-F(y)$ is supported in a collection of small balls. We make this more precise in the definitions below.

\begin{definition}\label{Def delta admissible}
    Let $\delta\leq\inj(M,g)$. A $\delta$-admissible family on $M$ is a finite collection of disjoint open sets $\{U_{i}\}_{i\in I}$ such that
    \begin{enumerate}
        \item If $U_{i}\cap\partial M=\emptyset$ then $U_{i}$ is a ball of radius $r_{i}$.
        \item If $U_{i}$ intersects $\partial M$ then there exists a ball $\beta_{i}$ in $\partial M$ with $\radius(\beta_{i})=r_{i}$ such that $U_{i}=\Exp(\beta_{i}\times[0,r_{i}])$ ($\Exp$ is defined in Definition \ref{Def Exp}).
        \item $\sum_{i\in I}r_{i}<\delta$.
    \end{enumerate}
    An $(N,\delta)$-admissible family is a $\delta$-admissible family such that $|I|\leq N$.
\end{definition}

\begin{remark}
    If $P$ is a Riemannian polyhedron, we just require the $U_{i}$ to be metric balls such that $\sum_{i\in I}r_{i}<\delta$. 
\end{remark}

\begin{definition}\label{Def delta localized}
    Let $X$ be a $p$-dimensional cubical complex and $1\leq j\leq p$. A family of chains $F:X_{j}\to\mathcal{I}_{k}(M)$ is said to be $\delta$-localized if for every cell $C$ of $X$ there exists a $\delta$-admissible family $\{U_{i}^{C}\}_{i\in I_{C}}$ such that for every $x,y\in C\cap X_{j}$
    \begin{equation*}
        \support(F(x)-F(y))\subseteq \bigcup_{i\in I_{C}}U_{i}^{C}.
    \end{equation*}
     If each $\delta$-admissible family $\{U_{i}^{C}\}$ is also $(N,\delta)$-admissible, we say that $F$ is $(N,\delta)$-localized. We introduce the same definition for families of relative chains and for families of absolute and relative cycles by replacing $\mathcal{I}_{k}(M)$ by $\mathcal{I}_{k}(M,\partial M)$, $\mathcal{Z}_{k}(M)$ and $\mathcal{Z}_{k}(M,\partial M)$ respectively.
\end{definition}

\begin{definition}
    In relation to the previous definitions, we say that a family $F:X\to\mathcal{I}_{k}(M)$ (resp. $\mathcal{I}_{k}(M,\partial M)$, $\mathcal{Z}_{k}(M)$ or $\mathcal{Z}_{k}(M,\partial M)$) is localized in a collection of open sets $\{U_{i}\}_{i\in I}$ if for every $x,y\in X$
    \begin{equation*}
        \support(F(x)-F(y))\subseteq\bigcup_{i\in I}U_{i}.
    \end{equation*}
\end{definition}

We will often need to construct $\delta'$-admissible (respectively, $(N',\delta')$-admissible) families from $\delta$-admissible (respectively $(N,\delta)$-admissible) ones. When carrying out such constructions, we normally first need to deal with families of balls which are not pairwise disjoint. The following lemma tells us that we can always get an accurate $\delta$-admissible (resp. $(N,\delta)$-admissible) family out of that.

\begin{lemma}\label{Lemma non disjoint}
    Let $\{B_{j}\}_{1\leq j\leq N}$ be a finite collection of open sets in $M$ such that each $B_{j}$ is either a metric ball in $(M,g)$ of radius $r_{j}$ or it is of the form $B_{j}=\Exp(\beta_{j}\times[0,r_{j}])$ for some ball $\beta_{j}$ in $\partial M$ of radius $r_{j}$. Assume that
    \begin{equation*}
        \sum_{j=1}^{N}r_{j}\leq\frac{\inj(M,g)}{3}.
    \end{equation*}
    Then there exists an $(N,\delta)$-admissible family $\{U_{i}\}_{i\in I}$ such that $\bigcup_{j=1}^{N}B_{j}\subseteq\bigcup_{i\in I}U_{i}$ and
    \begin{equation*}
        \delta\leq 3\sum_{j=1}^{N}\radius(B_{j}).
    \end{equation*}
\end{lemma}

\begin{proof}
    This is a consequence of \cite{GL22}[Lemma~2.4] and \cite{GL22}[Lemma~2.5]. 
\end{proof}

\begin{definition}
    A monotonously $(N,\delta)$-localized family is an $(N,\delta)$-localized family such that given $C\subseteq C'$ cells of $X$,
    \begin{equation*}
        \bigcup_{i\in I_{C}}U_{i}^{C}\subseteq\bigcup_{i\in I_{C'}}U_{i}^{C'}.
    \end{equation*}
    We state the same definition replacing $(N,\delta)$-localized by $\delta$-localized.
\end{definition}

\begin{lemma}\label{Lemma monotonously localized}
    For each $p\in\mathbb{N}$, there exists a constant $c(p)\in\mathbb{N}$ such that the following is true. If $F:X^{p}\to\mathcal{I}_{k}(M)$ (resp. $\mathcal{I}_{k}(M,\partial M)$, $\mathcal{Z}_{k}(M)$ or $\mathcal{Z}_{k}(M,\partial M)$ ) is $(N,\delta)$-localized (resp. $\delta$-localized) then it is monotonously $(c(p)N,c(p)\delta)$-localized (resp. monotonously $c(p)\delta$-localized). 
\end{lemma}

\begin{proof}
    By induction in $p$. For $p=1$, any $(N,\delta)$-localized family is automatically monotonously localized so the result holds for $c(1)=1$. Assume that the result holds for $p-1$. Let $F:X^{p}\to\mathcal{I}_{k}(M)$ be $(N,\delta)$-localized and for each cell $C$ of $X$ denote $\{U_{i}^{C}\}_{i\in I_{C}}$ the $(N,\delta)$-admissible family associated to $C$. Then by inductive hypothesis $F_{p-1}=F|_{X_{p-1}}$ is monotonously $(c(p-1)N,c(p-1)\delta)$-localized. Given $E$ cell of $X$ with $\dim(E)\leq p-1$, denote $\{\tilde{U}^{E}_{i}\}_{i\in\tilde{I}_{E}}$ its corresponding $(c(p-1)N,c(p-1)\delta)$-admissible family. Given $C\in\faces_{p}(X)$, let $\{\tilde{U}^{C}_{i}\}_{i\in\tilde{I}_{C}}$ be an $(c(p)N,c(p)\delta)$ admissible family which covers
    \begin{equation*}
        \bigcup_{i\in I_{C}}U_{i}^{C}\cup\bigcup_{\substack{E\in\faces_{p-1}(X) \\ E\subseteq C}}\bigcup_{i\in\tilde{I}_{E}}\tilde{U}_{i}^{E}.
    \end{equation*}
    By Lemma \ref{Lemma non disjoint}, this can be done provided $c(p)\geq 3(1+2pc(p-1))$, as the  total number of balls in the previous union is $N(1+2pc(p-1))$ (being $2p$ the number of $(p-1)$-faces of $C$) and the sum of their radius is at most $(1+2pc(p-1))\delta$. Thus setting $c(p)=3(1+2pc(p-1))$ we can see that $F$ is $(c(p)N,c(p)\delta)$-monotonously localized in the above defined families $\{\tilde{U}_{i}^{C}\}_{i\in\tilde{I}_{C}}$ for each $C\in\faces(X)$.
\end{proof}

Given a Riemannian polyhedron $P$,  a $\delta$-localized family $G:X\to\mathcal{I}_{k}(P)$ and a top cell $Q$ of $P$, $G$ induces a continuous family $\overline{G}_{Q}:X\to\mathcal{I}_{k}(Q,\partial Q)$ of relative $k$-chains given by $\overline{G}_{Q}(x)=[G(x)\llcorner Q]$. We may wonder whether we can represent $\overline{G}_{Q}$ by a continuous $\delta$-localized family of absolute chains in $Q$. The first attempt is to consider the family $x\mapsto G(x)\llcorner Q$ but this one may not be continuous. However, it is possible to construct such family by only adding $G(x)\llcorner Q$ certain $k$-chains supported in $\partial Q$ as the next lemma shows.

\begin{lemma}\label{Lemma G_Q}
    Let $P$ be a Riemannian polyhedron and let $Q$ be a top dimensional cell of $P$. Let $G:X\to\mathcal{I}_{k}(P)$ be a continuous $\delta$-localized family of absolute $k$-chains, $k\geq 1$. Then there exists a continuous $G_{Q}:X\to\mathcal{I}_{k}(Q)$ such that
    \begin{enumerate}
        \item $G(x)\llcorner Q-G_{Q}(x)$ is supported in $\partial Q$ for every $x\in X$.
        \item Given a cell $C$ of $X$ and $x,y\in C$, $G_{Q}(x)-G_{Q}(y)$ is supported in $\{\tilde{U}_{i}^{C}\}_{i\in I_{C}}$ where $\tilde{U}_{i}^{C}=U_{i}^{C}\cap Q$. Thus $G_{Q}$ is $\delta$-localized.
    \end{enumerate}
    The same result holds taking $P=M$ and taking $Q$ to be a closed embedded $n$-dimensional piecewise smooth Riemannian submanifold of $M$.
    
\end{lemma}

\begin{proof}
    Let $p=\dim(X)$. By Lemma \ref{Lemma monotonously localized}, we can assume that $F$ is monotonously $\delta$-localized, being $\{U_{i}^{C}\}_{i\in I_{C}}$ the $\delta$-admissible family of each cell $C$.
    Given $E\in\faces(X)$ denote $\tilde{I}_{E}=\{i\in I_{E}:U_{i}\cap Q\neq\emptyset\}$ and $\tilde{U}_{i}^{E}=U_{i}^{E}\cap Q$, being $\{\tilde{U}_{i}^{E}\}_{i\in\tilde{I}_{E}}$ a collection of metric balls in $Q$ whose radii sum less than $\delta$. In addition, given a chain $\tau$ supported in $\bigcup_{i\in I_{E}}U_{i}^{E}$ denote
    \begin{equation*}
        r_{E}(\tau)=\sum_{i\in\tilde{I}_{E}}r_{i}^{E}(\tau\llcorner U_{i}^{E})
    \end{equation*}
    where $r_{i}^{E}:U_{i}^{E}\to\tilde{U}_{i}^{E}$ is a Lipschitz retraction mapping $U_{E}\cap(M\setminus Q)$ onto $U_{E}\cap\partial Q$.

    We will construct $G_{Q}$ inductively with the following properties in the $j$-skeleton:
    \begin{enumerate}
        \item $G_{Q}(x)-G(x)\llcorner Q$ is supported in $\partial Q$ for $x\in X_{j}$.
        \item $G_{Q}|_{E}$ is localized in $\{\tilde{U}_{i}\}_{i\in \tilde{I}_{E}}$ for every $E\in\faces_{j}(X)$.
    \end{enumerate}

    At the $0$-skeleton we just define $G_{Q}(x)=G(x)\llcorner Q$. Suppose $G_{Q}$ was defined in the $j$-skeleton and let $E$ be an $(j+1)$-cell of $X$. Pick a vertex $v$ of $E$. As $G|_{E}$ is localized in $\{U_{i}^{E}\}_{i\in I_{E}}$, there exists $G_{E}$ continuous and supported there such that
    \begin{equation*}
        G(y)=G(v)+G_{E}(y)
    \end{equation*}
    for $y\in E$. For $y\in E$ set
    \begin{equation*}
        G_{Q}'(y)=G_{Q}(v)+r_{E}(G_{E}(y))
    \end{equation*}
    which is localized in $\{\tilde{U}_{i}^{E}\}_{i\in\tilde{I}_{E}}$ and verifies
    \begin{align*}
        G_{Q}'(y)-G(y)\llcorner Q & =G_{Q}(v)+r_{E}(G_{E}(y))-G(v)\llcorner Q-G_{E}(y)\llcorner Q \\
        & =r_{E}(G_{E}(y))-G_{E}(y)\llcorner Q \\
        & =r_{E}(G_{E}(y)\llcorner(M\setminus Q))
    \end{align*}
    which is supported in $\partial Q$. But $G_{Q}'$ may not coincide with $G_{Q}$ in $\partial E$. As $G_{Q}(x)$ is localized in $\{\tilde{U}_{i}^{E}\}_{i\in\tilde{I}_{E}}$ for $x\in\partial E$ (by inductive hypothesis and monotonously $\delta$-localization) and $G_{Q}(v)=G_{Q}'(v)$, if $x\in\partial E$ then $G_{Q}(x)-G_{Q}'(x)$ is supported in $\bigcup_{i\in\tilde{I}_{E}}\tilde{U}_{i}^{E}$ and also in $\partial Q$ as
    \begin{equation*}
        G_{Q}(x)-G_{Q}'(x)=[G_{Q}(x)-G(x)\llcorner Q]-[G_{Q}'(x)-G(x)\llcorner Q]
    \end{equation*}
    and both terms on the right are supported in $\partial Q$. Let $H:E\to \mathcal{I}_{k}(\partial Q\cap\bigcup_{i\in\tilde{I}_{E}}\tilde{U}_{i}^{E})$ be the map obtained by contracting the family $G_{Q}(x)-G_{Q}'(x)$ radially on each $\tilde{U}^{E}_{i}\cap\partial Q$ (in particular, $H(x)=G_{Q}(x)-G_{Q}'(x)$ for $x\in\partial E$). Then if we set
    \begin{equation*}
        G_{Q}(y)=G_{Q}'(y)+H(y)
    \end{equation*}
    for $y\in E$ then we obtain an extension of the previously defined $G_{Q}$ in $\partial E$ which verifies properties (1) and (2).
\end{proof}

\subsection{Main Theorems}

In this subsection, we state the approximation results by $\delta$-localized families in the discrete and continuous setting. These are crucial in the proof of the Parametric Isoperimetric Inequality discussed later in this article and also to obtain the Parametric Coarea Inequality in \cite{StaCoarea}. For simplicity, they are stated for families of absolute cycles and chains, but the corresponding relative versions in $(M,\partial M)$ also hold. They apply to any coefficient group $\mathbb{G}$ ($\mathbb{G}=\mathbb{Z}$ or $\mathbb{G}=\mathbb{Z}_{p}$). The manifold $M$ is compact, of dimension at least $2$.

\begin{theorem}\label{Thm filling small families}
    Given $0\leq k\leq n-1$ and $p,N\in\mathbb{N}$ there exist natural numbers $K(N,k,p)$, $n(k,p)$ and positive constants $c(k,p,M)$, $\varepsilon_{0}(k,p,M)$ and $\delta_{0}(k,p,M)$ so that the following is true. Let $\varepsilon\leq\varepsilon_{0}$, $\delta\leq\delta_{0}$ and let $F:X_{0}\to\mathcal{Z}_{k}(M)$ be a discrete $(N,\delta)$-localized family of cycles with $\mathcal{F}(F(x))\leq\varepsilon$ for every $x\in X_{0}$ and $\dim(X)=p$. Then there exists a $(K(N,k,p),K(N,k,p)\delta)$-localized family $\tau:X(q)_{0}\to\mathcal{I}_{k+1}(M)$ such that $\partial\tau(x)=R^{q}F(x)$ for $q=q(k,p,M,\delta)$ and $\mass(\tau(x))\leq c(k,p,M)\frac{\varepsilon}{\delta^{n(k,p)}}$.
\end{theorem}

Roughly speaking, the previous theorem tells us that if we have a discrete family of $k$-cycles $F:X_{0}\to\mathcal{Z}_{k}(M)$ which has flat norm at most $\varepsilon$ at every point and is $\delta$-localized for $\varepsilon$ and $\delta$ small enough, then we can construct a discrete filling with similar properties. To be precise, the filling is a $\delta'$-localized family of $(k+1)$-chains $\tau$ such that up to refinement $\partial \tau(x)=F(x)$, $\delta'$ is controlled in terms of $\delta$ and $\mass(\tau(x))$ is controlled in terms of $\varepsilon$ and $\delta$. Theorem \ref{Thm filling small families} is proved in subsection \ref{Subsection Filling small families} and is used to obtain the following result.

\begin{theorem}\label{Thm discrete delta approx cycles}
    There exist constants $\varepsilon_{0}(k,p,M)>0$, $\delta_{0}(k,p,M)>0$, $C(k,p,M)>0$, $n(k,p)\in\mathbb{N}$ and $N(k,p)\in\mathbb{N}$ such that the following is true. Let $F:X_{0}\to\mathcal{Z}_{k}(M)$ be an $\varepsilon$-fine family, $0<\varepsilon\leq\varepsilon_{0}$, $0\leq k\leq n-1$, $\dim(X)=p$. Let $0<\delta<\delta_{0}$. Then there exists a discrete family $F':X(q)_{0}\to\mathcal{Z}_{k}(M)$ for $q=q(k,p,M,\delta)$ such that
    \begin{enumerate}
        \item $F'(x)=F(x)$ for every $x\in X_{0}$.
        \item Given a cell $C$ of $X$ and $x\in C$,
        \begin{equation*}
            \mass(F'(x))\leq\max_{v\in V(C)}\{\mass(F(v))\}+C(k,p,M)\frac{\varepsilon}{\delta^{n(k,p)}}.
        \end{equation*}
        \item $F'$ is $C(k,p,M)\frac{\varepsilon}{\delta^{n(k,p)}}$-fine in $X$ and $(N(k,p),\delta)$-localized in $X(q)$.
    \end{enumerate}
\end{theorem}

We can interpret the previous theorem as that every $\varepsilon$-fine discrete family $F$ of $k$-cycles can be refined to an $\varepsilon'$-fine and $\delta$-localized $F'$ without increasing the mass more than $\varepsilon'$, for $\varepsilon'$ controlled in terms of $\varepsilon$ and $\delta$. The proof is discussed in subsection \ref{Subsection approximating families of cycles}. The following is a generalization of Theorem \ref{Thm discrete delta approx cycles} for families of chains and is discussed in subsection \ref{Subsection approximating families of chains}.

\begin{theorem}\label{Thm discrete delta approx chains}
    There exist constants $\varepsilon_{0}(k,p,M)>0$, $\delta_{0}(k,p,M)>0$, $C(k,p,M)$, $n(k,p)\in\mathbb{N}$ and $N(k,p)\in\mathbb{N}$ such that the following is true. Let $G:X_{0}\to\mathcal{I}_{k+1}(M)$ be an $\varepsilon$-fine family, $0<\varepsilon\leq\varepsilon_{0}$, $0\leq k\leq n-1$, $\dim(X)=p$ and denote $F=\partial G$. Let $0<\delta<\delta_{0}$. Then there exists a discrete family $G':X(q)_{0}\to\mathcal{I}_{k+1}(M)$ for $q=q(k,p,M,\delta)$ such that if we denote $F'=\partial G'$ then
    \begin{enumerate}
        \item $G'(x)=G(x)$ and $F'(x)=F(x)$ for every $x\in X_{0}$.
        \item Given a cell $C$ of $X$ and $x\in C$,
        \begin{equation*}
            \mass(G'(x))\leq\max_{v\in V(C)}\{\mass(G(v))\}+C(k,p,M)\frac{\varepsilon}{\delta^{n(k,p)}}
        \end{equation*}
        and
        \begin{equation*}
            \mass(F'(x))\leq\max_{v\in V(C)}\{\mass(F(v))\}+C(k,p,M)\frac{\varepsilon}{\delta^{n(k,p)}}.
        \end{equation*}
        \item $F'$ and $G'$ are $C(k,p,M)\frac{\varepsilon}{\delta^{n(k,p)}}$-fine in $X$ and $(N(k,p),\delta)$-localized in $X(q)$.
    \end{enumerate}
\end{theorem}

Next we state the versions of Theorem \ref{Thm discrete delta approx cycles} and Theorem \ref{Thm discrete delta approx chains} for continuous families. They are proved in subsection \ref{Subsection approximating continuous families}.

\begin{theorem}\label{Thm continuous delta approx cycles}
    Let $F:X\to\mathcal{Z}_{k}(M)$ be a continuous map without concentration of mass. Let $\varepsilon,\delta>0$. Then there exists $q,q'\in\mathbb{N}$ and a continuous map $F':X\to\mathcal{Z}_{k}(M)$ without concentration of mass such that
    \begin{enumerate}
        \item $F'$ is $\delta$-localized in $X(qq')$.
        \item $F'(x)=F(x)$ for every $x\in X(q)_{0}$.
        \item $\mathcal{F}(F(x),F'(x))\leq\varepsilon$ for every $x\in X$.
        \item If $x\in C$ for some cell $C$ of $X(q)$ then 
        \begin{equation*}
            \mass(F'(x))\leq\max\{\mass(F(v)):v\in V(C)\}+\varepsilon.
        \end{equation*}
    \end{enumerate}
\end{theorem}

\begin{theorem}\label{Thm continuous delta approx chains}
    Let $G:X\to\mathcal{I}_{k+1}(M)$ be a continuous map without concentration of mass. Denote $F=\partial G$ and let $\varepsilon,\delta>0$. Then there exists $q,q'\in\mathbb{N}$ and a continuous map $G':X\to\mathcal{I}_{k+1}(M)$ without concentration of mass such that if we denote $F'=\partial G'$ then
    \begin{enumerate}
        \item $G'$ and $F'$ are $\delta$-localized in $X(qq')$.
        \item $G'(x)=G(x)$ and $F'(x)=F(x)$ for every $x\in X(q)_{0}$.
        \item $\mathcal{F}(G(x),G'(x))\leq\varepsilon$ and $\mathcal{F}(F'(x),F(x))\leq\varepsilon$ for every $x\in X$.
        \item If $x\in C$ for some cell $C$ of $X(q)$ then 
        \begin{equation*}
            \mass(G'(x))\leq\max\{\mass(G(v)):v\in V(C)\}+\varepsilon.
        \end{equation*}
        and
        \begin{equation*}
            \mass(F'(x))\leq\max\{\mass(F(v)):v\in V(C)\}+\varepsilon.
        \end{equation*}
    \end{enumerate}
\end{theorem}

\subsection{Filling discrete families of small cycles}\label{Subsection Filling small families}

The main goal of this subsection is to prove the following proposition. We use the notation from Section \ref{Section Preliminaries}. It implies Theorem \ref{Thm filling small families} by the choice of $L$ detailed below.

\begin{proposition}\label{Prop filling small families}
    Given $0\leq k\leq n-1$ and $p,N\in\mathbb{N}$ there exist natural numbers $K(N,k,p)$, $n(k,p)$ and positive constants $c(k,p)$, $\varepsilon_{0}(k,p,M)$ and $\delta_{0}(k,p,M)$ so that the following is true. Let $\varepsilon\leq\varepsilon_{0}$, $\delta\leq\delta_{0}$ and let $F:X_{0}\to\mathcal{Z}_{k}(M)$ be a discrete $(N,\delta)$-localized family of cycles with $\mathcal{F}(F(x))\leq\varepsilon$ for every $x\in X_{0}$ and $\dim(X)=p$. Then there exists a $(K(N,k,p),K(N,k,p)\delta)$-localized family $\tau:X(q)_{0}\to\mathcal{I}_{k+1}(M)$ such that $\partial\tau(x)=R^{q}F(x)$ for $q=q(k,p,M,\delta)$ and $\mass(\tau(x))\leq c(k,p)(\frac{L}{\delta})^{n(k,p)}\varepsilon$.
\end{proposition}

In order to prove the proposition, we need to introduce some definitions and do some constructions first. Let $r<\inj(M,g)$ and consider a collection of $L$ points $x_{1},...,x_{L}\in M$ such that $\{B(x_{l},r):1\leq l\leq L\}$ covers $M$ and $L\leq c(n)\frac{\Vol(M,g)}{r^{n}}$ for some universal constant $c(n)$ depending only on $n$. For this section, we will consider $r=\delta$ and the corresponding $L$ is the one that appears in the statement of Proposition \ref{Prop filling small families}. In the next sections, we will choose other values of $r$, being that the reason why we do the following constructions for arbitrary $r<\inj(M,g)$. 

\begin{definition}
    Given a $(k+1)$-chain $\tau$ supported in $B(x_{l},2r)$, we consider a radius $r_{l}(\tau)\in[r,2r]$ such that
    \begin{equation*}
        \mass(\tau\llcorner \partial B(x_{l},r_{l}(\tau)))\leq\frac{\mass(\tau)}{r}.
    \end{equation*}
    The existence of such radius is guaranteed by the Coarea Inequality. Given $\tau\in\mathcal{I}_{k+1}(M)$ (which is not necessarily supported in $B(x_{l},2r)$) we define
    \begin{equation*}
        r_{l}(\tau)=r_{l}(\tau\llcorner B(x_{l},2r)).
    \end{equation*}
    Observe that under this definition, if $\support(\tau-\tau')\cap B(x_{l},2r)=\emptyset$ then $r_{l}(\tau)=r_{l}(\tau')$.
\end{definition}

\begin{definition}\label{Def chopping}
    Given $1\leq l\leq L$ and $\tau\in\mathcal{I}_{k+1}(M)$, we denote
    \begin{equation*}
        d_{l}(\tau)=\tau\llcorner(M\setminus\bigcup_{l'\leq l}B(x_{l'},r_{l'}(\tau)).
    \end{equation*}
    Observe that $d_{0}(\tau)=\tau$ and $d_{l}(\tau)=0$ for all $l\geq L$.
\end{definition}

We introduce the following definitions, which will be used in the rest of the section.

\begin{definition}
    Let $\phi:[0,1]\to[0,1]$ be the piecewise linear function
    \[
    \phi(x)=\begin{cases}
        0 & \text{ if }x\leq\frac{1}{3}\\
        3(x-\frac{1}{3})
        & \text{ if } \frac{1}{3}\leq x\leq\frac{2}{3}\\
        1 & \text{ if }\frac{2}{3}\leq x\leq 1.
    \end{cases}
    \]
    For each $j\in\mathbb{N}$, define $\Xi_{j}:[0,1]^{j}\to[0,1]^{j}$ as $\Xi_{j}(x_{1},...,x_{j})=(\phi(x_{1}),...,\phi(x_{j}))$. Given a cubical complex $X$, denote $\Xi:X\to X$ the map which restricts to $\Xi_{j}$ on each $j$-dimensional cell $C$ of $X$ (identifying $C\cong[0,1]^{j}$) and $\Xi^{3}:X\to X$ the map which restricts to $\Xi_{j}$ at each $j$-cell $C$ of $X(3)$.
\end{definition}

\begin{definition}\label{Def metrics on X}
    For each cell $j$-cell $C$ of $X$, identify $C\cong[0,1]^{j}$ and consider the metrics
    \begin{align*}
        d_{0}(x,y) & =\max\{|x_{i}-y_{i}|:1\leq i\leq j\},\\
        d_{1}(x,y) & =\sum_{i=1}^{j}|x_{i}-y_{i}|
    \end{align*}
    where $x=(x_{1},...,x_{j})$ and $y=(y_{1},...,y_{j})$. Given $q\in\mathbb{N}$, we consider the following metric $d^{q}$ in $X(q)_{0}$
    \begin{equation*}
        d^{q}(x,y)=qd_{1}(x,y).
    \end{equation*}
    $d^{q}(x,y)$ equals the minimal number of vertices in $X(q)_{0}$ that we must visit if we want to travel from $x$ to $y$ by jumping between vertices in $X(q)_{0}$ differing in only one coordinate. When $q$ is clear (because we are considering $x,y\in X(q)_{0}$ for some specific value of $q$) we may just denote $d^{q}$ by $d$.
\end{definition}

Later we will be interested in constructing discrete $\delta$-localized families inductively skeleton by skeleton on a cubical complex $X$. When doing that, we will have to extend a family $\tau(x)$ defined on the boundary of a certain cell $C$ of $X$ to the interior of $C$, preserving $\delta$-localization. The next lemma shows that this is possible if the extension is done by a chopping procedure as in Definition \ref{Def chopping}. The subsequent lemma shows that if in addition $\mass(\tau(x))$ and $\mass(\partial\tau(x))$ are small, then the extension and its boundary also have small mass.

\begin{lemma}\label{Prelemma chopping}
    Let $\tau,\tau'\in\mathcal{I}_{k+1}(M)$ be such that $\tau-\tau'$ is supported in the union of a certain collection of balls $\{B(y_{i},s_{i})\}_{1\leq i\leq N}$. Then
    \begin{equation*}
        \support(d_{l}(\tau)-d_{l}(\tau'))\subseteq\bigcup_{i\leq N}B(y_{i},s_{i}+4r)
    \end{equation*}
    for every $1\leq l\leq L$.
\end{lemma}

\begin{proof}
    Denote $r_{m}=r_{m}(\tau)$ and $r_{m}'=r_{m}(\tau')$. Then
    \begin{multline*}
        d_{l}(\tau)-d_{l}(\tau') =\tau-\tau'-[\tau\llcorner\bigcup_{m=1}^{l}B(x_{m},r_{m})-\tau'\llcorner\bigcup_{m=1}^{l}B(x_{m},r'_{m})]\\
          =\tau-\tau'-[(\tau-\tau')\llcorner\bigcup_{m=1}^{l}B(x_{m},r_{m}')+\tau\llcorner\bigcup_{m=1}^{l}B(x_{m},r_{m})\setminus B(x_{m},r'_{m})\\
         -\tau\llcorner\bigcup_{m=1}^{l}B(x_{m},r'_{m})\setminus B(x_{m},r_{m})].
    \end{multline*}
    The previous implies that
    \begin{equation*}
        \support(d_{l}(\tau)-d_{l}(\tau'))\subseteq\support(\tau-\tau')\cup\bigcup_{\substack{1\leq m\leq l \\ r_{m}\neq r_{m}'}}B(x_{m},2r).
    \end{equation*}
    But by a previous observation, $r_{m}\neq r'_{m}$ implies that $B(x_{m},2r)\cap\support(\tau-\tau')\neq\emptyset$. Therefore $B(x_{m},2r)\subseteq B(y_{i},s_{i}+4r)$ for some $1\leq i\leq N$ which yields the desired result.
\end{proof}

\begin{lemma}\label{Lemma chopping}
    Let $C=[0,1]^{j}$ and $q\in\mathbb{N}$, $q\geq L$. Let $\tau:\partial C\cap C(q)_{0}\to\mathcal{I}_{k+1}(M)$ be an $(N,\delta)$-localized family with $\mass(\tau(x))\leq\varepsilon_{1}$ and $\mass(\partial\tau(x))\leq\varepsilon_{2}$. Then the family $\overline{\tau}:C(3q)_{0}\to\mathcal{I}_{k+1}(M)$ given by
    \[
        \overline{\tau}(y)=\begin{cases}
            d_{l(y)}(\tau(\Xi(y)) & \text{ if } l(y)=d_{\infty}(y,\partial C)\leq q\\
            0 & \text{ otherwise}
        \end{cases}
    \]
    is $(N+1,3[\delta+(4N+2)r])$-localized, $\mass(\overline{\tau}(x))\leq\varepsilon_{1}$ and
    \begin{equation*}
        \mass(\partial\overline{\tau}(x))\leq \varepsilon_{2}+L\frac{\varepsilon_{1}}{r}.
    \end{equation*}
     
\end{lemma}

\begin{proof}[Proof of Lemma \ref{Lemma chopping}]
    First we prove $(N,\delta)$-localization. Let $y_{1},y_{2}$ be vertices of the same cell $E$ in $C(3q)$. If $y_{2},y_{3}\in\Center(C)$ the result follows, so we can assume $l(y_{1}),l(y_{2})\leq q$. Let $E'$ be a $(j-1)$-face of $\partial C(q)$ such that $\Xi(E)\subseteq E'$. Then without loss of generality, there exists another vertex $y_{3}$ of $E$ such that $\Xi(y_{2})=\Xi(y_{3})$, $l(y_{2})=l(y_{3})+1$, $l(y_{1})=l(y_{3})$ and $\Xi(y_{1}),\Xi(y_{3})$ belong to $E'$. By Lemma \ref{Prelemma chopping} applied to the chains $\tau=\tau(\Xi(y_{1}))$, $\tau'=\tau(\Xi(y_{3}))$, the number $l=l(y_{1})=l(y_{3})$ and the $(N,\delta)$-admissible family $\{U_{i}\}_{i\in I_{E'}}$ with $U_{i}=B(y_{i},s_{i})$ it holds
    \begin{equation*}
        \support(\overline{\tau}(y_{1})-\overline{\tau}(y_{3}))\subseteq\bigcup_{i\in I_{E'}}B(y_{i},s_{i}+4r).
    \end{equation*}
    On the other hand, by definition 
    \begin{equation*}
        \support(\overline{\tau}(y_{2})-\overline{\tau}(y_{3}))\subseteq B(x_{l(y_{2})},2r).
    \end{equation*}
    This implies that $\overline{\tau}(y_{1})-\overline{\tau}(y_{2})$ is supported in $B(x_{l(y_{2})},2r)\cup\bigcup_{i\in I_{E'}}B(y_{i},s_{i}+4r)$ and thus in an $(N+1,3[\delta+(4N+2)r])$-localized family by Lemma \ref{Lemma non disjoint}. Observe that if $l=l(x)\leq q$ and we denote $\tau=\tau(\Xi(x))$,
    \begin{align*}
        \partial\overline{\tau}(x) & =\partial d_{l}(\tau)\\
        & =\partial[\tau\llcorner M\setminus\bigcup_{l'\leq l}B(x_{l'},r_{l'}(\tau))]\\
        & =\partial\tau\llcorner [M\setminus\bigcup_{l'\leq l}B(x_{l'},r_{l'}(\tau))]+\tau\llcorner\partial [M\setminus\bigcup_{l'\leq l}B(x_{l'},r_{l'}(\tau))]
    \end{align*}
    hence
\begin{align*}
    \mass(\partial\overline{\tau}(x)) & \leq \mass(\partial\tau)+\sum_{l'=1}^{l}\mass(\tau\llcorner\partial B(x_{l'},r_{l'}(\tau)))\\
    & \leq\varepsilon_{2}+L\frac{\varepsilon_{1}}{r}.
\end{align*}

\end{proof}

\begin{proof}[Proof of Proposition \ref{Prop filling small families}]
    For each cell $C$ of $X$, let $\{U_{i}^{C}\}_{i\in I_{C}}$ be the $\delta$-admissible family associated to $C$ and let $U_{C}=\bigcup_{i\in I_{C}}U_{i}^{C}$. We proceed by induction on $n-k$, being $k=n-1$ the base case.

    \textbf{Base case 1.} Let $k=n-1$. For each $x\in X_{0}$, as $\mathcal{F}(F(x))\leq\varepsilon$ there exists an $n$-chain $\tau(x)$ such that $\partial\tau(x)=F(x)$ and $\mass(\tau(x))\leq\varepsilon$. Consider the corresponding family $\tau:X_{0}\to\mathcal{I}_{n}(M)$. We claim that if $C$ is a cell of $X$ and $x,y\in X_{0}\cap C$ then $\tau(x)-\tau(y)$ is supported in $U_{C}$. Indeed, as $\support(F(x)-F(y))\subseteq U_{C}$, there exists $\tau_{xy}\in\mathcal{I}_{n}(U_{C})$ such that $\partial\tau_{xy}=F(x)-F(y)$ and $\mass(\tau_{xy})\leq 2\varepsilon$. Therefore $\eta_{xy}=\tau(x)-\tau(y)-\tau_{xy}$ is an $n$-cycle on $M$ of mass at most $4\varepsilon$. Thus if we set $\varepsilon_{0}<\frac{1}{4}\Vol(M^{n},g)$, by the Constancy Theorem it must be $\eta_{xy}=0$ and therefore $\tau(x)-\tau(y)=\tau_{xy}$ thus it is supported in $U_{C}$ as desired. Therefore $q(n-1,p,M,\delta)=1$, $K(N,n-1,p)=N$, $n(n-1,p)=0$, $\delta_{0}=\inj(M,g)$ and $c(n-1,p)=1$ for every $p\in\mathbb{N}$ and every $N\in\mathbb{N}.$

    \textbf{Inductive step 1.} Assume that the result holds for families of $(k+1)$-cycles and let $F:X_{0}\to\mathcal{Z}_{k}(M)$ be as in the statement of the proposition. For each $x\in X_{0}$, we define $\tau(x)=\tau_{0}(x)$ to be an arbitrary $(k+1)$-chain such that $\partial\tau(x)=F(x)$ and $\mass(\tau(x))\leq\varepsilon$. Following the strategy of Guth and Liokumovich \cite{GL22}, we define a sequence of maps $\tau_{j}:X_{j}(q_{j})_{0}\to\mathcal{I}_{k+1}(M)$ ($q_{j}=q(k,j,p,M,\delta)$ odd for each $j$) which agree with $\tau_{0}$ at $X_{0}$, are $\delta$-localized, small in mass and verify $\partial\tau_{j}(x)=R^{q_{j}}F(x)$. In order to do this, we consider the dyadic subdivision $X(2)$ of $X$ and for each $C\in\faces_{p}(X)$ and $v\in V(C)$ we denote
    \begin{equation*}
        \overline{C}_{v}=\{x\in C:d_{1}(x,v)\leq\frac{1}{2}\}
    \end{equation*}
    being $\{\overline{C}_{v}:C\in\faces_{p}(X),v\in V(C)\}$ the collection of top dimensional cells of $X(2)$. Observe that by definition, if $x\in \overline{C}_{v}\cap X(q_{j})_{0}$ then $R^{q_{j}}F(x)=F(v)$ and hence $R^{q_{j}}F$ is constant on each $\overline{C}_{v}$. Therefore, in order to define $\tau_{j}$ inductively it will be easier to proceed one $\overline{C}_{v}$ at a time. This is because if $\tau_{j-1}$ was already defined in the boundary of a $j$-face $E$ of some $\overline{C}_{v}$ and $v_{E}$ is some vertex of $E$, then $\tilde{\tau}(x)=\tau_{j-1}(x)-\tau_{j-1}(v_{E})$ is a discrete family of $(k+1)$-cycles with domain $\partial E\cap X(q_{j-1})_{0}$ (observe that this is not true if $E$ is a cell of $C$ as the values of $\partial\tau_{j-1}(x)=R^{q_{j-1}}F(x)$ may vary). Therefore by the inductive hypothesis on the codimension of the cycles, there will exist $\sigma_{E}:\partial E\cap X(q_{j})_{0}\to\mathcal{I}_{k+2}(M)$ such that (up to refinement) $\partial\sigma_{E}(x)=\tilde{\tau}_{j}(x)$ and thus applying Lemma \ref{Lemma chopping} to $\sigma_{E}$ we can extend $\sigma_{E}$ and therefore $\tau_{j-1}$ to all of $E$ as $\tau_{j}(x)=\tau_{j-1}(v_{E})+\partial\sigma_{E}(x)$. But we have to be careful with these extensions in order to obtain $\delta$-localized maps, because it may happen that two vertices $x,y$ are in the same cell of $X(q_{j})_{0}$ but do not belong to the same $\overline{C}_{v}$ (hence we can not define the $\tau_{j}$ on each $\overline{C}_{v}$ independently, some relations must hold for adjacent $\overline{C}_{v}$'s). We need to make sure that in that case the difference $\tau_{j}(x)-\tau_{j}(y)$ is still supported in a $\delta$-admissible family. For that purpose, we will consider a certain equivalence relation on the faces of the $\overline{C}_{v}$'s and extend $\tau_{j-1}$ one class at a time so that we maintain a certain relation between $\tau_{j}(x)$ and $\tau_{j}(y)$ if $x$ and $y$ are adjacent  (namely, that the difference $\tau_{j}(x)-\tau_{j}(y)$ is supported in a $\delta$-admissible family).

    With the previous in mind, it will be easier for us to work with a complex $\tilde{X}$ whose cells are essentially the $\overline{C}_{v}$'s and to define maps there which verify certain compatibility relations between different cells $\overline{C}_{v}$ instead of working with our initial cubical complex $X$. We make this idea precise in the definitions below.

    \begin{definition}
        Given $C\in\faces_{p}(X)$, $v\in V(C)$ and $\lambda\in[0,1]$, let $H_{\lambda}=H_{\lambda}^{C,v}:\overline{C}_{v}\to \overline{C}_{v}$ be the linear homotecy on $\overline{C}_{v}$ centered at $v$ of ratio $\lambda$. If $\lambda>1$ we set $H_{\lambda}=H_{\lambda^{-1}}^{-1}: H_{\lambda^{-1}}(\overline{C}_{v})\to\overline{C}_{v}$. Define $C_{v}=H_{\frac{2}{3}}(\overline{C}_{v})$, notice that
        \begin{align*}
            \overline{C}_{v} & =\{x\in C:d_{1}(x,v)\leq\frac{1}{2}\}\\
            C_{v} & =\{x\in C:d_{1}(x,v)\leq\frac{1}{3}\}
        \end{align*}
        and hence $C_{v}\in\faces_{p}(X(3))$.
    \end{definition}

    \begin{definition}
        We denote
        \begin{equation*}
            \tilde{X}=\bigsqcup_{C\in\faces_{p}(X)}\bigsqcup_{v\in V(C)}C_{v}.
        \end{equation*}
        being $\tilde{X}\subseteq X(3)$ a subcomplex. Let $\iota:\tilde{X}\to X(2)$ the map which restricts to $H_{\frac{3}{2}}^{C,v}$ on each $C_{v}$.
    \end{definition}

    \begin{definition}
        For each $q\in\mathbb{N}$ we define $\iota_{q}:\tilde{X}(q)\to X(2q+1)$ the map sending each $x\in C_{v}$ to $H^{C,v}_{\frac{3q}{2q+1}}(x)$.
    \end{definition}

    \begin{remark}
        For every $q$, $\iota_{q}:\tilde{X}(q)\to X(2q+1)$ is a simplicial map. Moreover, its restriction to the $0$-skeleton $\iota_{q}:\tilde{X}(q)_{0}\to X(2q+1)_{0}$ is bijective and is the identity when $q=1$. Therefore, there is a bijective correspondence between discrete maps $\tilde{F}:\tilde{X}(q)_{0}\to \mathcal{I}_{k}(M)$ and discrete maps $F:X(2q+1)_{0}\to \mathcal{I}_{k}(M)$ given by by $F\circ\iota_{q}=\tilde{F}$. We will often refer to $F$ and $\tilde{F}$ using the same symbol considering this identification.
    \end{remark}

    \begin{definition}
        Let $\sim$ be the equivalence relation on the cells of $\tilde{X}$ given by $E\sim E'$ if and only if $\iota(E)=\iota(E')$. If $E\sim E'$, we define an affine map $T_{EE'}:E\to E'$ given by $T_{EE'}=\iota|_{E'}^{-1}\circ\iota|_{E}$. Given $F,F'\in\faces_{j}(\tilde{X}(q_{j}))$ we say that $F\sim F'$ if there exist $E,E'\in\faces_{j}(X)$ such that $F\subseteq E$, $F'\subseteq E'$, $E\sim E'$ and $T_{EE'}(F)=F'$. In that case, we denote $T_{FF'}=T_{EE'}$.
    \end{definition}

    \begin{remark}
        Let $x,y\in \tilde{X}(q)_{0}$. Observe that if $\iota_{q}(x)$ is adjacent to $\iota_{q}(y)$ in $X(2q+1)_{0}$ (meaning that $d^{2q+1}(\iota_{q}(x),\iota_{q}(y))=1$ as in Definition \ref{Def metrics on X}), then either $x$ is adjacent to $y$ in $\tilde{X}(q)_{0}$ or there exists $F,F'\in\faces_{j}(\tilde{X}_{j}(q_{j}))$ such that $x\in F$, $y\in F'$ and $F\sim F'$.
    \end{remark}

    Taking the previous observation in mind, we are going to inductively define $\delta_{j}$-localized maps $\tau_{j}:\tilde{X}_{j}(q_{j})_{0}\to\mathcal{I}_{k+1}(M)$ satisfying the following: if $F\sim F'$, $F,F'\in\faces_{j}(\tilde{X}(q_{j}))$ and $x\in F$, the chain $\Delta^{j}_{FF'}(x):=\tau_{j}(x)-\tau_{j}(T_{FF'}(x))$ is supported in a certain $\delta_{j}$-admissible family $\{B_{i}^{F}\}_{i\in I_{F}}$. This will allow us to induce a map $\tau$ on $X(2q_{p}+1)_{0}$ from $\tau_{p}$ which will be $\delta$-localized and will have the other desired properties.

    We start by defining $\tau_{0}:\tilde{X}_{0}\to\mathcal{I}_{k+1}(M)$. As we observed previously, $\tilde{X}_{0}=X(3)_{0}$ so we will regard $X(3)_{0}$ as the domain of $\tau_{0}$. As $\tau_{0}$ was already defined at $X_{0}$, it is enough to define it in $C(3)_{0}\setminus V(C)$ for each $C\in\faces_{p}(X)$. Observe that
    \begin{equation*}
        C(3)_{0}\setminus V(C)=\bigsqcup_{E\text{ face of }C}V(\Center(E))
    \end{equation*}
    where $V(\Center(E))=\Center(E)\cap C(3)_{0}$ is the set of vertices of $\Center(E)$. Hence it suffices to define $\tau_{0}$ on $V(\Center(E))$ for each face $E$ of $C$. Given a face $E$ of $C$, we choose $v_{E}\in V(E)$ (arbitrary) and we define $\tau_{0}(w)=\tau_{0}(v_{E})+\tau(\Xi(w),v_{E})$ where given $v,v'\in V(E)$ we select $\tau(v,v')\in\mathcal{I}_{k+1}(M)$ a chain such that $\partial\tau(v,v')=F(v)-F(v')$, $\support(\tau(v,v'))\subseteq U_{E}$ and $\mass(\tau(v,v'))\leq 2\varepsilon$. Notice that $\mass(\tau_{0}(w))\leq 3\varepsilon$ for every $w\in C(3)_{0}$ and that in fact this definition is independent of the cell $C$ (we get the same $\tau_{0}$ if we regard $E$ as a face of another top-dimensional cell $C'$) hence it extends $\tau_{0}$ to $X(3)_{0}$. But this refinement has some extra advantages. Given $x,y\in X(3)_{0}$ with $d(x,y)=1$ (here $d=d^{3}$ as in Definition \ref{Def metrics on X}), either there exists $v$ such that $x,y\in C_{v}$ and hence $\partial\tau_{0}(x)=\partial\tau_{0}(y)=F(v)$ or $\tau_{0}(x)-\tau_{0}(y)$ is supported in $U_{E}$ because $x,y\in\Center(E)$ for some face $E$ of $C$. Thus to obtain $\delta$-localization we will perform refinements to the domain $\tilde{X}$ to replace ``jumps'' as in the first case from a chain $\tau_{0}(x)$ to another chain $\tau_{0}(y)$ with (possibly) very different support but the same boundary by ``smaller jumps'' involving consecutive chains with similar supports, taking advantage of the fact that $\tau(x)-\tau(y)$ is a $(k+1)$-cycle which we can chop using Lemma \ref{Lemma chopping}. We proceed by induction in the skeleta of $\tilde{X}$.
    
    \textbf{Inductive property 2.} Given $1\leq j\leq p$, denote $q_{j}=q(k,j,p,M,\delta)$. There exists  a discrete family $\tau_{j}:\tilde{X}(q_{j})_{0}\to\mathcal{I}_{k+1}(M)$ and for each $j$-cell $F$ of $\tilde{X}_{j}(q_{j})$ a $(K(N,k,j),K(N,k,j)\delta)$-admissible family $\{B_{i}^{F}\}_{i\in I_{F}}$ such that
    \begin{enumerate}[label=(A\arabic*)]
        \item $\partial\tau_{j}(x)=R^{q_{j}}F(x)$.
        \item $\mass(\tau_{j}(x))\leq c(k,j)(\frac{L}{\delta})^{n(k,j)}\varepsilon$.
        \item If $F\in\faces_{j}(\tilde{X}_{j}(q_{j}))$ and $x,y\in V(F)$ then $\support(\tau(x)-\tau(y))$ is covered by $\{B_{i}^{F}\}_{i\in I_{F}}$.
        \item If $F\sim F'$ then $\{B_{i}\}_{i\in I_{F}}=\{B_{i}\}_{i\in I_{F'}}$. Furthermore, let $\Delta^{j}_{FF'}(x)=\tau_{j}(x)-\tau_{j}(T_{FF'}(x))$ for $x\in F\cap\tilde{X}(q_{j})_{0}$. Then $\Delta^{j}_{FF'}(x)$ is supported in $\bigcup_{i\in I_{F}}B_{i}^{F}$.
    \end{enumerate}

\textbf{Base case 2.} Let $q_{1}=q_{1}(k,1,p,M,\delta)$ be the smallest odd number which is greater or equal than $L$. Define $\tau_{1}(x)=\tau_{0}(x)$ for $x\in\tilde{X}_{0}$. Our goal is to extend $\tau_{1}$ to $\tilde{X}_{1}(q_{1})_{0}$.
To do that, we pick for each class $\mathcal{E}$ of $1$-cells under $\sim$ a representative $E\in\mathcal{E}$. We first define $\tau_{1}$ on $E$ and then we extend it to the other edges $E'\in\mathcal{E}$. Let $C\in\faces_{p}(X)$ and $v\in V(C)$ be such that $E\subseteq C_{v}$. Let $v_{E},w_{E}$ be the vertices of $E$. Observe that $\tau_{1}(w_{E})-\tau_{1}(v_{E})$ is a $(k+1)$-cycle because $\partial\tau_{1}(v_{E})=F(v)=\partial\tau_{1}(w_{E})$. Let $\sigma_{E}\in\mathcal{I}_{k+2}(M)$ be such that $\partial\sigma_{E}=\tau_{1}(w_{E})-\tau_{1}(v_{E})$ and $\mass(\sigma_{E})\leq (6\varepsilon)^{\frac{k+2}{k+1}}\leq 6\varepsilon$ (such $\sigma_{E}$ exists by the Isoperimetric Inequality). Using the Coarea Inequality, we can define for each $1\leq l\leq L$ a radius $r_{l}^{E}\in[\delta,2\delta]$ such that
\begin{equation*}
    \mass(\sigma_{E}\llcorner\partial B(x_{l},r_{l}^{E}))\leq\frac{\mass(\sigma_{E})}{\delta}\leq 6\frac{\varepsilon}{\delta}.
\end{equation*}
For each $1\leq l\leq L$, denote $M_{l}^{E}=\bigcup_{1\leq l'\leq l} B(x_{l'},r_{l'}^{E})$ and $M_{0}^{E}=\emptyset$. Define
\begin{equation*}
    \tau_{1}(x)=\tau_{1}(v_{E})+\partial[\sigma_{E}\llcorner M_{l(x)}^{E}]
\end{equation*}
where $l(x)=d(x,v_{E})$ (here $d$ denotes the metric $d^{q_{1}}$ on $\tilde{X}$ as in Definition \ref{Def metrics on X}). Then by construction, if $x,y\in E\cap \tilde{X}(q_{1})_{0}$ verify $d(x,y)=1$ then $\tau_{1}(x)-\tau_{1}(y)$ is supported in a certain ball $B(x_{l},2\delta)$ and
\begin{equation*}
    \mass(\tau_{1}(x))\leq 3\varepsilon +6L\frac{\varepsilon}{\delta}\leq 9 L\frac{\varepsilon}{\delta}.
\end{equation*}

Now assume $|\mathcal{E}|>1$ and let $E'$ be another element of the equivalence class $\mathcal{E}$ and let $v'$ be the vertex of $C$ such that $E'\subseteq C_{v'}$. Define $\Delta_{EE'}^{1}(x)=\tau_{1}(x)-\tau_{1}(T_{EE'}(x))$ for $x=v_{E},w_{E}$. For $x\in E\cap\tilde{X}(q_{1})_{0}\setminus\{v_{E},w_{E}\}$ set
\begin{equation*}
    \Delta_{EE'}^{1}(x)=\tau_{1}(v_{E})-\tau_{1}(T_{EE'}(v_{E})).
\end{equation*}
As we are assuming that $|\mathcal{E}|>1$, there exist faces $E_{1}$ and $E_{2}$ of $C$ such that $v_{E}\in V(\Center(E_{1}))$ and $w_{E}\in V(\Center(E_{2}))$. This implies that $T_{EE'}(v_{E})\in V(\Center(E_{1}))$ and $T_{EE'}(w_{E})\in V(\Center(E_{2}))$ for every $E'\in\mathcal{E}$. Therefore, for each $x\in E\cap \tilde{X}(q_{1})_{0}$ the chain $\Delta^{1}_{EE'}(x)$ is supported in $\{U_{i}^{E_{1}}\}_{i\in I_{E_{1}}}$ or in $\{U_{i}^{E_{2}}\}_{i\in I_{E_{2}}}$ by construction of $\tau_{0}:\tilde{X}_{0}\to\mathcal{I}_{k+1}(M)$. In addition, $\partial\Delta^{1}_{EE'}(x)=F(v)-F(v')=R^{q_{1}}F(x)-R^{q_{1}}F(T_{EE'}(x))$ and it has mass at most $2\varepsilon$  for every $x\in E\cap\tilde{X}(q_{1})_{0}$. Define for $x\in E'\cap\tilde{X}(q_{1})_{0}$
\begin{equation*}
    \tau_{1}(x)=\tau_{1}(T_{E'E}(x))-\Delta_{EE'}(x)
\end{equation*}
so that we have $\partial\tau_{1}(x)=R^{q_{1}}F(x)$. For each $1$-cell $F$ of $\tilde{X}_{1}(q_{1})$ contained in $E$, let $l=d_{\infty}(F,\{v_{E}\})$ and let $\{B_{i}^{F}\}_{i\in I_{F}}$ be an $(2N+1,\delta_{1})$-admissible family which covers $B(x_{l+1},2\delta)\cup\bigcup_{i\in I_{E_{1}}}U_{i}^{E_{1}}\cup\bigcup_{i\in I_{E_{1}}}U_{i}^{E_{2}}$, which exists by Lemma \ref{Lemma non disjoint} for $\delta_{1}=12\delta$. For $F'=T_{EE'}(F)$, define $\{B_{i}^{F'}\}_{i\in I_{F'}}=\{B_{i}^{F}\}_{i\in I_{F}}$. Then we have $\mass(\tau_{1}(x))\leq 15L\frac{\varepsilon}{\delta}$ and if $x,y\in F'$ then
\begin{equation*}
    \tau_{1}(x)-\tau_{1}(y)=\tau_{1}(T_{E'E}(x))-\tau_{1}(T_{E'E}(y))+\Delta_{EE'}(y)-\Delta_{EE'}(x)
\end{equation*}
is supported in $\bigcup_{i\in I_{F'}}B_{i}^{F'}$. In case $|\mathcal{E}|=1$, we just define the admissible family to be $\{B(x_{l+1},2\delta)\}$. Thus properties (A1) to (A4) hold with $c(k,1)=15$, $n(k,1)=1$, and $K(N,k,1)=12$.

\textbf{Inductive step 2.} Assume $\tau_{j-1}$ has been defined in  $\tilde{X}_{j-1}(q_{j-1})_{0}$ ($q_{j-1}=q(k,j-1,p,M,\delta)$) verifying properties (A1) to (A4). Define 
\begin{align*}
    d_{j} & =q(k+1,j-1,M,K(N,k,j-1)\delta),\\
    \overline{q}_{j}& =q_{j-1}d_{j},\\
    q_{j} & =3\overline{q}_{j}
\end{align*}
and $\overline{\tau}_{j}(x)=R^{d_{j}}\tau_{j-1}(x)$ for $x\in X_{j-1}(q_{j})_{0}$. Let $\mathcal{E}$ be a class of $j$-dimensional faces under $\sim$ and $E\in\mathcal{E}$. We know that the family $\tau_{j-1}:\partial E\cap \tilde{X}(q_{j-1})_{0}\to\mathcal{I}_{k+1}(M)$ verifies \textbf{Inductive property 2}. Then consider the family of $(k+1)$-cycles $\tilde{\tau}_{j-1}(x)=\tau_{j-1}(x)-\tau_{j-1}(v_{E})$ ($v_{E}$ preferred vertex of $E$) with domain $\partial E\cap \tilde{X}(q_{j-1})_{0}$. As it is $(K(N,k,j-1),K(N,k,j-1)\delta)$-localized and has small mass, by \textbf{Inductive Hypothesis 1} the refined family $\tilde{\tau}_{j}=R^{d_{j}}\tilde{\tau}_{j-1}:\partial E\cap\tilde{X}(\overline{q}_{j})_{0}\to\mathcal{Z}_{k+1}(M)$ admits an $(K_{1}(N,k,j),K_{1}(N,k,j)\delta)$-localized filling $\tilde{\sigma}_{E}:\partial E\cap C_{v}(\overline{q}_{j})_{0}\to\mathcal{I}_{k+2}(M)$ with
\begin{equation*}
    \mass(\tilde{\sigma}_{E}(x))\leq c(k+1,j-1)c(k,j-1)(\frac{L}{\delta})^{n(k,j-1)+n(k+1,j-1)}\varepsilon
\end{equation*}
and $K_{1}(N,k,j)=K(K(N,k,j-1),k+1,j-1)K(N,k,j-1)$. We can extend $\tilde{\sigma}_{E}$ to $E\cap \tilde{X}(q_{j})_{0}$ (using that $q_{j}=3\overline{q}_{j}\geq 3q_{1}\geq 3L$) as
\[
    \sigma_{E}(x)=\begin{cases}
        0 & \text{ if } x\in\Center(E)\\
        d_{l(x)}(\tilde{\sigma}_{E}(\Xi(x)) & \text{ otherwise }
    \end{cases}
\]
being $l(x)=d_{\infty}(x,\partial E)$. We define
\begin{equation*}
    \tau_{j}(x)=\tau_{j-1}(v_{E})+\partial\sigma_{E}(x).
\end{equation*}
Then it follows that $\partial\tau_{j}(x)=R^{q_{j}}F(x)$,
\begin{equation*}
    \mass(\tau_{j}(x))\leq 3c(k+1,j-1)c(k,j-1)(\frac{L}{\delta})^{n(k,j-1)+n(k+1,j-1)+1}\varepsilon
\end{equation*}
and for each $j$-cell $F\subseteq E$ of $X_{j}(q_{j})$ the family $\tau_{j}:V(F)\to\mathcal{Z}_{k}(M)$ is localized in a $(K_{2}(N,k,j),K_{2}(N,k,j)\delta)$ admissible family by Lemma \ref{Lemma chopping} with $r=\delta$ which yields $K_{2}(N,k,j)=12(K_{1}(N,k,j)+1)$. If the equivalence class $\mathcal{E}$ of $E$ has more than one element, there exists a unique cell $C_{E}$ of $X$ such that $v_{E}\in\Center(C_{E})$ and it holds that $T_{EE'}(v_{E})\in\Center(C_{E})$ for every $E'\in\mathcal{E}$. Let us enlarge the previously constructed admissible family to a $(K_{2}(N,k,j)+N,3(K_{2}(N,k,j)+1)\delta)$-localized family which also covers $\{U_{i}^{C_{E}}:i\in I_{C_{E}}\}$ (using Lemma \ref{Lemma non disjoint}). This will be the family $\{B_{i}\}_{i\in I_{F}}$ and thus we set $K(N,k,j)=3(K_{2}(N,k,j)+1)$. Notice that $\tau_{0}(v_{E})-\tau_{0}(T_{EE'}(v_{E}))$ is supported in $\bigcup_{i\in I_{C_{E}}}U_{i}^{C_{E}}$ and therefore in $\{B_{i}\}_{i\in I_{F}}$ for every $E'\sim E$.

Define $\tau_{j}(x)=R^{d_{j}}\tau_{j-1}(\Xi(x))$ for each $x\in X_{j-1}\cap X(q_{j})_{0}$, which extends the previous definition of $\tau_{j}$ along $\partial E$. For $x\in X_{j-1}\cap X(q_{j})_{0}\cap E$ set $\Delta_{EE'}^{j}(x)=\tau_{j}(x)-\tau_{j}(T_{EE'}(x))$. Let $E'\in \mathcal{E}$, $E'\neq E$. We extend $\Delta^{j}_{EE'}(x)$ to $\interior(E)\cap\tilde{X}(q_{j})_{0}$ by
\begin{equation*}
    \Delta^{j}_{EE'}(x)=\Delta^{j}_{EE'}(v_{E})=\tau_{0}(v_{E})-\tau_{0}(T_{EE'}(v_{E})).
\end{equation*}
Observe that if $E\subseteq C_{v}$ and $E'\subseteq C_{v'}$ then $\partial\Delta^{j}_{EE'}(x)=F(v)-F(v')$ and $\Delta^{j}_{EE'}(x)$ is supported in $\{B_{i}\}_{i\in I_{F}}$ for every $x\in F\subseteq E$, $F\in\faces_{j}(\tilde{X}_{j}(q_{j}))$. We set
\begin{equation*}
    \tau_{j}(x)=\tau_{j}(T_{E'E}(x))-\Delta^{j}_{EE'}(T_{E'E}(x))
\end{equation*}
where $\Delta^{j}_{E'E}(x)=-\Delta^{j}_{EE'}(T_{E'E}(x))$. Then it follows that
\begin{equation*}
    \partial\tau_{j}(x)=R^{q_{j}}F(T_{E'E}(x))-F(v)+F(v')=F(v')=R^{q_{j}}F(x)
\end{equation*}
for all $x\in E'\cap \tilde{X}(q_{j})_{0}$ and that
\begin{equation*}
    \mass(\tau_{j}(x))\leq\mass(\tau_{j}(T_{E'E}(x)))+\mass(\Delta^{j}_{EE'}(T_{E'E}(x)))\leq c(k,j)(\frac{L}{\delta})^{n(k,j)}\varepsilon
\end{equation*}
if we set $c(k,j)=9c(k+1,j-1)c(k,j-1)$ and $n(k,j)=n(k+1,j-1)+n(k,j-1)+1$. Additionally, if $x,y\in F'\cap \tilde{X}(q_{j})_{0}$ then
\begin{equation*}
    \tau_{j}(x)-\tau_{j}(y)=\tau_{j}(T_{E'E}(x))-\tau_{j}(T_{E'E}(y))-\Delta^{j}_{EE'}(T_{E'E}(x))+\Delta^{j}_{EE'}(T_{E'E}(y))
\end{equation*}
which by construction of $\tau_{j}$ along $F$ is supported in $\bigcup_{i\in I_{F}}B_{i}^{F}$. In case $\mathcal{E}$ has only one element, the $(K_{2}(N,k,j),K_{2}(N,k,j)\delta)$-admissible families constructed first already have the desired properties.
\end{proof}

\subsection{Approximating discrete families of cycles}\label{Subsection approximating families of cycles}

The main goal of this subsection is to prove the following proposition. We use the notation from sections \ref{Section Almgren-Pitts} and \ref{Subsection Filling small families}. It implies Theorem \ref{Thm discrete delta approx cycles} by the choice of $L$ detailed below.

\begin{proposition}\label{Prop delta loc aprox cycles}
    There exist constants $\varepsilon_{0}(k,p,M)>0$, $\delta_{0}(k,p,M)>0$, $q(k,p,M,\delta)\in\mathbb{N}$ and $N(k,p)\in\mathbb{N}$ such that the following is true. Let $F:X_{0}\to\mathcal{Z}_{k}(M)$ be an $\varepsilon$-fine family, $\varepsilon\leq\varepsilon_{0}$, $0\leq k\leq n-1$, $\dim(X)=p$. Let $0<\delta<\delta_{0}$. Then there exists a discrete family $F':X(q)_{0}\to\mathcal{Z}_{k}(M)$ such that
    \begin{enumerate}
        \item $F'(x)=F(x)$ for every $x\in X_{0}$.
        \item Given a cell $C$ of $X$ and $x\in C$,
        \begin{equation*}
            \mass(F'(x))\leq\max_{v\in V(C)}\{\mass(F(v))\}+c(k,p)(\frac{L}{\delta})^{n(k,p)}\varepsilon.
        \end{equation*}
        \item $F'$ is $c(k,p)(\frac{L}{\delta})^{n(k,p)}\varepsilon$-fine in $X$ and $(N(k,p),\delta)$-localized in $X(q)$.
    \end{enumerate}
\end{proposition}

We start by fixing for each pair of adjacent vertices $v,w\in X_{0}$ a $(k+1)$-chain $\tau_{vw}=\tau(v,w)$ such that $\mass(\tau_{vw})\leq\varepsilon$ and $\partial\tau_{vw}=F(v)-F(w)$. We let $r$ be a sufficiently small number compared to $\delta$ so that the following constructions yield a $\delta$-localized family (it will be enough to choose $r\leq\frac{\delta}{N(k,p)}$ for the constant $N(k,p)$ which is inductively constructed below). We fix points $p_{1},...,p_{L}$ in the manifold $M$ such that $\{B(p_{l},r):1\leq l\leq L\}$ covers $M$. They can be chosen so that $Lr^{n}\leq c(n)\Vol(M,g)$ for some universal constant $c(n)$ which depends only on the dimension of the manifold $M$. For each cell $C$ of $X$ we denote $V(C)=C\cap X_{0}$ the set of its vertices and we choose a preferred vertex $v_{C}\in V(C)$.

Following the strategy in \cite{GL22}, we will inductively define for each $1\leq j\leq p=\dim(X)$ a $\delta_{j}$-localized and $\varepsilon_{j}$-fine map $F_{j}:X_{j}(q_{j})_{0}\to\mathcal{Z}_{k}(M)$, where $q_{1}<...<q_{p}=q$. All of the previous maps will agree with $F$ at $X_{0}$. Given a $j$-dimensional cell $C$ of $X(q)$, the map $F_{j}|_{C\cap X(q_{j})_{0}}$ will be given by an interpolation of the values of $F(v)$ over the vertices $v$ of $C$ which does not increase the mass too much with respect to $\max\{\mass(F(v)):v\in C\cap X_{0}\}$. Let us explain how this interpolation works for $j=1$. 

Let $C\in\faces_{1}(X)$ and let $v,w$ be the vertices of $C$. We are going to choose a certain grid $\mathcal{D}_{C}=\{D_{1},...,D_{L}\}$ on $M$ consisting of $L$ domains with disjoint interiors and piecewise smooth boundaries. Roughly speaking, the idea to interpolate between $F(v)$ and $F(w)$ is to start from $F(v)$ and replace $F(v)\llcorner D_{l}$ by the chain with smallest mass among $F(w)\llcorner D_{l}$ and $F(v)\llcorner D_{l}$ for each $l=1,...,L$. This operation is done one domain at a time. Thus we start from $z_{0}=F(v)$ and inductively define $z_{l}=z_{l-1}$ if $\mass(F(v)\llcorner D_{l})\leq\mass(F(w)\llcorner D_{l})$ and
\begin{equation*}
    z_{l}=z_{l-1}-\partial[\tau_{vw}\llcorner D_{l}]
\end{equation*}
otherwise. As
\begin{equation*}
    \partial[\tau_{vw}\llcorner D_{l}]=[F(v)-F(w)]\llcorner D_{l}+\tau_{vw}\llcorner\partial D_{l}
\end{equation*}
in order to avoid increasing the mass too much we need to choose the grid so that $\mass(\tau_{vw}\llcorner\partial D)$ is small for all $D\in\mathcal{D}_{C}$. This can be done using Coarea Inequality in the following way. For each $1\leq l\leq L$, we can pick a radius $r_{l}^{C}\in [r,2r]$ such that
\begin{equation*}
    \mass(\tau_{vw}\llcorner\partial B(x_{l},r_{l}^{C}))\leq\frac{\mass(\tau_{vw})}{r}\leq\frac{\varepsilon}{r}.
\end{equation*}
Notice that the collection of balls $\{B(x_{l},r_{l}^{C}):1\leq l\leq L\}$ covers $M$, and therefore the domains
\begin{equation*}
    D_{l}=\overline{B}(x_{l},r_{l}^{C})\setminus\bigcup_{l'<l}B(x_{l'},r_{l'}^{C})
\end{equation*}
form a grid in $M$ which verifies
\begin{equation*}
    \sum_{l=1}^{L}\mass(\tau_{vw}\llcorner \partial D_{l})\leq 2L\frac{\varepsilon}{r}.
\end{equation*}
The previous implies that $\mass(z_{l})\leq\mass(F(v))+2L\frac{\varepsilon}{r}$ for every $1\leq l\leq L$. In a similar way, we can start from $\tilde{z}_{0}=F(w)$ and recursively define $\tilde{z}_{1},...,\tilde{z}_{L}\in\mathcal{Z}_{k}(M)$ as $\tilde{z}_{l}=\tilde{z}_{l-1}$ if $\mass(F(w)\llcorner D_{l})\leq\mass(F(v)\llcorner D_{l})$ and
\begin{equation*}
    \tilde{z}_{l}=\tilde{z}_{l-1}+\partial[\tau_{vw}\llcorner D_{l}^{C}]
\end{equation*}
otherwise. It can be shown that $\tilde{z}_{L}=z_{L}$ and hence this allows to define $F_{1}$ along $C(2l)_{0}$ by $z_{0},z_{1},...,z_{L},\tilde{z}_{L-1},...,\tilde{z}_{1},\tilde{z}_{0}$. This has the desired $\delta$-localization property as $z_{l}-z_{l-1}$ is supported in $D_{l-1}\subseteq B(x_{l-1},2r)$ for every $l$ (respectively $\tilde{z}_{l}-\tilde{z}_{l-1}$). We can formalize the previous in the following way. For each $x\in C\cap X(3L)_{0}$  and each $D\in\mathcal{D}_{C}$, we have a certain vertex $w^{D}_{C}(x)\in\{v,w\}$ such that
\begin{equation}\label{Expression 1 for F1}
    F_{1}(x)=F(v)-\sum_{D\in\mathcal{D}_{C}}\partial[\tau(v,w_{C}^{D}(x))\llcorner D]
\end{equation}
where the $w_{C}^{D}(x)$ are chosen as described above to avoid increasing the mass too much (an explicit formula is given later). But we can also write
\begin{equation}\label{Expression 2 for F1}
   F_{1}(x)=\sum_{D\in\mathcal{D}_{C}}F(w^{D}_{C}(x))\llcorner D+I_{1}(x) 
\end{equation}
where
\begin{equation*}
    I_{1}(x)=\sum_{D\in\mathcal{D}_{C}}\tau(v,w_{C}^{D}(x))\llcorner\partial D
\end{equation*}
is a $k$-chain of mass bounded by $2 L\frac{\varepsilon}{r}$. Both expressions could be generalized for higher-dimensional cells. Given $C\in\faces_{j}(X)$, $j\geq 1$, we will need to choose a certain grid $\mathcal{D}_{C}$ as before with the property that $\tau_{vw}\llcorner\partial D$ has controlled mass for each $v,w\in V(C)$. And then we will have to define for each $x\in C\cap X(q_{j})_{0}$ and $D\in\mathcal{D}_{C}$ a certain vertex $w^{D}_{C}(x)\in V(C)$ so that $F_{j}$ is given by
\begin{equation}\label{Expression 1 for Fj}
    F_{j}(x)=F(v_{C})+\sum_{D\in\mathcal{D}_{C}}\partial[\tau(v_{C},w_{C}^{D}(x))\llcorner D]
\end{equation}
for some $v_{C}\in V(C)$ if we want to generalize (\ref{Expression 1 for F1}), or
\begin{equation}\label{Expression 2 for Fj}
    F_{j}(x)=\sum_{D\in\mathcal{D}_{C}}F(w_{C}^{D}(x))\llcorner D+I_{j}(x)
\end{equation}
for some family $I_{j}:X_{j}(q_{j})_{0}\to\mathcal{I}_{k}(M)$ with small mass in case we try to generalize (\ref{Expression 2 for F1}). It will be convenient to work with (\ref{Expression 2 for Fj}) because when proceeding inductively from the $(j-1)$-skeleton to the $j$-skeleton, the preferred vertex $v_{E}$ of a $(j-1)$-face $E$ of $C$ will not always coincide with $v_{C}$ (for example when $v_{C}\notin E)$ and therefore we would need to do a certain adjustment to the $\tau(v_{E},w_{j}^{D}(x))$ in order to have (\ref{Expression 1 for Fj}) which is not necessary if we adopt (\ref{Expression 2 for Fj}). This problem already arises when $j=2$, as for $x\in E\subseteq \partial C$ with $v_{E}\neq v_{C}$ we would have
\begin{align*}
    F_{1}(x) & =F(v_{C})+\partial\tau(v_{C},v_{E})+\sum_{D\in\mathcal{D}_{E}}\partial[\tau(v_{E},w^{D}_{E}(x))\llcorner D]\\
    & =F(v_{C})+\sum_{D\in\mathcal{D}_{E}}\partial[(\tau(v_{C},v_{E})+\tau(v_{E},w_{E}^{D}(x)))\llcorner D]
\end{align*}
and it might be the case that $\tau(v_{C},v_{E})+\tau(v_{E},w_{E}^{D}(x))\neq\tau(v_{C},w^{E}_{D}(x))$ (if the $\tau_{vw}$ could be chosen so that they verified $\tau(v_{1},v_{2})+\tau(v_{2},v_{3})=\tau(v_{1},v_{3})$ for every triplet of adjacent vertices $v_{1},v_{2},v_{3}\in X_{0}$, then we could construct a filling $G:X_{0}\to\mathcal{I}_{k+1}(M)$ of $F$ such that $G(v)-G(w)=\tau(v,w)$ and hence $\mass(G(v)-G(w))\leq\varepsilon$ for any adjacent vertices $v$ and $w$, and that is not possible for an arbitrary $\varepsilon$-fine family $F$ as the one we have). Several attempts to use this approach or a similar one have been made by the author (for example, allowing finite sums of the $\tau_{vw}$) but all of them have failed. Nevertheless, (\ref{Expression 2 for Fj}) seems to be much more versatile to work with because we do not have the previously mentioned problem with the vertex $v_{C}$ and also it allows to circumvent other technical issues that arise when interpolating in cells of dimension greater than $1$. Let us explain first how that interpolation works.

Given a cell $C$ and $D\in\mathcal{D}_{C}$, let $w^{D}_{C}$ be the vertex of $C$ that minimizes $\mass(F(v)\llcorner D)$ among all $v\in V(C)$. As we want to decrease mass, the strategy is to start from $x\in\partial C$ and replace $F(w^{D}_{C}(x))\llcorner D$ by $F(w^{D}_{C})\llcorner D$ one domain $D\in\mathcal{D}_{C}$ at a time as we approach the center of the cell $C$. This is done by subtracting at each step $\partial[\tau(w_{C}^{D}(x),w_{C}^{D})]$ to our $k$-cycle, which causes a change in $I_{j}(x)$ by subtracting $\tau(w_{C}^{D}(x),w_{C}^{D})\llcorner\partial D$ (which is small in mass provided the grid $\mathcal{D}_{C}$ is chosen accurately). When we get to the centre of $C$ we have $w_{C}^{D}(x)=w_{C}^{D}$ for all $D$ by construction. Nevertheless, it might happen that $I_{j}(x_{1})\neq I_{j}(x_{2})$ for some $x_{1},x_{2}\in\partial\Center(C)$. This is due to the fact that in general
\begin{equation*}
    \tau(v_{1},v_{2})+\tau(v_{2},v_{3})\neq\tau(v_{1},v_{3})
\end{equation*}
as discussed before. Hence we could have the following situation. Consider the case of a $2$-cell $C$ with $V(C)=\{v_{1},v_{2},v_{3},v_{4}\}$ and two $1$-faces $E_{1},E_{2}$ with vertices $v_{1},v_{2}$ and $v_{1},v_{3}$ respectively. Suppose that for some $D_{1}\in\mathcal{D}_{C}$ 
\begin{equation*}
    \mass(F(v_{4})\llcorner D_{1})>\mass(F(v_{1})\llcorner D_{1})>\mass(F(v_{2})\llcorner D_{1})>\mass(F(v_{3})\llcorner D_{1}),
\end{equation*}
that for $D\in\mathcal{D}_{C}\setminus\{D_{1}\}$
\begin{equation*}
    \mass(F(v_{1})\llcorner D)<\mass(F(v_{j})\llcorner D)
\end{equation*}
for $j=2,3,4$ and that $\mathcal{D}_{E_{1}}=\mathcal{D}_{E_{2}}=\mathcal{D}_{C}$ for simplicity. Then we will have points $x_{1}\in E_{1}$ and $x_{2}\in E_{2}$ such that $F_{1}(x_{1})=F(v_{1})+\partial[\tau(v_{1},v_{2})\llcorner D_{1}]$ and $F_{1}(x_{2})=F(v_{1})+\partial[\tau(v_{1},v_{3})\llcorner D_{1}]$. And when extending to the $2$-skeleton, we will get points $y_{1},y_{2}\in\partial\Center(C)$ such that $F_{2}(y_{1})=F(v_{1})+\partial[(\tau(v_{1},v_{2})+\tau(v_{2},v_{3}))\llcorner D_{1}]$ and $F_{2}(y_{2})=F(v_{1})+\partial[\tau(v_{1},v_{3})\llcorner D_{1}]$. Thus $w^{D}_{C}(y_{1})=w_{C}^{D}(y_{2})$ for every $D\in\mathcal{D}_{C}$ but maybe $F_{2}(y_{1})\neq F_{2}(y_{2})$ as it could be the case that
\begin{equation*}
    (\tau(v_{1},v_{2})+\tau(v_{2},v_{3}))\llcorner\partial D_{1}\neq\tau(v_{1},v_{3})\llcorner\partial D_{1}.
\end{equation*}
Nevertheless, if
\begin{equation*}
    z_{C}=F(v_{C})-\sum_{D\in\mathcal
    {D}_{C}}\partial\tau(v_{C},w_{C}^{D})
\end{equation*}
the family of $k$-cycles $I_{j}^{C}(x)=F_{j}(x)-z_{C}$ with domain $\partial\Center(C)$ is $\delta_{j}$-localized and has very small mass. Therefore we can contract it in a $\delta$-localized and mass-controlled way using Proposition \ref{Prop filling small families} and hence extend $F_{j}$ to $\Center(C)$ by $F_{j}(x)=z_{C}+I_{j}^{C}(x)$. This is the second advantage of (\ref{Expression 2 for Fj}) that we previously mentioned: the term $I_{j}(x)$ allows to compensate the differences in $F_{j}(x)$ generated by the failure of $\tau(v_{1},v_{2})+\tau(v_{2},v_{3})=\tau(v_{1},v_{3})$. This step also justifies why we proved Proposition \ref{Prop filling small families}. We now proceed to formalize this argument, starting by stating the required definitions.

\begin{definition}
    Given $\tau_{1},...,\tau_{K}\in\mathcal{I}_{k+1}(M)$ and a collection of balls $\mathcal{B}=\{B_{1},...,B_{L}\}$ we say that $\mathcal{B}$ is compatible with $\tau_{1},...,\tau_{K}$ if
    \begin{equation*}
        \mass(\tau_{i}\llcorner \partial B_{l})\leq K\frac{\mass(\tau_{i})}{r}
    \end{equation*}
    for every $1\leq i\leq K$ and every $1\leq l\leq L$.
\end{definition}

\begin{lemma}[Coarea inequality for several chains]\label{Lemma balls}
    Given $\tau_{1},...,\tau_{K}\in\mathcal{I}_{k+1}(M)$, for each $1\leq l\leq L$ we can find a radii $r_{l}\in[r,2r]$ such that the balls $B_{l}=B(x_{l},r_{l})$ form a family which is compatible with $\tau_{1},...,\tau_{K}$.
\end{lemma}
\begin{proof}
    Fix $l\in\{1,...,L\}$ and for each $1\leq i\leq K$ define $A_{i}=\{s\in(r,2r):\mass(\tau_{i}\llcorner\partial B_{s})>K\frac{\mass(\tau_{i})}{r}\}$ where $B_{s}=B(x_{l},s)$. Observe that
    \begin{equation*}
        \mass(\tau_{i}) \geq\int_{r}^{2r}\mass(\tau_{i}\llcorner\partial B_{s})ds>\int_{A_{i}}K\frac{\mass(\tau_{i})}{r}ds=K\frac{\mass(\tau_{i})}{r}\mathcal{L}^{1}(A_{i})
    \end{equation*}
    which implies that $\mathcal{L}^{1}(A_{i})<\frac{r}{K}$. Therefore $\mathcal{L}^{1}(\bigcup_{i=1}^{K}A_{i})<r$ and hence there exists $r_{l}\in(r,2r)\setminus\bigcup_{i=1}^{K}A_{i}$ which verifies the required coarea inequality.
\end{proof}

\begin{definition}
    A grid $\mathcal{D}=\{D_{1},...,D_{L'}\}$ on $M$ is a finite collection of closed domains with piecewise smooth boundary such that $D_{l}\cap D_{l'}=\partial D_{l}\cap\partial D_{l'}$ for every $l,l'$ and $\bigcup_{l=1}^{L'}D_{l}=M$.
\end{definition}

\begin{definition}\label{From balls to grids}
    Given a list of balls $\{B_{1},...,B_{L}\}$ with $B_{l}=B(x_{l},r_{l})$ and $r_{l}\in[r,2r]$ as before, it induces a grid $\mathcal{D}=\{D_{1},...,D_{L}\}$ given by
    \begin{equation*}
        D_{l}=\overline{B}_{l}\setminus\bigcup_{l'<l}B_{l'}.
    \end{equation*}
\end{definition}

\begin{definition}
    Given a vertex $v$ of $X$, we denote
    \begin{align*}
        \Adj(v) & =\{w\in X_{0}:\exists C\text{ cell of }X\text{ s.t. }v,w\in C\}\\
        & =\bigcup_{\substack{C\in\faces_{p}(X) \\ v\in C}}V(C).
    \end{align*}
    As $\dim(X)=p$, by Lemma \ref{Lemma equiv def of widths} we can assume $X\subseteq[0,1]^{2p+1}$ and therefore $|\Adj(v)|\leq C(p)$.
\end{definition}

\begin{definition}
    Given a cell $C$ of $X$, we denote  $\mathcal{T}_{C}=\{\tau_{vw}:v,w\in C\}$. 
\end{definition}

\begin{definition}
    Given $q\in\mathbb{N}$ and a cell $C$ of $X(q)$, we denote by $V(C)$ its set of vertices: $V(C)=C\cap X(q)_{0}$.
\end{definition}

\begin{definition}\label{Def grids 1 cells}
    Let $v$ be a vertex of $X$. Using Lemma \ref{Lemma balls}, we associate to $v$ a collection of balls $B_{l}^{v}=B(x_{l},r_{l}^{v})$ such that for each cell $C$ containing $v$, $B_{1}^{v},...,B_{L}^{v}$ is compatible with the elements of $\mathcal{T}_{C}$. Equivalently, $\{B_{l}^{v}:1\leq l\leq L\}$ is compatible with $\tau_{w,w'}$ for every $w,w'\in\Adj(v)$. We consider the induced grid $\{D_{1}^{v},...,D_{L}^{v}\}$ in $M$ as in Definition \ref{From balls to grids}.
\end{definition}

\begin{remark}\label{Rk compatibility with grids}
    Let $v$ be a vertex of $X$, let $C$ be a cell containing $v$ and let $\tau\in\mathcal{T}_{C}$. Then 
\begin{equation*}
    \mass(\tau\llcorner\partial B_{l}^{v})\leq C(p)\frac{\mass(\tau)}{r}
\end{equation*}
for every $1\leq l\leq L$ and therefore
\begin{equation*}
    \sum_{l=1}^{L}\mass(\tau\llcorner\partial D_{l}^{v})\leq C(p)L\frac{\mass(\tau)}{r}.
\end{equation*}
\end{remark}

\begin{definition}
    Given $\mathcal{D}_{1},...,\mathcal{D}_{J}$ grids in $M$, we define the intersection $\mathcal{D}=\bigcap_{j=1}^{J}\mathcal{D}_{j}$ as the grid whose domains are all possible finite intersections $D_{1}\cap...\cap D_{J}$ with $D_{j}\in\mathcal{D}_{j}$.
\end{definition}

\begin{definition}\label{Def grids for any cell}
    To each cell of $X$ we associate a grid in the following way. For $0$-dimensional cells we just consider $\mathcal{D}_{v}=\{D_{1}^{v},...,D_{L}^{v}\}$ as defined above. If $1\leq\dim(C)\leq p$ we define
    \begin{equation*}
        \mathcal{D}_{C}=\bigcap_{v\in V(C)}\mathcal{D}_{v}.
    \end{equation*}
    We enumerate the elements of $\mathcal{D}_{C}$ and denote $l(D)$ the number assigned to the domain $D\in\mathcal{D
    }_{C}$. Observe that if $C\subseteq C'$ then $\mathcal{D}_{C'}$ is a refinement of $\mathcal{D}_{C}$ and that if $\dim(C)=j$ then $|\mathcal{D}_{C}|\leq L^{2^{j}}$.
\end{definition}

\begin{remark}
    If $\dim(C)=j$ and $v,w$ are vertices of $C$ then by Remark \ref{Rk compatibility with grids}
    \begin{equation*}
        \sum_{D\in\mathcal{D}_{C}}\mass(\tau_{vw}\llcorner\partial D)\leq C(p)2^{j} L\frac{\mass(\tau_{vw})}{r}
    \end{equation*}
    where $2^{j}$ appears in the formula because of being the number of vertices of a $j$-dimensional cube, which is the number of grids that we are intersecting.
\end{remark}

\begin{definition}
    Given odd numbers $q,q'\in\mathbb{N}$, we define $\Lambda_{q}:X(qq')_{0}\to X(q)_{0}$ to be the function which maps each $x$ to the closest $v\in X(q)_{0}$. Notice that if $F:X(q)\to\mathcal{I}_{k}(M)$ then $R^{q'}F(x)=F(\Lambda_{q}(x))$.
\end{definition}

Now we proceed to prove Proposition \ref{Prop delta loc aprox cycles} by induction.

\textbf{Inductive property.} For each $1\leq j\leq p$, there exist discrete families $F_{j}:X_{j}(q_{j})_{0}\to\mathcal{Z}_{k}(M)$ and $I_{j}:X_{j}(q_{j})_{0}\to\mathcal{I}_{k}(M)$ so that the following is true. Let $C$ be a cell of $X$ with $\dim(C)= j$. For each $x\in C\cap X_{j}(q_{j})_{0}$ there exists a collection $(w_{C}^{D}(x))_{D\in\mathcal{D}_{C}}$ with $w_{C}^{D}(x)\in V(C)$ for every $D\in\mathcal{D}_{C}$; and for each $j$-cell $E$ of $X_{j}(q_{j})$ contained in $C$ there exists an $(N(k,j),N(k,j)r)$-admissible family $\{B_{i}^{E}\}_{i\in I_{E}}$ verifying the following properties.
    \begin{enumerate}[label=(B\arabic*)]
        \item \label{ite : 2.1} $F_{j}(x)=\sum_{D\in\mathcal{D}_{C}}F(w_{C}^{D}(x))\llcorner D+I_{j}(x)$.
        \item \label{ite : 2.2} $\sum_{D\in\mathcal{D}_{C}}\mass(F(w_{C}^{D}(x))\llcorner D)\leq \max\{\mass(F(v)):v\in V(C)\}$.
        \item \label{ite : 2.3} $\mass(I_{j}(x))\leq C(p)c(k,j)(\frac{L}{r})^{n(k,j)}\varepsilon$.
        \item \label{ite : 2.4} For each $E\subseteq C$, $E\in\faces_{j}(X_{j}(q_{j}))$ there exists a subset $\mathcal{D}_{C}^{E}\subseteq\mathcal{D}_{C}$ such that
        \begin{enumerate}
            \item If $x,y\in V(E)$ then $w^{D}_{C}(x)=w^{D}_{C}(y)$ for all $D\in\mathcal{D}_{C}\setminus\mathcal{D}_{C}^{E}$.
            \item $\{B_{i}^{E}\}_{i\in I_{E}}$ covers $\bigcup_{D\in\mathcal{D}_{C}^{E}}D$.
        \end{enumerate}
        \item \label{ite : 2.5} $I_{j}|_{E\cap X(q_{j})_{0}}$ is localized in $\{B_{i}^{E}\}_{i\in I_{E}}$.
        \item \label{ite : 2.6} $F_{j}$ is $C(p)c(k,j)(\frac{L}{r})^{n(k,j)}\varepsilon$-fine in $X_{j}$.
       \item \label{ite : 2.7} If $E\in\faces_{j-1}(X)$ is contained in $\partial C$ and $x\in E\cap X_{j}(q_{j})_{0}$, then
       \begin{equation*}
           w_{C}^{D}(x)=w_{E}^{D'}(\Sigma_{j}(x))
       \end{equation*}
       where $D'\in\mathcal{D}_{E}$ is the unique domain such that $D\subseteq D'$ and
       \[
       \Sigma_{j}(x)=\begin{cases}
           x & \text{ if } j=1\\
           \Xi\circ\Lambda_{3q_{j-1}}\circ\Xi^{3}(x) & \text{ if } j\geq 2.
       \end{cases}
       \]
    \end{enumerate}

We need to introduce the following definitions.

\begin{definition}\label{Def ZE}
    Given a $j$-cell $C$ of $X(q)$ and $D\in\mathcal{D}_{C}$, we define $w^{D}_{C}\in V(C)$ to be a vertex such that
    \begin{equation*}
        \mass(F(w^{D}_{C})\llcorner D)=\min\{\mass(F(w)\llcorner D):w\in V(C)\}.
    \end{equation*}
    We denote
    \begin{align*}
        \tau_{C} & =\sum_{D\in\mathcal{D}_{C}}\tau(v_{C},w^{D}_{C})\llcorner D;\\
         z_{C} & =F(v_{C})-\partial\tau_{C};\\
         I_{C} & =-\sum_{D\in\mathcal{D}_{C}}\tau(v_{C},w^{D}_{C})\llcorner \partial D.
    \end{align*}
\end{definition}

\begin{remark}
    Notice that
    \begin{equation*}
        z_{C}=\sum_{D\in\mathcal{D}_{C}}F(w^{D}_{C})\llcorner D+I_{C}
    \end{equation*}
    minimizes
    \begin{equation}\label{Sum of mass F(w) in D}
        \sum_{D\in\mathcal{D}_{C}}\mass(F(w^{D}_{C}(x))\llcorner D)
    \end{equation}
    among all possible values of $F_{j}(x)$ for $x\in C\cap X_{j}(q_{j})_{0}$ according to our inductive hypothesis. Thus when extending $F_{j}$ from $\partial C$ to all of $C$, we will homotop $F_{j}|_{\partial C\cap X_{j}(q_{j})_{0}}$ to $z_{C}$ in a way that decreases (\ref{Sum of mass F(w) in D}) as we approach $\Center(C)$.
\end{remark}

\noindent\textbf{Base case.} Let $C$ be a $1$-cell of $X$. Denote $w_{C}$ the vertex of $C$ which is not $v_{C}$. Let $q_{1}\in\mathbb{N}$ be the smallest odd number such that $q_{1}\geq 3L^{2^{p}}$ so that $q_{1}\geq 3|\mathcal{D}_{C'}|$ for every cell $C'$ of $X$. For $x\in C\cap X(q_{1})_{0}$, we define

\[
w^{D}_{C}(x)=\begin{cases}
    w^{D}_{C} & \text{ if } l(D)\leq d_{\infty}(x,\partial C)\\
    v_{C} & \text{ if } l(D)>d_{\infty}(x,v_{C})\\
    w_{C} & \text{ if } l(D)>d_{\infty}(x,w_{C})
\end{cases}
\]
Notice that $w^{D}_{C}(x)\in\{v_{C},w_{C}\}$ for every $x$ and if $d_{\infty}(x,\partial C)\geq L$ we have $w^{D}_{C}(x)=w^{D}_{C}$ for every $D$ (in particular, $w^{D}_{C}(x)=w^{D}_{C}$ for $x\in\Center(C)$). We set
\begin{equation*}
    \tau_{C}(x)=\sum_{D\in\mathcal{D}_{C}}\tau(v_{C},w^{D}_{C}(x))\llcorner D
\end{equation*}
and
\begin{equation*}
    F_{1}(x)=F(v_{C})-\partial\tau_{C}(x)
\end{equation*}
in both cases for $x\in C\cap X(q_{1})_{0}$. Notice that the previous induces a well-defined extension $F_{1}:X_{1}(q_{1})_{0}\to\mathcal{Z}_{k}(M)$ of $F_{0}$. Define
\begin{equation*}
    I_{1}(x)=-\sum_{D\in\mathcal{D}_{C}}\tau(v_{C},w^{D}_{C}(x))\llcorner\partial D
\end{equation*}
so that
\begin{equation*}
    F_{1}(x)=\sum_{D\in\mathcal{D}_{C}}F(w^{D}_{C}(x))\llcorner D +I_{1}(x).
\end{equation*}
Given $E\in\faces_{1}(C(q_{1}))$ with vertices $x$ and $y$, by construction we know that there exists $D_{E}\in\mathcal{D}_{C}$ such that $w^{D}_{C}(x)=w^{D}_{C}(y)$ for every $D\neq D_{E}$. Define $\mathcal{D}_{C}^{E}=\{D_{E}\}$ and $\{B_{i}^{E}\}_{i\in I_{E}}=\{B(x_{l},2r)\}$ if $D=D_{l}^{E}$ (see Definition \ref{Def grids 1 cells} and Definition \ref{Def grids for any cell}). Because of the fact that
\begin{align*}
    F_{j}(y)-F_{j}(x) & =\partial[\tau_{C}(x)-\tau_{C}(y)]\\
    & =\partial[(\tau(v_{C},w_{C}^{D_{E}}(y))-\tau(v_{C},w_{C}^{D_{E}}(x)))\llcorner D_{E}]
\end{align*}
and
\begin{align*}
    I_{j}(y)-I_{j}(x)=(\tau(v_{C},w_{C}^{D_{E}}(y))-\tau(v_{C},w_{C}^{D_{E}}(x)))\llcorner \partial D_{E}
\end{align*}
we can see that \ref{ite : 2.1} to \ref{ite : 2.7} in the inductive property hold with $N(k,1)=n(k,1)=1$ and $c(k,1)=2$.

\textbf{Inductive step.} Assume $F_{j-1}:X_{j-1}(q_{j-1})_{0}\to\mathcal{Z}_{k}(M)$ has been defined satisfying the inductive property. Let $\overline{q}_{j}=3q_{j-1}$. First we will define $\overline{F}_{j}(x)$, 
$\overline{I}_{j}(x)$ and $\overline{w}^{D}_{C}(x)$ with domain $(C\setminus\Center(C))\cap X(\overline{q}_{j})_{0}$ for each $j$-cell $C$ of $X(q)$. Fix a $j$-cell $C$. For $x\in\partial C\cap X(\overline{q}_{j})_{0}$ we set $\overline{F}_{j}(x)=F_{j-1}(\Xi(x))$, $\overline{I}_{j}(x)=I_{j-1}(\Xi(x))$ and we define $w^{D}_{\partial C}(x)=w^{D'}_{E}(x)$ where $E\subseteq\partial C$ is a $(j-1)$-cell of $X$ containing $x$ and $D'\in\mathcal{D}_{E}$ is the unique domain such that $D\subseteq D'$. The previous is well defined by (B7) in the \textbf{Inductive Property}. For $x\in (C\setminus\Center(C))\cap X(\overline{q}_{j})_{0}$ we set
\[
\overline{w}^{D}_{C}(x)=\begin{cases}
    w^{D}_{C} & \text{ if }l(D)\leq d_{\infty}(x,\partial C)\\
    w_{\partial C}^{D}(\Xi(x)) & \text{ otherwise}
\end{cases}
\]
and
\begin{equation*}
    \overline{F}_{j}(x)=F_{j-1}(\Xi(x))-\partial\big[\sum_{\substack{D\in\mathcal{D}_{C} \\ l(D)\leq d_{\infty}(x,\partial C)}}\tau(w^{D}_{\partial C}(\Xi(x)),w^{D}_{C})\llcorner D\big].
\end{equation*}
Then it is clear that if
\begin{equation*}
    \overline{I}_{j}(x)=I_{j-1}(\Xi(x))-\sum_{\substack{D\in\mathcal{D}_{C} \\ l(D)\leq d_{\infty}(x,\partial C)}}\tau(w^{D}_{\partial C}(\Xi(x)),w^{D}_{C})\llcorner\partial D
\end{equation*}
it holds
\begin{equation*}
    \overline{F}_{j}(x)=\sum_{D\in\mathcal{D}_{C}}F(\overline{w}^{D}_{C}(x))\llcorner D+\overline{I}_{j}(x)
\end{equation*}
and $\overline{F}_{j}$ is well-defined as it extends $F_{j-1}\circ\Xi:X_{j-1}(\overline{q}_{j})_{0}\to\mathcal{Z}_{k}(M)$ (same for $\overline{I}_{j}$). Consider $E\in\faces_{j}(X_{j}(\overline{q}_{j}))$ such that $E\subseteq C\setminus\Center(C)$ and let $E'\in\faces_{j-1}(X_{j-1}(q_{j-1}))$ and $C'\in\faces_{j-1}(X)$ be such that $\Xi(E)\subseteq E'\subseteq C'$ (there could be more than one pair of cells $(E',C')$ with this property, in that case we just pick one). Denote $l(x)=d_{\infty}(x,\partial C)$, then there exists $l\in\mathbb{N}$ such that if $x\in E$ either $l(x)=l-1$ or $l(x)=l$. Define
\begin{equation*}
    \mathcal{D}_{C}^{E}=\{D\in\mathcal{D}_{C}:D\subseteq D'\text{ for some }D'\in\mathcal{D}_{C'}^{E'}\}\cup\{D_{E}\}
\end{equation*}
where $D_{E}\in\mathcal{D}_{C}$ is the unique domain which verifies $l(D_{E})=l$. Let $\{B_{i}^{E}\}_{i\in I_{E}}$ be an admissible family which covers $D_{E}\cup \bigcup_{i\in I_{E'}}B_{i}^{E'}$. Notice that under this choice, \ref{ite : 2.4} and \ref{ite : 2.5} follow. We now prove the $\varepsilon$-fineness of $\overline{F}_{j}$. Let $C'$ be a face of $\partial C$ which contains $\Xi(x)$. Then

\begin{equation*}
    \mathcal{F}(\overline{F}_{j}(x),F(v_{C}))\leq\mathcal{F}(\overline{F}_{j}(x),F_{j-1}(\Xi(x)))+\mathcal{F}(F_{j-1}(\Xi(x)),F(v_{C'}))+\mathcal{F}(F(v_{C'}),F(v_{C})).
\end{equation*}

The last two terms of the RHS are bounded by $C(p)c(k,j-1)(\frac{L}{r})^{n(k,j-1)}\varepsilon$ and $\varepsilon$, and the first term is bounded by

\begin{align*}
    \sum_{D\in\mathcal{D}_{C}}\mass\big(\tau(w^{D}_{\partial C}(\Xi(x)),w^{D}_{C})\llcorner D\big) & \leq\sum_{D\in \mathcal{D}_{C}}\sum_{v,w\in V(C)}\mass(\tau_{v,w}\llcorner D)\\
    & =\sum_{v,w\in V(C)}\mass(\tau_{vw}) \\
    & \leq {2^{j} \choose 2}\varepsilon.
\end{align*}
Therefore, properties \ref{ite : 2.1} to \ref{ite : 2.7} follow.


Observe that for $x\in\partial\Center(C)\cap X(\overline{q}_{j})_{0}$ it holds $\overline{w}^{D}_{C}(x)=w^{D}_{C}$ for every $D$. Thus by \ref{ite : 2.1} and \ref{ite : 2.5} we can see that
\begin{equation*}
    \overline{I}^{C}_{j}(x)=\overline{F}_{j}(x)-z_{C}=\overline{I}_{j}(x)-I_{C}
\end{equation*}
is an $(N(k,j),N(k,j)r)$-localized family of $k$-cycles with
\begin{equation*}
    \mass(\overline{I}_{j}^{C}(x))\leq C(p)c(k,j)(\frac{L}{r})^{n(k,j)}\varepsilon.
\end{equation*}
Therefore we can refine it and find a $\delta$-localized and small in mass filling of it by Proposition \ref{Prop filling small families}. To be precise, we can define $\tilde{q}_{j}=\overline{q}_{j} q(j,k,M,N(k,j)r)$, $\tilde{I}^{C}_{j}=R^{q(j,k,M)}\overline{I}^{C}_{j}$ and $\tilde{\tau}_{j}:(\partial \Center(C))\cap X(\tilde{q}_{j})_{0}\to\mathcal{I}_{k+1}(M)$ such that $\partial\tilde{\tau}_{j}(x)=\tilde{I}^{C}_{j}(x)$, $\tilde{\tau}_{j}$ is $(N(k,j),N(k,j)r)$-localized and is bounded in mass by $C(p)c(k,j)(\frac{L}{r})^{n(k,j)}\varepsilon$. Then we can ``extend'' $\tilde{\tau}_{j}$ to a map $\tau_{j}:\Center(C)\cap X(q_{j})_{0}\to\mathcal{I}_{k+1}(M)$, $q_{j}=3\tilde{q}_{j}$ by a chopping procedure as in Lemma \ref{Lemma chopping}:
\[
\tau_{j}(x)=\begin{cases}
    d_{l(x)}(\tilde{\tau}_{j}(\Xi(x))) & \text{ if } l(x)=d_{\infty}(x,\partial\Center(C))\leq L\\
    0 & \text{ otherwise.}
\end{cases}
\]
Defining 
\begin{align*}
    I_{j}^{C}(x) & =\tilde{I}_{j-1}\circ\Xi^{3}(x)\\
    F_{j}(x) & =R^{q(j,k,M)}\overline{F}_{j}(\Xi^{3}(x))\\
    w^{D}_{C}(x) & =R^{q(j,k,M)}\overline{w}^{D}_{C}(\Xi^{3}(x))
\end{align*}
for $x\in (C\setminus\Center(C))\cap X_{j}(q_{j})_{0}$ we can extend $I_{j}^{C}$ to $\Center(C)$ as $I_{j}^{C}(x)=\partial\tau_{j}(x)$ and use that extension to extend $F_{j}$ to $\Center(C)$ as
\begin{equation*}
    F_{j}(x)=z_{C}+I_{j}^{C}(x).
\end{equation*}
All the previous constructions preserve $\delta$-localization and the required upper bounds in mass. Setting $w_{C}^{D}(x)=w_{D}^{C}$ and $I_{j}(x)=I_{C}+I_{j}^{C}(x)$ for all $x\in\Center(C)\cap X(q_{j})_{0}$, and $\mathcal{D}_{C}^{E}=\emptyset$ for each $j$-face $E\subseteq\Center(E)$ of $X_{j}(q_{j})$ we see that properties \ref{ite : 2.1} to \ref{ite : 2.7} hold.

\subsection{Approximating discrete families of chains}\label{Subsection approximating families of chains}

Now we want to prove a version of Proposition \ref{Prop delta loc aprox cycles} for families of chains.

\begin{proposition}\label{Prop delta loc aprox chains}
    There exist constants $\varepsilon_{0}(k,p,M)>0$, $\delta_{0}(k,p,M)>0$, $C(k,p,M)\in\mathbb{N}$, $n(k,p)\in\mathbb{N}$ and $N(k,p)\in\mathbb{N}$ such that the following is true. Let $G:X_{0}\to\mathcal{I}_{k+1}(M)$ be an $\varepsilon$-fine family, $\varepsilon\leq\varepsilon_{0}$, $0\leq k\leq n-1$, $\dim(X)=p$ and denote $F=\partial G$. Let $0<\delta<\delta_{0}$. Then there exists a discrete family $G':X(q)_{0}\to\mathcal{I}_{k+1}(M)$ for $q=q(k,p,M,\delta)$ such that if we denote $F'=\partial G'$ then
    \begin{enumerate}
        \item $G'(x)=G(x)$ for every $x\in X_{0}$.
        \item Given a cell $C$ of $X$ and $x\in C$,
        \begin{equation*}
            \mass(G'(x))\leq\max_{v\in V(C)}\{\mass(G(v))\}+c(k,p)(\frac{L}{\delta})^{n(k,p)}\varepsilon
        \end{equation*}
        and
        \begin{equation*}
            \mass(F'(x))\leq\max_{v\in V(C)}\{\mass(F(v))\}+c(k,p)(\frac{L}{\delta})^{n(k,p)}\varepsilon.
        \end{equation*}
        \item $G'$ is $c(k,p)(\frac{L}{\delta})^{n(k,p)}\varepsilon$-fine in $X$ and $(N(k,p),\delta)$-localized in $X(q)$.
    \end{enumerate}
\end{proposition}

In order to obtain the previous, we first prove the following result which is an extension of Proposition \ref{Prop filling small families} to families of cycles with large flat norm.

\begin{proposition}\label{Prop filling big families}
    There exist constants $\varepsilon_{0}(k,p,M)>0$, $\delta_{0}(k,p,M)>0$, $q(k,p,M)\in\mathbb{N}$ and $K(N,k,p)\in\mathbb{N}$ such that the following is true. Let $\varepsilon\leq\varepsilon_{0}$ and let $G:X_{0}\to\mathcal{I}_{k+1}(M)$ be an $\varepsilon$-fine discrete family of chains with $\dim(X)=p$ such that the family $F=\partial G:X_{0}\to\mathcal{I}_{k}(M)$ is $(N,\delta)$-localized for some $\delta\leq\delta_{0}$. Then there exists a discrete family $G':X(q)_{0}\to\mathcal{I}_{k+1}(M)$ such that
    \begin{enumerate}
        \item $G'(x)=G(x)$ for every $x\in X_{0}$.
        \item $\partial G'(x)=R^qF(x)$.
        \item $G'$ is $c(k,p)(\frac{L}{\delta})^{n(k,p)}\varepsilon$-fine in $X$ and $(K(N,k,p),K(N,k,p)\delta)$-localized in $X(q)$.
        \item Given $C$ a cell of $X$ and $x\in C\cap X(q)_{0}$,
        \begin{equation*}
            \mass(G'(x))\leq\max_{v\in V(C)}\{\mass(G(v))\}+c(k,p)(\frac{L}{\delta})^{n(k,p)}\varepsilon.
        \end{equation*}
    \end{enumerate}
\end{proposition}

\begin{proof}[Proof of Proposition \ref{Prop filling big families}]
    We are going to use the same strategy as in the proof of Proposition \ref{Prop filling small families}. Namely, we will inductively construct refinements of
    \begin{equation*}
    \tilde{X}=\bigsqcup_{C\in\faces_{p}(X)}\bigsqcup_{v\in V(C)}C_{v}
    \end{equation*}
    and define maps $G_{j}:\tilde{X}_{j}(q_{j})_{0}\to\mathcal{I}_{k+1}(M)$ taking the equivalence relation $\sim$ on $\faces(\tilde{X})$ into account so that  $\varepsilon$-fineness and $\delta$-localization also hold for adjacent vertices in $X(2q_{j}+1)$ which are not neighbours in $\tilde{X}(q_{j})$. The main difference with Proposition \ref{Prop filling small families} is that the chains $G(x)$ may have very large mass and then we have to be careful in our construction in order to avoid amplifying the volumes of our chains by big factors. We will do that by using the interpolation techniques from Proposition \ref{Prop delta loc aprox cycles}.

    We define $G_{0}:\tilde{X}_{0}\to \mathcal{I}_{k+1}(M)$ in the a very similar way as in Proposition \ref{Prop filling small families}. For $v\in X_{0}\subseteq X(3)_{0}=\tilde{X}_{0}$ we set $G_{0}(v)=G(v)$. Given two adjacent vertices $x,y$ in $\tilde{X}_{0}\setminus X_{0}$, let $E$ be the unique cell of $X$ such that $x,y\in\Center(E)$ and denote $v=\Xi(x)$, $w=\Xi(y)$ being $v,w\in V(E)$. As $\mathcal{F}(F(v),F(w))\leq \mathcal{F}(G(v),G(w))\leq\varepsilon$ and $F$ is $(N,\delta)$-localized, there exist $\tau(v,w)\in\mathcal{I}_{k+1}(M)$ supported in $\bigcup_{i\in I_{E}}U_{i}^{E}$ such that $\partial\tau(v,w)=F(v)-F(w)$ and $\mass(\tau(v,w))\leq\varepsilon$. Thus given a face $E$ of $X$ and $x\in V(\Center(E))$ we define
    \begin{equation*}
        G_{0}(x)=G(v_{E})+\tau(\Xi(x),v_{E}).
    \end{equation*}
    This implies that $\partial G_{0}(x)=F(\Xi(x))$ and therefore $G_{0}$ verifies
    \begin{enumerate}
        \item If $x,y$ belong to the same top cell $C_{v}$ of $\tilde{X}_{0}$ then the chains $G_{0}(x)$, $G_{0}(y)$ have the same boundary $F(v)$.
        \item If $d(x,y)=1$ in $\tilde{X}_{0}$ but $x$ and $y$ do not belong to the same $C_{v}$ then there exists a unique face $E$ of $X$ such that $x,y\in\Center(E)$ and $G_{0}(x)-G_{0}(y)$ is supported in $\bigcup_{i\in I_{E}}U_{i}^{E}$ and has mass at most $\varepsilon$.
    \end{enumerate}
    These properties (or an analog one in the case of property 2) will be maintained when we go to higher dimensional skeleta and will allow us to interpolate with the method from Proposition \ref{Prop delta loc aprox cycles} as if the families of chains $G_{j}:C_{v}\cap X_{j}(q_{j})_{0}\to\mathcal{I}_{k+1}(M)$ with constant boundary $F(v)$ were families of cycles. The maps $G_{1},...,G_{p}$ are recursively constructed according the following.

    \textbf{Inductive Property.} For each $1\leq j\leq p$, there exist discrete families $G_{j},I_{j}:\tilde{X}_{j}(q_{j})_{0}\to \mathcal{I}_{k+1}(M)$ such that the following is true. For each $E\in\faces_{j}(\tilde{X}_{j})$ and $x\in E\cap\tilde{X}_{j}(q_{j})_{0}$ there are coefficients $(w_{E}^{D}(x))_{D\in\mathcal{D}_{E}}\in V(E)^{\mathcal{D}_{E}}$, and for each $j$-cell $F$ of $\tilde{X}_{j}(q_{j})$ there exists a $(K(N,k,j),K(N,k,j)\delta)$-admissible family $\{B_{i}\}_{i\in I_{F}}$ such that
    \begin{enumerate}[label=(C\arabic*)]
        \item $G_{j}(x)=G_{0}(x)$ for every $x\in\tilde{X}_{0}$.
        \item $\partial G_{j}(x)=R^{q_{j}}F(x)$.
        \item $G_{j}$ is $c(k,j)(\frac{L}{\delta})^{n(k,j)}\varepsilon$-fine.
        \item If $x\in E$ for some $j$-cell $E$ of $\tilde{X}$ then
        \begin{equation*}
            G_{j}(x)=\sum_{D\in\mathcal{D}_{E}}G(w^{D}_{E}(x))\llcorner D+I_{j}(x).
        \end{equation*}
        \item $\sum_{D\in\mathcal{D}_{E}}\mass(G(w_{E}^{D}(x))\llcorner D)\leq \max\{\mass(G(v)):v\in V(E)\}+C(j)\varepsilon$.
        \item $\mass(I_{j}(x))\leq C(p)c(k,j)(\frac{L}{\delta})^{n(k,j)}$.
        \item $I_{j}|_{F\cap X(q_{j})_{0}}$ is localized in $\{B_{i}^{F}\}_{i\in I_{F}}$.
        \item For each $F\subseteq E$, $E\in\faces_{j}(X_{j}(q_{j}))$ there exists a subset $\mathcal{D}_{E}^{F}\subseteq\mathcal{D}_{E}$ such that
        \begin{enumerate}
            \item If $x,y\in V(F)$ then $w^{D}_{E}(x)=w^{D}_{E}(y)$ for all $D\in\mathcal{D}_{E}\setminus\mathcal{D}_{E}^{F}$.
            \item $\{B_{i}^{F}\}_{i\in I_{F}}$ covers $\bigcup_{D\in\mathcal{D}_{E}^{F}}D$.
        \end{enumerate}
        \item If $E'\in\faces_{j-1}(X)$ is contained in $\partial E$ and $x\in E'\cap X_{j}(q_{j})_{0}$, then
       \begin{equation*}
           w_{E}^{D}(x)=w_{E'}^{D'}(\Sigma_{j}(x))
       \end{equation*}
       where $D'\in\mathcal{D}_{E'}$ is the unique domain such that $D\subseteq D'$ and
       \[
       \Sigma_{j}(x)=\begin{cases}
           x & \text{ if } j=1\\
           \Xi\circ\Lambda_{3q_{j-1}}\circ\Xi^{3}(x) & \text{ if } j\geq 2.
       \end{cases}
       \]
       \item If $F\sim F'$ then $\{B_{i}\}_{i\in I_{F}}=\{B_{i}\}_{i\in I_{F'}}$ and $\mathcal{D}_{F}^{E}=\mathcal{D}_{E'}^{F'}$. Furthermore, let $\Delta^{j}_{FF'}(x)=\tau_{j}(x)-\tau_{j}(T_{FF'}(x))$ for $x\in F\cap\tilde{X}(q_{j})_{0}$ then $\Delta^{j}_{FF'}(x)$ is supported in $\bigcup_{i\in I_{F}}B_{i}^{F}$ and
        \begin{equation*}
            \mass(\Delta_{FF'}(x))\leq C(j)\varepsilon.
        \end{equation*}
    \end{enumerate}

    Assume the Inductive Property holds for $j-1$, we sketch how to do the inductive step (obtaining the base case $j=1$ from $G_{0}$ is analog). Let $\mathcal{E}$ be an equivalence class of $j$-faces under $\sim$ and let $E\in\mathcal{E}$. As $R^{q_{j}}F$ is constant on $E$ (because so it is on the top dimensional cells $C_{v}$ of $\tilde{X}$), we know that $\partial G_{j-1}(x)=F(v_{E})$ for every $x\in\partial E\cap\tilde{X}(q_{j-1})_{0}$. Using this, we can interpolate between $\{G(v):v\in V(E)\}$ in the same way as we did between the cycles $\{F(v):v\in V(E)\}$ in the proof of Proposition \ref{Prop delta loc aprox cycles}, just replacing a ``family of cycles'' by a ``family of chains with constant boundary''. Indeed, the only thing we need to use for such interpolation is that for $x,y$ vertices of some $C_{v}$, the difference $G(x)-G(y)$ is a $(k+1)$-cycle of small flat norm, hence there exists $\sigma_{xy}\in\mathcal{I}_{k+2}(M)$ such that $\partial\sigma_{xy}=G(x)-G(y)$ and $\mass(\sigma_{xy})\leq\varepsilon$. We define the corresponding grids $\mathcal{D}_{E}$ for each cell $E$ of $\tilde{X}$ in the same way as done in the proof of Proposition \ref{Prop delta loc aprox cycles} but also requiring that $\mathcal{D}_{E}=\mathcal{D}_{E'}$ if $E\sim E'$ which can be done in the following way. Consider the equivalence relation $\approx$ in $\tilde{X}_{0}$ given by $v\approx v'$ if and only if there exist $E,E'\in\faces(\tilde{X})$ such that $v\in E$, $v'\in E'$, $E\sim E'$ and $T_{EE'}(v)=v'$. Given an equivalence class $\mathcal{V}$ of $\sim$, we fix a grid $\mathcal{D}_{\mathcal{V}}$ compatible with $\{\sigma(w,w'):w,w'\in\bigcup_{v\in\mathcal{V}}\Adj(v)\}$ where $\Adj(v)$ is taken in the simplicial complex $\tilde{X}$ and we set $\mathcal{D}_{v}=\mathcal{D}_{\mathcal{V}}$ if $v\in\mathcal{V}$. Then by construction, $\mathcal{D}_{v}=\mathcal{D}_{v'}$ if $v\approx v'$ and this implies that $\mathcal{D}_{E}=\mathcal{D}_{E'}$ if $E\sim E'$ following Definition \ref{Def grids for any cell} (where grids for higher dimensional cells are constructed as the intersection of the grids of their vertices). We will use the operation of adding
    \begin{equation*}
        \partial[\sigma_{vw}\llcorner D]=[G(v)-G(w)]\llcorner D+\sigma_{vw}\llcorner\partial D
    \end{equation*}
    to change $w^{D}_{E}(x)=w$ by $w^{D}_{E}(x)=v$ and $I_{j}(x)$ by $I_{j}(x)+\sigma_{vw}\llcorner\partial D$. That way we can extend $G_{j}$, $I_{j}$ and $w^{D}_{E}(x)$ to all of $E$ verifying properties (C1) to (C9) for $j$-faces $F\subseteq E$ following the proof of Proposition \ref{Prop delta loc aprox cycles}. Given $E'\in\mathcal{E}$, $E'\neq E$, we extend $\Delta^{j}_{EE'}:E\cap\tilde{X}(q_{j})_{0}$ to be constantly equal to $\Delta^{j}_{EE'}(v_{E})=G_{0}(v_{E})-G_{0}(T_{EE'}(v_{E}))$ on $E\setminus\partial E$. Then by inductive hypothesis, $\mass(\Delta_{EE'}^{j}(x))\leq C(j-1)\varepsilon$ for every $x$. We set
    \begin{equation*}
        G(x)=G(T_{E'E}(x))-\Delta^{j}_{EE'}(T_{EE'}(x))
    \end{equation*}
    for $x\in E'\cap\tilde{X}(q_{j})_{0}$ and define $\{B_{i}\}_{i\in I_{F}}=\{B_{i}\}_{i\in I_{F'}}$ accordingly as in Proposition \ref{Prop filling small families}. Notice that we have
    \begin{equation*}
        G_{j}(x)=\sum_{D\in\mathcal{D}_{E}}G(w^{D}_{E}(y))\llcorner D+I_{j}(y)-\Delta_{EE'}^{j}(y)
    \end{equation*}
    for $y=T_{E'E}(x)$. By construction, $\mathcal{D}_{E}=\mathcal{D}_{E'}$ but $w^{D}_{E}(y)\in V(E)$ and we require $w^{D}_{E'}(y)\in V(E')$. But if we define $w^{D}_{E'}(x)=T_{EE'}(w^{D}_{E}(y))$, using the fact that $\mass(G_{0}(w_{E}^{D}(y))-G_{0}(w_{E'}^{D}(x)))\leq 2\varepsilon$ because $w_{E}^{D}(y)$ and $w_{E'}^{D}(x)$ belong to the center of some cell of $X$, we can rewrite
    \begin{equation*}
        G_{j}(x)=\sum_{D\in\mathcal{D}_{E'}}G(w_{E'}^{D}(x))\llcorner D+I_{j}(x)
    \end{equation*}
    where
    \begin{equation*}
        I_{j}(x)=I_{j}(y)-\Delta_{EE'}^{j}(x)+\sum_{D\in\mathcal{D}_{E}}(G(w^{D}_{E}(y))-G(w^{D}_{E'}(x)))\llcorner D.
    \end{equation*}
    As the second term has mass bounded by $C(j-1)\varepsilon$ and the third one by $|V(E)|\varepsilon=2^{j}\varepsilon$, we obtain the desired upper bound for $\mass(I_{j}(x))$ and also (C5) by observing that
    \begin{equation*}
        \sum_{D\in\mathcal{D}_{E}}\mass(G(w^{D}_{j}(x)\llcorner D)\leq \sum_{D\in\mathcal{D}_{E}}\mass(G(w^{D}_{j}(y)\llcorner D)+\sum_{D\in\mathcal{D}_{E}}\mass((G(w^{D}_{j}(y))-G(w^{D}_{j}(x)))\llcorner D).
    \end{equation*}
    The fact that $I_{j}$ is localized in $\{B_{i}^{F'}\}_{i\in I_{F'}}$ for any $j$-face $F'\subseteq E'$ of $\tilde{X}(q_{j})$ is a consequence of the same property holding for $G_{j}:V(F')\to\mathcal{I}_{k+1}(M)$ and of (C8) being true with $\mathcal{D}^{F'}_{E'}=\mathcal{D}^{F}_{E}$.

\end{proof}

\begin{proof}[Proof of Proposition \ref{Prop delta loc aprox chains}]
    We start by applying Proposition \ref{Prop delta loc aprox cycles} to the family $F:X_{0}\to\mathcal{Z}_{k}(M)$ to obtain an $(N(k,p),\delta)$-localized and $c(k,p)(\frac{L}{\delta})^{n(k,p)}\varepsilon$-fine family $F':X_{0}\to\mathcal{Z}_{k}(M)$ satisfying 
    \begin{equation*}
        \mass(F'(x))\leq\max\{\mass(F(v)):v\in C\cap X(q)_{0}\}+c(k,p)(\frac{L}{\delta})^{n(k,p)}\varepsilon
    \end{equation*}
    if $x\in C\cap X(q)_{0}$. We will extend $G$ to a map $G':X(q)_{0}\to\mathcal{I}_{k+1}(M)$ which will also be $c(k,p)(\frac{L}{\delta})^{n(k,p)}\varepsilon$-fine and will verify $\partial G'(x)=F'(x)$. With that purpose, for each $x\in X(q)_{0}$ we denote by $C_{x}$ the unique cell such that $x\in\interior(C_{x})$ and we choose a vertex $v(x)\in C_{x}$. Notice that $v(x)=x$ if $x\in X_{0}$. Given $C\in\faces_{p}(X)$ and $x,y\in C\cap X(q)_{0}$,  we construct $\tau(F'(x),F'(y))$ to be a $(k+1)$-chain with boundary $F'(x)-F'(y)$ and mass bounded by $c(k,p)(\frac{L}{\delta})^{n(k,p)}\varepsilon$.
    Set $G'(x)=G(v(x))+\tau(F'(v(x)),F'(x))$. It follows that $\partial G'(x)=F'(x)$ and
    \begin{equation*}
        \mass(G'(x))\leq\mass(G(v(x))+c(k,p)(\frac{L}{\delta})^{n(k,p)}\varepsilon.
    \end{equation*}
    Given $x,y\in C$, $v(x),v(y)\in C$ and therefore as $G$ is $\varepsilon$-fine we can see that
    \begin{equation*}
        \mathcal{F}(G'(x),G'(y))\leq\mathcal{F}(G(v(x)),G(v(y)))+\mass(\tau(F'(v(x)),F'(x)))+\mass(\tau(F'(v(y)),F'(y)))
    \end{equation*}
    hence $G'$ is $c(k,p)(\frac{L}{\delta})^{n(k,p)}\varepsilon$-fine. Then $G'$ satisfies the hypothesis of Proposition \ref{Prop filling big families} and applying it we get the desired result.

\end{proof}

\subsection{Approximating continuous families}\label{Subsection approximating continuous families}

We start by proving Theorem \ref{Thm continuous delta approx cycles}. The proof of Theorem \ref{Thm continuous delta approx chains} is analog.

\begin{proof}[Proof of Theorem \ref{Thm continuous delta approx cycles}]
    In the proof we are going to apply Theorem \ref{Thm discrete delta approx cycles}. Let us make the following observation about that result first. To avoid bad notation, we will denote by $\delta'$ and $\varepsilon'$ the small numbers denoted by $\delta$ and $\varepsilon$ respectively in the statement of the theorem. Notice that if we set
    \begin{equation*}
        \delta'=\delta'(\varepsilon')=(C(k,p,M)\sqrt{\varepsilon'})^{\frac{1}{n(k,p)}}
    \end{equation*}
    and $\varepsilon'$ is sufficiently small, then $\delta'<\delta_{0}$ and $\varepsilon'<\varepsilon_{0}$; allowing us to apply the theorem. Moreover, given $\varepsilon_{1}>0$ and $\delta_{1}>0$ we can also make $\delta'<\delta_{1}$ and
    \begin{equation*}
        C(k,p,M)\frac{\varepsilon'}{(\delta')^{n(k,p)}}=\sqrt{\varepsilon'}<\varepsilon_{1}
    \end{equation*}
    provided $\varepsilon'$ is sufficiently small. We will do such choice of $\varepsilon'$ and $\delta'$ in Theorem \ref{Thm discrete delta approx cycles} for certain values of $\delta_{1}>0$ and $\varepsilon_{1}>0$ to be determined below.
    
    Let $F:X\to\mathcal{Z}_{k}(M)$ be a continuous family without concentration of mass and let $\varepsilon,\delta>0$. Let $\varepsilon_{1}=\frac{\varepsilon}{2(C(k,p)+1)}$. Let $\delta_{1}\leq\min\{\frac{\delta}{C(p,k)},1\}$ be such that
    \begin{equation}\label{Eq no concentration of mass}
        \sup_{x\in X}\sup_{p\in M}\{\mass(F(x)\llcorner B(p,\delta_{1}))\}<\frac{\varepsilon_{1}}{2^{p}N(k,p)}.
    \end{equation}
    The previous is possible by the no concentration of mass property of $F$. Define $\delta'>0$ and $\varepsilon'>0$ such that $\delta'<\delta_{1}$ and
    \begin{equation*}
        C(k,p,M)\frac{\varepsilon'}{(\delta')^{n(k,p)}}=\sqrt{\varepsilon'}<\varepsilon_{1}
    \end{equation*}
    as indicated above. Consider a refinement of $X$ (which for simplicity we will still call $X$) for which $F$ is $\varepsilon'$-fine. Apply Theorem \ref{Thm discrete delta approx cycles} to the $\varepsilon'$-fine family $F|_{X_{0}}$. Then we obtain $q\in\mathbb{N}$ and $F':X(q)_{0}\to\mathcal{Z}_{k}(M)$ such that
    \begin{enumerate}
        \item $F'(x)=F(x)$ for every $x\in X_{0}$.
        \item Given a cell $C$ of $X$ and $x\in C$,
        \begin{equation*}
            \mass(F'(x))\leq\max_{v\in V(C)}\{\mass(F(v))\}+\varepsilon_{1}.
        \end{equation*}
        \item $F'$ is $\varepsilon_{1}$-fine in $X$ and $(N(k,p),\delta_{1})$-localized in $X(q)$.
    \end{enumerate}
    Now we want to extend $F'$ to a continuous family with domain $X(q)$. For that purpose, we apply \cite{GL22}[Proposition~2.6] to the discrete $\varepsilon_{1}$-fine and $(N(k,p),\delta_{1})$-localized map $F':X(q)_{0}\to\mathcal{Z}_{k}(M)$, obtaining an extension $F':X(q)\to\mathcal{Z}_{k}(M)$ such that
    \begin{enumerate}
        \item The extension is $C(k,p)\delta_{1}$-localized.
        \item If $C$ is a cell of $X(q)$ and $x\in C$,
        \begin{equation*}
            \mass(F'(x))\leq\max_{v\in V(C)}\{\mass(F'(v))\}+C(k,p)\max_{v\in V(C)}\{\mass(F'(x)\llcorner(\bigcup_{i\in I_{C}}U_{i}^{C}))\}.
        \end{equation*}
    \end{enumerate}
    In the previous $\{U_{i}^{C}\}_{i\in I_{C}}$ denotes the $(N(k,p),\delta_{1})$-admissible family where $F'|_{C}$ is localized. Using the construction of $F'$, we are going to show that 
    \begin{equation*}
        \mass(F'(x)\llcorner U_{C})\leq\varepsilon_{1}
    \end{equation*}
    for every $x\in X(q)_{0}$, where $U_{C}=\bigcup_{i\in I_{C}}U_{i}^{C}$. Indeed, recall that if $x\in\overline{C}$ where $\overline{C}$ is a top cell of $X$, by the proof of Theorem \ref{Thm discrete delta approx cycles}
    \begin{equation*}
        F'(x)=\sum_{D\in\mathcal{D}_{\overline{C}}}F(w_{\overline{C}}^{D}(x))\llcorner D+I_{p}(x)
    \end{equation*}
    and therefore
    \begin{align*}
        \mass(F'(x)\llcorner U_{C}) & \leq\sum_{D\in\mathcal{D}_{\overline{C}}}\mass(F(w_{\overline{C}}^{D}(x))\llcorner(D\cap U_{C}))+\mass(I_{p}(x))\\
        & \leq\sum_{D\in\mathcal{D}_{\overline{C}}}\sum_{v\in V(\overline{C})}\mass(F(v)\llcorner(D\cap U_{C}))+C(k,p,M)\frac{\varepsilon'}{(\delta')^{n(k,p)}}\\
        & \leq\sum_{v\in V(C)}\mass(F(v)\llcorner U_{C})+\varepsilon_{1}\\
        & \leq 2\varepsilon_{1}
    \end{align*}
    using (\ref{Eq no concentration of mass}), the fact that $|V(C)|=2^{p}$ and that $U_{C}$ is the union of at most $N(k,p)$ balls of radius less than $\delta_{1}$. Using that and the upper bounds for $\mass(F'(x))$ for $x\in X(q)_{0}$, we can see that if $x\in\overline{C}$ for a top cell $\overline{C}$ of $X$ then
    \begin{equation*}
        \mass(F'(x))\leq \max_{v\in V(\overline{C})}\{\mass(F(v))\}+(2C(k,p)+1)\varepsilon_{1}.
    \end{equation*}
    From \cite{GL22}[Proposition~2.6] we also know that given a cell $C$ of $X(q)$ and $x,y\in C$,
    \begin{equation*}
        \mathcal{F}(F'(x),F'(y))\leq C(k,p)\delta_{1}\max_{v\in V(C)}\{\mass(F'(x)\llcorner(\bigcup_{i\in I_{C}}U_{i}^{C}))\}\leq 2C(k,p)\varepsilon_{1}.
    \end{equation*}
    Taking $\overline{C}$ a cell of $X$ such that $C\subseteq\overline{C}$, $v\in V(C)$ and $w\in V(\overline{C})$ we can see that
    \begin{align*}
        \mathcal{F}(F'(x),F(x)) & \leq\mathcal{F}(F'(x),F'(v))+\mathcal{F}(F'(v),F(w))+\mathcal{F}(F(w),F(x))\\
        & \leq 2C(k,p)\varepsilon_{1}+\varepsilon_{1}+\varepsilon'\\
        & \leq 2(C(k,p)+1)\varepsilon_{1}.
    \end{align*}
    As $\varepsilon_{1}=\frac{\varepsilon}{2(C(k,p)+1)}$ and $\delta_{1}\leq\frac{\delta}{C(k,p)}$ we obtain the desired result.
\end{proof}

We will need the following proposition in order to prove the Parametric Isoperimetric Inequality.

\begin{proposition}\label{Prop filling small continuous families}
    For each $\varepsilon>0$, there exists $\eta>0$ such that the following holds. Let $F:X^{p}\to\mathcal{Z}_{k}(M)$ be a continuous family such that $\mathcal{F}(F(x))\leq\eta$ for every $x\in X$. Then there exists $\tau:X^{p}\to\mathcal{I}_{k+1}(M)$ such that $\partial\tau=F$ and $\mass(\tau(x))\leq\varepsilon$ for every $x\in X$.
\end{proposition}

In order to prove the proposition, we first prove the following intermediate result.

\begin{lemma}\label{Lemma filling of continuous delta localized}
    For each $\varepsilon>0$, there exists $\eta>0$ such that the following holds. Let $F:X^{p}\to\mathcal{Z}_{k}(M)$ be a continuous family such that $\mathcal{F}(F(x))\leq\eta$ for every $x\in X$. Then given $\alpha>0$, there exist continuous families $F_{1},F_{2}:X^{p}\to\mathcal{Z}_{k}(M)$ and $\tau:X^{p}\to\mathcal{I}_{k+1}(M)$ such that
    \begin{enumerate}
        \item $F=F_{1}+F_{2}$,
        \item $\mathcal{F}(F_{2}(x))\leq\alpha$,
        \item $\partial\tau(x)=F_{1}(x)$,
        \item $\mass(\tau(x))\leq \varepsilon$.

    \end{enumerate}

\end{lemma}

\begin{proof}[Proof of Lemma \ref{Lemma filling of continuous delta localized}]
    Let $\delta>0$ be a small number. Assume without loss of generality that the cubulation of $X$ is fine enough so that $F$ is $\alpha''$-fine for $\alpha''$ sufficiently small (to be chosen in terms of $\alpha$ and $\delta$). By Theorem \ref{Thm filling small families} and Theorem \ref{Thm discrete delta approx cycles}, if $\eta$ is small enough (in terms of $\varepsilon$ and $\delta$), there exist $q>0$ and monotonously $(N(k,p),\delta)$-localized families $F_{1}:X(q)_{0}\to\mathcal{Z}_{k}(M)$ and $\tau:X(q)_{0}\to\mathcal{I}_{k+1}(M)$ such that
    \begin{enumerate}
        \item $F_{1}(x)=F(x)$ for all $x\in X_{0}$,
        \item $F_{1}$ is $\alpha'$-fine in $X$,
        \item $\partial\tau(x)=F_{1}(x)$ for all $x\in X(q)_{0}$,
        \item $\mass(\tau(x))\leq \varepsilon'$
    \end{enumerate}
      for $\varepsilon'$ and $\alpha'$ small to be determined later (in terms of $\alpha$, $\varepsilon$ and $\delta$). Given $C$ a cell of $X(q)$, denote $\{U_{i}^{C}\}_{i\in I_{C}}$ the $(N(k,p),\delta)$-admissible family in which $F_{1}|_{C}$ and $\tau|_{C}$ are localized. Now we want to extend $F_{1}$ and $\tau$ to continuous families $F_{1}:X(q)\to\mathcal{Z}_{k}(M)$ and $\tau:X(q)\to\mathcal{I}_{k+1}(M)$ respectively. We proceed skeleton by skeleton with the following inductive property.
      
      \textbf{Inductive property.} There exist a constant $C(l)$ and continuous extensions $F_{1}:X(q)_{l}\to\mathcal{Z}_{k}(M)$ and $\tau: X(q)_{l}\to\mathcal{I}_{k+1}(M)$ such that 
      \begin{enumerate}[label=(D\arabic*)]
          \item \label{ite : 4.1} $\partial\tau(x)=F_{1}(x)$,
          \item \label{ite : 4.2} $\mass(\tau(x))\leq C(l)\varepsilon'$ for every $x\in X(q)_{l}$,
          \item \label{ite : 4.3} $F_{1}$ is $C(l)\alpha'$-fine in $X(q)$, 
          \item \label{ite : 4.4} $\tau|_{E\cap X(q)_{l}}$ is localized in $\{U_{i}^{E}\}_{i\in I_{E}}$ for every cell $E$ of $X(q)$.
      \end{enumerate}

    The inductive property holds for $l=0$ taking $C(0)=1$. Assume it holds for $l-1$. Let $C$ be an $l$-cell of $X^{p}$ and fix a vertex $v$ of $C$. We know that there exist continuous families $z_{C}:\partial C\to\mathcal{Z}_{k}(M)$ and $\tau_{C}:\partial C\to\mathcal{I}_{k+1}(M)$ supported in $U_{C}=\bigcup_{i\in I_{C}}U_{i}^{C}$ such that
    \begin{align*}
        F_{1}(x) & =F_{1}(v)+z_{C}(x),\\
        \tau(x) & = \tau(v)+\tau_{C}(x),
    \end{align*}
    $\F(z_{C}(x))\leq C(l-1)\alpha'$ and $\mass(\tau_{C}(x))\leq C(l)\varepsilon'$ for every $x\in C$. We extend $z_{C}$ and $\tau_{ C}$ to all of $C$ by contracting radially on each $U_{i}^{C}$. To be precise, we define identify $C$ with with a cone over $\partial C$ and define
    \begin{align*}
        z_{C}(x,t) & =\sum_{i\in I_{C}}(1-t)[z_{C}(x)\llcorner U_{i}]\\
        \tau_{C}(x,t) & =\sum_{i\in I_{C}}(1-t)[\tau_{C}(x)\llcorner U_{i}].
    \end{align*}
    Defining $F_{1}(y)=F_{1}(v)+z_{C}(y)$ and $\tau(y)=\tau(v)+\tau_{C}(y)$ for $y\in C$, we obtain extensions of $F_{1}$ and $\tau$ satisfying $\partial\tau(y)=F_{1}(y)$. In addition, if $y$ corresponds to $(x,t)$ in our previous identification then
    \begin{equation*}
        \mass(\tau(y))\leq\mass(\tau(v))+\mass(\tau_{C}(x,t))\leq\mass(\tau(v))+\mass(\tau_{C}(x))\leq 3C(l-1)\varepsilon'
    \end{equation*}
    and
    \begin{equation*}
        \mathcal{F}(F_{1}(y)-F_{1}(v))=\mathcal{F}(z_{C}(x,t))\leq \mathcal{F}(z_{C}(x))\leq C(l-1)\alpha'.
    \end{equation*}
    This implies that $F_{1}:X(q)_{l}\to\mathcal{Z}_{k}(M)$ is $3C(l-1)\alpha'$-fine in $X(q)$ and that \ref{ite : 4.1} to \ref{ite : 4.4} follow with $C(l)=3C(l-1)$, which completes the inductive step.

    Now, if we choose $\alpha'$ such that $3C(p)\alpha'<\alpha$, $\varepsilon'$ such that $C(p)\varepsilon'=\varepsilon$ and $\eta$ sufficiently small in terms of $\varepsilon$ and $\delta$, we can perform the previous construction and obtain $F_{1}$, $\tau$ and $F_{2}=F-F_{1}$ satisfying
    \begin{equation*}
        \mathcal{F}(F_{2}(x))\leq\mathcal{F}(F_{1}(x)-F_{1}(w))+\mathcal{F}(F_{1}(w))-F_{1}(v))+\mathcal{F}(F(v)-F(x))\leq\alpha
    \end{equation*}
    and the remaining desired properties, which completes the proof.

\end{proof}

\begin{proof}[Proof of Proposition \ref{Prop filling small continuous families}]
    Given $\varepsilon>0$, define $\varepsilon_{i}=\frac{\varepsilon}{2^{i}}$. Let $\eta_{i}>0$ be given by Lemma \ref{Lemma filling of continuous delta localized} for $\varepsilon=\varepsilon_{i}$.  By making them smaller if necessary, we can assume $\lim_{i\to\infty}\eta_{i}=0$. Set $\eta=\eta_{1}$ and fix $F:X^{p}\to\mathcal{Z}_{k}(M)$ with $\mathcal{F}(F(x))\leq\eta$. We inductively define functions $F_{i},G_{i}:X^{p}\to\mathcal{Z}_{k}(M)$ and $\tau_{i}:X^{p}\to\mathcal{I}_{k+1}(M)$ such that
    \begin{enumerate}[label=(E\arabic*)]
        \item \label{ite : 5.1} $G_{i-1}=F_{i}+G_{i}$ for every $i\geq 1$ (we set $G_{0}=F)$.
        \item \label{ite : 5.2} $\F(G_{i}(x))\leq\eta_{i+1}$.
        \item \label{ite : 5.3} $\partial\tau_{i}=F_{i}$.
        \item \label{ite : 5.4} $\mass(\tau_{i}(x))\leq\varepsilon_{i}$.
    \end{enumerate}

    Applying Lemma \ref{Lemma filling of continuous delta localized} to $F=G_{0}$, $\varepsilon=\varepsilon_{1}$, $\eta=\eta_{1}$ and $\alpha=\eta_{2}$, we can define $F_{1}$ and $G_{1}$ satisfying \ref{ite : 5.1} to \ref{ite : 5.4} above. Assuming $F_{i-1},G_{i-1}$ were defined, apply Lemma \ref{Lemma filling of continuous delta localized} to $F=G_{i-1}$, $\varepsilon=\varepsilon_{i}$, $\eta=\eta_{i}$ and $\alpha=\eta_{i+1}$ to obtain $F_{i}$, $G_{i}$ and $\tau_{i}$ verifying the desired properties. Now define
    \begin{equation*}
        \tau(x)=\sum_{i=1}^{\infty}\tau_{i}(x).
    \end{equation*}
    By the Weierstrass $M$-test we can see that $\tau:X\to\mathcal{I}_{k+1}(M)$ is continuous in the flat topology and is the uniform limit of the corresponding partial sums. Therefore
    \begin{align*}
        \partial\tau(x) &=\lim_{N\to\infty}\sum_{i=1}^{N}\partial\tau_{i}(x)\\
        & =\lim_{N\to\infty}\sum_{i=1}^{N}F_{i}(x)\\
        & =\lim_{N\to\infty}\sum_{i=1}^{N}G_{i-1}(x)-G_{i}(x)\\
        & =\lim_{N\to\infty}G_{0}(x)-G_{N}(x)\\
        &=G_{0}(x)\\
        &= F(x)
    \end{align*}
    and in addition $\mass(\tau(x))\leq \varepsilon$ for every $x$, which completes the proof.
\end{proof}

\section{Parametric Isoperimetric Inequality}\label{Section isoperimetric inequality}

\subsection{Parametric isoperimetric inequality in $\mathbb{D}^{n}$}\label{Section Isoperimetric in the disk} 

The goal of this section is to prove a parametric isoperimetric inequality for continuous families of absolute $0$-cycles in the $n$-disk
\begin{equation*}
    \mathbb{D}^{n}=\{(x_{1},...,x_{n})\in \mathbb{R}^{n}:x_{1}^{2}+...+x_{n}^{2}\leq 1\}.
\end{equation*}
Denote $\mathbb{S}^{n-1}=\partial \mathbb{D}^{n}$, $e=(0,0,...,0,-1)$ the south pole and $e'=(0,0,...0,1)$ the north pole. In this section and in the rest of the paper, we consider flat cycles and chains with $\mathbb{Z}_{2}$ coefficients. We will follow the strategy from \cite{GL22} where such inequality was proved in the $2$-disk $\mathbb{D}^{2}$.

\begin{theorem}\label{Isoperimetric inequality 1}
    Let $F:X^{p}\to\mathcal{Z}_{0}(\mathbb{D}^{n})$ be a continuous family and let $L<\frac{\pi}{2}$. Denote $B$ the ball in $\mathbb{S}^{n-1}$ of radius $L$ centered at $e$ and
    \begin{equation*}
        \overline{F}(x)=F(x)\llcorner\interior(\mathbb{D}^{n})+F(x)\llcorner B.
    \end{equation*}
    Then given a small number $r<1$, there exists a continuous $G:X^{p}\to\mathcal{I}_{1}(\mathbb{D}^{n})$ such that
    \begin{enumerate}
        \item $\partial G(x)-F(x)$ is a contractible family supported in $\partial\mathbb{D}^{n}$.
        \item $\mass(\partial G(x))\leq 2\mass(\overline{F}(x))$.
        \item There exists a constant $C=C(n)$ depending only on $n$ such that
        \begin{equation*}
            \mass(G(x))\leq C(\mass(\overline{F}(x))r+r^{-(n-1)}).
        \end{equation*}
        
    \end{enumerate}
\end{theorem}

We can think that the previous theorem provides a short filling $\tilde{G}:X\to\mathcal{I}_{1}(\mathbb{D}^{n},\partial\mathbb{D}^{n})$ of the family $\tilde{F}:X\to\mathcal{Z}_{0}(\mathbb{D}^{n},\partial\mathbb{D}^{n})$ of relative $0$-cycles induced by $F$. Or alternatively, that reduces the problem of finding a short filling of $F$ to that of finding a short filling of the contractible family $F(x)-\partial G(x)$ which is supported in $\partial\mathbb{D}^{n}$; suggesting that proceeding by induction in the dimension may be a good idea for the purpose of finding short fillings of absolute families of $0$-cycles on high dimensional manifolds. Observe that in order to actually have control over $\mass(G(x))$, we need to have a bound on the mass of $F(x)$ restricted to $\interior(\mathbb{D}^{n})\cup B$. Notice that radius $L$ of $B$ can be arbitrarily small and we will take advantage of that in the applications of the theorem. Indeed, we will show in the next section that any $\delta$-localized family $F$ for which $\mass(F(x)\llcorner\interior(\mathbb{D}^{n}))$ is controlled can be slightly perturbed to a new family $F'$ for which $\mass(F'(x)\llcorner(\interior(\mathbb{D})\cup B))$ is controlled, where the radius $L$ of $B$ can be chosen to be any number smaller than $\delta$. We state and demonstrate some lemmas which are necessary for Theorem \ref{Isoperimetric inequality 1} and then we end the section by proving the latter.

\begin{lemma}\label{Deformation Lemma 1}
    Let $S$ be the $1$-skeleton of a unit grid in $\mathbb{R}^{n}$ and let $T$ be the dual $(n-2)$-skeleton. Let $R>0$ and let $B\subseteq S_{R}(0)=\partial B_{R}(0)$ be a closed ball. Let $\{\tilde{T}_{i}\}_{i\in I}$ be the finite collection of $(n-2)$-dimensional affine subspaces of $\mathbb{R}^{n}$ such that $\bigcup_{i\in I}\tilde{T}_{i}\cap B_{R}(0)=T\cap B_{R}(0)$ and denote $T_{i}=\tilde{T}_{i}\cap B_{R}(0)$. Then there exists a point $P\in\mathbb{R}^{n}\setminus B_{R}(0)$ and a real number $\varepsilon>0$ verifying the following:
    \begin{enumerate}
        \item \label{ite : 6.1} Let $C$ be the cone tangent to $S_{R}(0)$ centered at $P$ and let $\Omega$ be the region enclosed by $C$ and $S_{R}(0)$. Then $\partial \Omega\cap S_{R}(0)\subseteq B$. This implies that given a line $L$ through $P$ intersecting $B_{R}(0)$, the point in $L\cap S_{R}(0)$ closest to $P$ is in $B$.
        \item \label{ite : 6.2} Denote $\mathscr{L}$ the set of lines through $P$ which intersect $B_{R}(0)$, which is a compact subset of $\mathbb{RP}^{n-1}$. Then for each $L\in\mathscr{L}$, $L$ intersects $N_{\varepsilon}T_{i}$ for at most ${n \choose 2}$ different values of $i$ and for such values $L\cap N_{\varepsilon}T_{i}$ is contained in a ball $B_{i}^{L}$ of radius smaller than $1$. 
    \end{enumerate}
\end{lemma}

\begin{remark}
    The proof of the previous lemma is based in Larry Guth's construction on bending planes around a skeleton which is explained in the proof of \cite{GuthWidthVolume}[Theorem~1]. In Guth's paper, they consider families of $k$-planes perpendicular to a fixed $(n-k)$-plane $P$ in general position and here we deal with families of lines passing through a generic point $P$. All the arguments in this section extend those of Guth and Liokumovich in \cite{GL22}[Section~5], where they prove the Parametric Isoperimetric Inequality for $0$-cycles in the $2$-disk.
\end{remark}

\begin{proof}[Proof of Lemma \ref{Deformation Lemma 1}]
    We can write
    \begin{equation*}
        \{T_{i}\}_{i\in I}=\bigcup_{j=1}^{n\choose 2}\{T_{i}^{j}\}_{i\in I_{j}}
    \end{equation*}
    where two planes $T_{i}$, $T_{i'}$ are parallel if and only if there exists $1\leq j\leq{n\choose 2}$ such that $i,i'\in I_{j}$ (each class of planes is determined by choosing $n-2$ different elements of the canonical basis $\{e_{1},...e_{n}\}$ of $\mathbb{R}^{n}$). Given $1\leq j\leq n-2$ and $i,i'\in I_{j}$, let $V_{i,i'}^{j}$ be the $(n-1)$-dimensional plane containing $T_{i}^{j}$ and $T_{i'}^{j}$. Denote
    \begin{equation*}
        V=\bigcup_{j=1}^{n\choose 2}\bigcup_{i,i'\in I_{j}}V_{i,i'}^{j}
    \end{equation*}
    being $V\subseteq\mathbb{R}^{n}$ a subset of Lebesgue $n$-measure $0$. As the set of points $P$ satisfying (\ref{ite : 6.1}) has nonempty interior, we can pick such $P$ which is not in $V$. This implies that given a line $L\in\mathscr{L}$ and a natural number $1\leq j\leq{n\choose 2}$, there exists at most one $i\in I_{j}$ such that $L\cap T_{i}\neq\emptyset$ and moreover for such $i$ the intersection $L\cap T_{i}$ consists of a single point $p^{L}_{i}$. Denote $\alpha_{i}^{L}$ the angle of intersection of $L$ with $T_{i}$.

    Given $L\in\mathscr{L}$, denote 
    \begin{equation*}
        I_{L}=\{i\in I:L\cap T_{i}\neq\emptyset\}
    \end{equation*}
    and $\tilde{I}_{L}=I\setminus I_{L}$. Let
    \begin{equation*}
        \varepsilon_{L}=\frac{1}{2}\min\{\dist(L,T_{i}):i\in \tilde{I}_{L}\}
    \end{equation*}
    Let $\delta_{L}>0$ be a small number such that if $\dist(L',L)<\delta_{L}$ in $\mathscr{L}$ then
    \begin{itemize}
        \item $\dist(L',T_{i})>\varepsilon_{L}$ for every $i\in\tilde{I}_{L}$.
        \item Let $i\in I_{L}$. Denote $v(L')$ the direction vector of $L'$ and $T_{i}'$ the plane parallel to $T_{i}$ through the origin. Then the angle $\alpha_{i}^{L'}$ between $v(L')$ and $T'_{i}$ verifies $\sin(\alpha_{i}^{L'})>\frac{1}{2}\sin(\alpha_{i}^{L})$ (if $\Pi_{i}:\mathbb{R}^{n}\to T_{i}'$ is the orthogonal projection onto $T_{i}'$, then $\sin(\alpha_{i}^{L})=\sqrt{1-\Vert\Pi_{i}(v(L))\Vert^{2}}$).
    \end{itemize}

    Observe that $\{B(L,\delta_{L}):L\in\mathscr{L}\}$ is an open cover of $\mathscr{L}$. As $\mathscr{L}$ is compact, we can pick a finite subcover $\{B(L_{k},\delta_{L_{k}}):1\leq k\leq K\}$. Now define $\varepsilon_{1}=\min\{\varepsilon_{L_{k}}:1\leq k\leq K\}$ and $\beta=\min\{\frac{\sin(\alpha_{i}^{L_{k}})}{2}:1\leq k\leq K,i\in I_{L_{k}}\}$ and pick $\varepsilon<\min\{\varepsilon_{1},\frac{\beta}{2}\}$. Let $L\in \mathscr{L}$. Let $1\leq k\leq K$ be such that $L\in B(L_{k},\delta_{L_{k}})$. Then by construction
    \begin{equation*}
        L\cap N_{\varepsilon}T\subseteq\bigcup_{i\in I_{L_{k}}}L\cap N_{\varepsilon}T_{i}
    \end{equation*}
    hence it suffices to show that for each $i\in I_{L_{k}}$ the intersection $L\cap N_{\varepsilon}T_{i}$ is contained in a ball of radius $1$. Now assume $L\cap N_{\varepsilon}T_{i}\neq\emptyset$ and pick $q\in L\cap N_{\varepsilon}T_{i}$. Let $q'$ be the orthogonal projection of $q$ onto $T_{i}$ and let $\overline{T}_{i}=(q-q')+T_{i}$ which is the plane parallel to $T_{i}$ through $q$, being $\{q\}=\overline{T}_{i}\cap L$ because the angle $\alpha_{L}^{i}$ is not zero by hypothesis. Observe that $N_{\varepsilon}T_{i}\subseteq N_{2\varepsilon}\overline{T}_{i}$ so it is enough to show that $L\cap N_{2\varepsilon}\overline{T}_{i}\subseteq B_{1}(q)$. But given a point $x\in L$,
    \begin{equation*}
        \dist(x,\overline{T}_{i})=\sin(\alpha_{i}^{L})\dist(x,q)>\frac{\sin(\alpha^{L_{k}}_{i})}{2}\dist(x,q)\geq\beta\dist(x,q)
    \end{equation*}
    and hence if $x\in L\cap N_{2\varepsilon}\overline{T}_{i}$ we must have $\dist(x,q)<\frac{2\varepsilon}{\beta}<1$ which completes the proof.

\end{proof}

We will also need \cite{GuthWidthVolume}[Lemma~2.1] which we state below. We use the notation introduced in Lemma \ref{Deformation Lemma 1}.

\begin{lemma}\label{Deformation Lemma 2}
    Given $\varepsilon>0$, there is a piecewise linear map $\Psi:\mathbb{R}^{n}\to\mathbb{R}^{n}$ with the following properties. $\Psi$ is linear on each simplex of a certain triangulation of $\mathbb{R}^{n}$. Each top-dimensional simplex of this triangulation is labeled good or bad. For each good simplex $\Delta$, $\Psi(\Delta)$ lies in $S$. Each bad simplex lies in $N_{\varepsilon}T$. There exists a constant $C(n)$ depending only on $n$ such that the triangulation and the map obey the following bounds
    \begin{enumerate}
        \item The number of simplices of our triangulation meeting any unit ball is bounded by $C(n)$.
        \item The displacement $|\Psi(x)-x|$ is bounded by $C(n)$.
        \item The diameter of each simplex is bounded by $C(n)$.
    \end{enumerate}
\end{lemma}

\begin{lemma}\label{Lemma Psi}
    Let $\tau_{1},...,\tau_{k}$ be line segments in $B_{R}(0)$ with $\tau_{i}$ contained in a certain line $L_{i}$ through $P$, $P$ chosen as in Lemma \ref{Deformation Lemma 1}. Let $\tau=\tau_{1}+...+\tau_{k}$. Then there exists a constant $D(n)>0$ such that
    \begin{equation*}
        \mass(\Psi(\tau))\leq D(n)(k+R^{n}).
    \end{equation*}
\end{lemma}

\begin{proof}
    Let $1\leq j\leq n$. First, we will show that the number of bad simplices $\Delta$ such that $\support(\tau_{j})\cap \Delta\neq\emptyset$ is bounded by ${n\choose 2} C(n)$. As all bad simplices are contained in $N_{\varepsilon}T$ we see that for such $\Delta$, $\support(\tau_{j})\cap \Delta$ is contained in $L_{j}\cap N_{\varepsilon}T$. By Lemma \ref{Deformation Lemma 1}, $L_{j}\cap N_{\varepsilon}T$ is contained in the union of at most ${n\choose 2}$ balls $\{B_{i}^{L_{j}}\}$ of radius smaller than $1$. By property (1) in Lemma \ref{Deformation Lemma 2}, the collection $\{B_{i}^{j}\}_{i}$ intersects at most ${n\choose 2}C(n)$ different bad simplices.
    Next we observe that if $\Delta$ is any simplex, by (2) and (3) in Lemma \ref{Deformation Lemma 2} we have $\diameter(\Psi(\Delta))\leq 3C(n)$. Therefore as $\Psi(\tau_{j}\llcorner\Delta)$ is a linear segment,
    \begin{equation*}
        \length(\Psi(\tau_{j}\llcorner \Delta))\leq 3C(n)
    \end{equation*}
    which together with the previous implies that
    \begin{equation}
        \sum_{\Delta \text{ bad simplex}}\length(\Psi(\tau_{j})\llcorner\Delta)\leq 3{n\choose 2}C(n)^{2}
    \end{equation}
    On the other hand, if $\Delta$ is a good simplex then $\Psi(\tau_{j}\llcorner\Delta)$ is a $1$-chain supported in $S$. Adding over $j$ and using that the mass of $S$ is bounded by $E(n)R^{n}$ for a certain constant $E(n)$ depending only on $n$, we can see that
    \begin{equation*}
        \length(\Psi(\tau))\leq 3{n\choose 2}C(n)^{2}k+E(n)R^{n}
    \end{equation*}
    which yields the desired result by setting $D(n)=\max\{3{n\choose 2}C(n)^{2},E(n)\}$.
\end{proof}

\begin{proof}[Proof of Theorem \ref{Isoperimetric inequality 1}]
    Consider a grid of width $r$ in the unit disk $\mathbb{D}^{n}$ with $1$-skeleton $S$ and dual $(n-2)$-skeleton $T$. Set $R=\frac{1}{r}$. Identify $\mathbb{D}^{n}$ with $B_{R}(0)$ under the map $x\mapsto Rx$ which maps the skeleta $S$ and $T$ of $\mathbb{D}^{n}$ to the corresponding $S$ and $T$ of $B_{R}(0)$ defined in Lemma \ref{Deformation Lemma 1}. Use that lemma to pick a suitable point $P$ outside $\mathbb{D}^{n}$ and construct a continuous family of $1$-chains $G_{1}:X\to\mathcal{I}_{1}(\mathbb{D}^{n})$ in the following way. Given two points $z,w\in \mathbb{D}^{n}$ let $\Xi(z,w)$ be the segment connecting them. Given a $0$-chain $x$ in $\mathbb{D}^{n}$ and a Lipschitz map $T:\mathbb{D}^{n}\to \mathbb{D}^{n}$ define
\begin{equation*}
    \Xi(x,T(x))=\sum_{z\in\support(x)}\Xi(z,T(z)).
\end{equation*}

    Given a point $z\in \mathbb{D}^{n}$, denote $L_{z}$ the line through $P$ and $z$ and let $H(z)$ be the point in $L_{z}\cap\mathbb{S}^{n-1}$ which is furthest from $P$. Define $\tilde{H}:\mathcal{I}_{0}(\mathbb{D}^{n})\to\mathcal{I}_{1}(\mathbb{D}^{n})$ as
    \begin{equation*}
        \tilde{H}(z)=\Xi(z,H(z)).
    \end{equation*}
    Observe that $\tilde{H}:\mathcal{I}_{0}(\mathbb{D}^{n})\to\mathcal{I}_{1}(\mathbb{D}^{n})$ is continuous in the flat topology. We define
    \begin{equation*}
        G_{1}(x)=\tilde{H}(F(x))=\tilde H(\overline{F}(x)).
    \end{equation*}
    Notice that $\partial G_{1}(x)=H(F(x))-F(x)=H(\overline{F}(x))-\overline{F}(x)$. Now we want to deform $G_{1}$ to decrease its mass, using a bend and cancel argument as in \cite{GuthMinMax} and \cite{GL22}. We use the previously constructed map $\Psi$. By Lemma \ref{Lemma Psi},
    \begin{equation}\label{Eq mass bound 1}
        \mass(\Psi(G_{1}(x)))\leq D(n)(\mass(\overline{F}(x))r+r^{-(n-1)}).
    \end{equation}
    Let $G_{2}:X\to\mathcal{I}_{1}(\mathbb{D}^{n})$ be given by
    \begin{equation*}
        G_{2}(x)=\Xi(F(x),\Psi(F(x)))=\Xi(\overline{F}(x),\Psi(\overline{F}(x))).
    \end{equation*}
   By property (2) of Lemma \ref{Deformation Lemma 2}, each of those segments has length at most $C(n)r$ hence
    \begin{equation}\label{Eq mass bound 2}
        \mass(G_{2}(x))\leq C(n)\mass(\overline{F}(x))r.
    \end{equation}
    We define
    \begin{equation*}
        G(x)=\Psi(G_{1}(x))+G_{2}(x).
    \end{equation*}
    Then the mass bound for $G$ follows from (\ref{Eq mass bound 1}) and (\ref{Eq mass bound 2}). We can also check that $\partial G(x)-F(x)=\Psi(H(F(x)))$ and $\partial G(x)-\overline{F}(x)=\Psi(H(\overline{F}(x)))$. Therefore $\partial G(x)-F(x)$ is a contractible family in $\partial\mathbb{D}^{n}$. In addition,
    \begin{equation*}
        \mass(\partial G(x))\leq\mass(\overline{F}(x))+\mass(\Psi(H(\overline{F}(x))))\leq 2\mass(\overline{F}(x))
    \end{equation*}
    for every $x\in X$.
\end{proof}

\subsection{Two important propositions to go from disks to manifolds}\label{Section Two important propositions}

To prove the Weyl law, we will need to deal with families of relative $0$-cycles supported in a compact manifold with boundary $(M,\partial M)$ rather than with absolute families supported in $\mathbb{D}^{n}$ as in Theorem \ref{Isoperimetric inequality 1}. To extend our result to such families, the strategy will be to pick a triangulation $\{Q_{j}\}_{j}$ of the underlying compact manifold $M$ where our family $F$ is supported and find a short filling of $F(x)\llcorner Q_{j}$ for each $j$. But $F(x)\llcorner Q_{j}$ is not a continuous family of absolute cycles in $Q_{j}$ but only a continuous family of relative cycles in $(Q_{j},\partial Q_{j})$ (even if $\partial M=\emptyset$). Assuming that the original family $F$ is contractible, it is possible to represent $F\llcorner Q_{j}$ by a continuous family $F_{j}$ of absolute cycles. However, $F_{j}$ may have very large mass in $\partial Q_{j}$, even in arbitrary tiny balls in $\partial Q_{j}$ as the ones used in the proof of Theorem \ref{Isoperimetric inequality 1}; preventing from its direct application to find a short filling of $F_{j}$. The next proposition provides a way to slightly perturb a family $F:X^{p}\to\mathcal{Z}_{0}(\mathbb{D}^{n})$ into a new family $F'(x)$ for which there exists a small ball $B\subseteq\mathbb{S}^{n-1}$ such that $\mass(F'(x)\llcorner B)$ is controlled for all $x\in X$. We first state and prove the proposition for $n=3$  and then we discuss the case $n\geq 4$. The analog proposition for $n=2$ was proved by Guth and Liokumovich in \cite{GL22}[Lemma~5.2] and it was used to obtain the Parametric Isoperimetric Inequality for $0$-cycles on surfaces. 

\begin{proposition}\label{Avoid ball D^3}
    There exists $\delta_{0}=\delta_{0}(p)>0$ such that the following is true. Let $F:X^{p}\to\mathcal{Z}_{0}(\mathbb{D}^{3})$ be a $\delta$-localized family, $\delta<\delta_{0}$. Let $0<L<\delta$ be a very small real number and let $B$ be the open ball in $\mathbb{S}^{2}$ centered at $e$ of radius $L$. There exists a continuous family $F':X^{p}\to\mathcal{Z}_{0}(\mathbb{D}^{3})$ with the following properties: 
    \begin{enumerate}
        \item \label{ite : 7.1} $F'$ is $L+(p+2)\delta$-localized and therefore $(L+(p+2)\delta)\max_{x\in X_{0}}\mass(F'(x))$-fine.
        \item \label{ite : 7.2} $F'(x)\llcorner(\mathbb{D}^{3}\setminus B)=F(x)\llcorner (\mathbb{D}^{3}\setminus B)$ for every $x\in X^{p}_{0}$.
        \item \label{ite : 7.3} If $\tilde{B}=B_{L+(p+2)\delta}(e)$ then $F'(x)\llcorner(\mathbb{D}^{3}\setminus\tilde{B})=F(x)\llcorner(\mathbb{D}^{3}\setminus\tilde{B})$ for every $x\in X^{p}$.
        \item \label{ite : 7.4} Given a cell $C$, denote
        \begin{equation*}
            \mass_{C}=\max_{x\in C\cap X_{0}}\{\mass(F(x)\llcorner\interior(\mathbb{D}^{3}))\}.
        \end{equation*}
        Then if $x\in C$ and $\dim(C)=k$ it holds
        \begin{equation*}
        \mass(F'(x)\llcorner\interior(\mathbb{D}^{3}))+\mass(F'(x)\llcorner B)\leq \mass_{C}+k+1.
        \end{equation*}
        Therefore if
        \begin{equation*}
            \mass_{0}=\max_{x\in X_{0}}\{\mass(F(x)\llcorner\interior(\mathbb{D}^{3}))\}
        \end{equation*}
        it holds
        \begin{equation*}
            \mass(F'(x)\llcorner\interior(\mathbb{D}^{3}))+\mass(F'(x)\llcorner B)\leq \mass_{0}+p+1
        \end{equation*}
        for every $x\in X$.

    \end{enumerate}
\end{proposition}

Before proving the proposition, we introduce the following definition.

\begin{definition}
    Take polar coordinates $(r,\theta)\in[0,\infty)\times[0,2\pi)$ in the $xy$-plane and consider the induced cylindrical coordinates $(r,\theta,z)$ and spherical coordinates $(r,\theta,\varphi)$ in $\mathbb{R}^{3}$ ($z\in\mathbb{R}$, $\varphi\in[-\frac{\pi}{2},\frac{\pi}{2})$). Given $\theta\in[0,2\pi)$, let $s_{\theta}$ denote the geodesic segment in $\mathbb{S}^{2}$ contained in the meridian passing by $e=(0,0,-1)$ and the point $(\cos\theta,\sin\theta,0)$ given by the set of points there with
    \begin{equation*}
        -\frac{\pi}{2}+\frac{L}{2}\leq\varphi\leq 0
    \end{equation*}
    or equivalently
    \begin{equation*}
        -\cos(\frac{L}{2})\leq z\leq 0.
    \end{equation*}
    Given $0<\eta<1$ denote
\begin{equation*}
    T_{\theta,\eta}=\{(1-t)q:q\in s_{\theta},0\leq t\leq\eta\}.
\end{equation*}
\end{definition}

\begin{proof}[Proof of Proposition \ref{Avoid ball D^3}]
    We define $F'$ inductively skeleton by skeleton. At the $0$-skeleton, we proceed as in \cite{GL22}. Namely, given $x\in X_{0}$ we define
    \begin{equation*}
        F'(x)=F(x)-F(x)\llcorner B+e(x)
    \end{equation*}
    where $e(x)=e$ (the $0$-chain associated to the point $e=(0,0,-1)\in \mathbb{S}^{2}$) if $\mass(F(x)\llcorner B)$ is odd and $e(x)=0$ otherwise. Notice that (\ref{ite : 7.1}) to (\ref{ite : 7.3}) hold and also that
    \begin{equation*}
        \mass(F'(x)\llcorner\interior(\mathbb{D}^{3}))+\mass(F'(x)\llcorner B)\leq\mass(F(x)\llcorner\interior(\mathbb{D}^{3}))+1.
    \end{equation*}
    To motivate our construction to extend $F'$ to the higher-dimensional skeleta, let us first explain what is done in \cite{GL22} for the case $n=2$. We first introduce the following notation. Given $r\in(0,\frac{\pi}{2})$, let
    \begin{align*}
        B_{r} & =\{q\in\mathbb{S}^{1}:\dist_{\mathbb{S}^{1}}(q,e)\leq r\}\\
        C_{r} & =\{(y,z)\in\mathbb{D}^{2}:z=-\cos(r)\}\\
        U_{r} &=\{(y,z)\in\mathbb{D}^{2}:z\leq -\cos(r)\}.
    \end{align*}
    Observe that $\partial U_{r}=B_{r}\cup C_{r}$ and $B=B_{L}$. Denote $U=U_{L}$. Suppose $F'$ was defined in $\partial C$ for a certain cell $C$ of $X$ verifying that $F'|_{\partial C}$ is localized in the $\delta$-admissible family $\{U_{i}^{C}\}_{i\in I_{C}}$. Our goal is to contract $F'|_{\partial C}$ in order to extend $F'$ to all of $C$. Consider a radius $r_{C}\in[L,L+\delta]$ such that $F'(x)\llcorner C_{r_{C}}=0$ for every $x\in\partial C$ (such $r_{C}$ exists by the $\delta$-localization property). Denote $U_{C}=U_{r_{C}}$ and $B_{C}=B_{r_{C}}$. The previous allows us to write $F'|_{\partial C}$ as the sum of two continuous families
    \begin{equation*}
        F'(x)=F'(x)\llcorner(\mathbb{D}^{2}\setminus U_{C})+F_{C}(x)
    \end{equation*}
    where $F_{C}(x)=F(x)\llcorner U_{C}$. The challenging part is how to contract $F_{C}$ without introducing too many points in $B$. In particular, just applying a radial homotopy on $U_{C}$ towards $e$ will not work, because there might be arbitrarily many points in $B_{C}\setminus B$ which are pushed to $B$ by this homotopy. Guth and Liokumovich proposed the following solution. Let $R_{C}:U_{C}\times[0,1]\to U_{C}$ be the map which is the identity on $U$ and retracts the region $U_{C}\setminus U$ radially onto the line $C_{L}$, mapping points in $\mathbb{S}^{1}$ to $\mathbb{S}^{1}$. For $x\in\partial C$, consider
    \begin{equation*}
        F_{C}(x,t)=R_{C}(F_{C}(x),t)
    \end{equation*}
    Then it is clear that $F_{C}(x,t)\llcorner B=F_{C}(x)\llcorner B$ for every $t\in[0,1)$, and at $t=1$ because of $\mathbb{Z}_{2}$ cancellation we introduce at most two new points to $B$ (the only two points in $\mathbb{S}^{1}$ at distance $L$ from $e$). If we then homotop $F_{C}(x,1)$ to the constant family $F_{C}(x,2)=e$ by contracting $U$ radially towards $e$, we get an extension with the desired properties. 
    
    We can perform the same constructions on $\mathbb{D}^{3}$. In this case, the definitions are
     \begin{align*}
        B_{r} & =\{q\in\mathbb{S}^{2}:\dist_{\mathbb{S}^{2}}(q,e)\leq r\}\\
        C_{r} & =\{(x,y,z)\in\mathbb{D}^{3}:z=-\cos(r)\}\\
        U_{r} &=\{(x,y,z)\in\mathbb{D}^{3}:z\leq -\cos(r)\}.
    \end{align*}
    Nevertheless, in this case $F_{C}(x,1)$ can potentially have arbitrarily many points in $\partial B$, because it is a circle (and not a set of two points as before). Therefore for $t\in(1,2)$, $F_{C}(x,t)\llcorner B$ can have a very large number of points. To fix the previous, we will perform a different homotopy to make sure that $F_{C}(x,1)\llcorner \partial B$ is supported in a discrete set which is independent of $x$, instead of in a $1$-dimensional set as we obtained with the previous construction. This will be obtained by first homotoping $F_{C}=F_{C}^{0}$ to a map $F^{1}_{C}$ for which $F^{1}_{C}\llcorner (B_{C}\setminus B)$ is supported in a finite collection of meridians and then applying the homotopy $R_{C}$ to $F^{1}_{C}$. To illustrate this, suppose that we want $F^{1}_{C}\llcorner (B_{C}\setminus B)$ to be supported in exactly two meridians $s_{\theta_{1}}$ and $s_{\theta_{2}}$, for example with $\theta_{1}=0$ and $\theta_{2}=\pi$. If we had two other meridians $s_{\theta_{1}'}$ and $s_{\theta_{2}'}$ (say $\theta_{1}'=\frac{\pi}{2}$ and $\theta_{2}'=\frac{3\pi}{2})$ such that $\support(F_{C}\llcorner (B_{C}\setminus B))$ does not intersect $s_{\theta_{1}'}\cup s_{\theta_{2}'}$, then as $(B_{C}\setminus B)\setminus(s_{\theta_{1}'}\cup s_{\theta_{2}'})$ retracts onto $s_{\theta_{1}}\cup s_{\theta_{2}}$, we could obtain the desired $F^{1}_{C}$ by applying a suitable extension (to $U_{E}$) of the previous retraction to $F_{C}$. We will extend $F'$ inductively such that for each cell $C$ of $X$, $F_{C}$ avoids a certain collection of meridians which depend on $C$, like $s_{\theta'_{1}}$ and $s_{\theta'_{2}}$ in the previous example. The latter property will allow us to extend $F'$ to the skeleton which is $1$-dimension higher.
    
    We start by explaining the construction of the meridians and then we proceed to build $F'$ by induction. By Lemma \ref{Lemma monotonously localized}, we can assume that $F$ is monotonously $\delta$-localized. Given $C$ a top dimensional cell of $X$, let $\{B_{i}^{C}\}_{i\in I_{C}}$ be the $\delta$-admissible family where $F|_{C}$ is localized.
    Let
   \begin{equation*}
       \mathcal{S}=\bigcup_{x\in X_{0}}\support(F'(x))
   \end{equation*}
   being $\mathcal{S}$ a finite set of points.
   
   Given $C\in\faces_{p}(X)$, let $r_{C}\in[L+\delta,L+2\delta]$ be a real number such that $C_{r_{C}}\cap B_{i}^{C}=\emptyset$ for every $i\in I_{C}$ and $C_{r_{C}}\cap\mathcal{S}=\emptyset$ (such $r_{C}$ exists because $\sum_{i\in I_{C}}\diameter(B_{i}^{C})<\delta$). Denote
   \begin{equation*}
       J_{C}=\{i\in I_{C}:B_{i}^{C}\cap U_{r_{C}}=\emptyset\}.
   \end{equation*}
   We claim that there exists $\theta_{C}\in[0,2\pi)$ such that $s_{\theta_{C}}\cap B_{i}^{C}=\emptyset$ for every $i\in J_{C}$ and $s_{\theta_{C}}\cap \mathcal{S}=\emptyset$. Given $i\in J_{C}$, denote
   \begin{equation*}
       \tilde{I}_{i}=\{\theta\in[0,2\pi):s_{\theta}\cap B_{i}^{C}\neq\emptyset\}
   \end{equation*}
   which is an interval in $S^{1}=[0,2\pi]/0\sim 2\pi$. Let
   \begin{equation*}
       I_{i}=\{(\sin(r_{C})\cos\theta,\sin(r_{C})\sin\theta,-\cos(r_{C})):\theta\in\tilde{I}_{i}\}
   \end{equation*}
   be the interval in the circle $\partial B_{r_{C}}\subseteq\mathbb{S}^{2}$ consisting of those points with $\theta\in\tilde{I}_{i}$ in cylindrical coordinates. Notice that
   \begin{equation*}
       \diameter(I_{i})\leq\diameter(B_{i}^{C})
   \end{equation*}
   and therefore
   \begin{equation*}
       \sum_{i\in J_{C}}\diameter(I_{i})\leq\sum_{i\in I_{C}}\diameter(B_{i}^{C})<\delta
   \end{equation*}
   and as the length of the circle $\partial B_{r_{C}}$ is $2\pi\sin(r_{C})$ we can see
   \begin{equation*}
       \sum_{i\in J_{C}}\diameter(I_{i})<\delta<2\pi\sin(L+\delta)\leq 2\pi\sin(r_{C})
   \end{equation*}
   provided $\delta$ is sufficiently small. Therefore, there must exist a point in $\partial B_{r_{C}}\setminus(\mathcal{S}'\cup\bigcup_{i\in J_{C}}I_{i})$ which yields the existence of the desired $s_{\theta_{C}}$ (here $\mathcal{S'}$ is the set of points in $\partial B_{r_{C}}$ whose $\theta$-coordinate coincides with that of some point in $\mathcal{S}$). By compactness, there exists $\eta_{C}>0$ such that the $\eta_{C}$-tubular neighborhood of $s_{\theta_{C}}$ in $\mathbb{D}^{3}$ does not intersect $\mathcal{S}\cup\bigcup_{i\in J_{C}}B_{i}^{C}$. We set $T_{C}=T_{\theta_{C},\eta_{C}}$. Then by construction, we have an $L+3\delta$-admissible family
   \begin{equation*}
       \{B_{i}^{C}\}_{i\in I^{*}_{C}}=\{U_{r_{C}}\}\cup\{B_{i}^{C}:i\in J_{C}\}
   \end{equation*}
   in which $F'|_{C\cap X_{0}}$ is localized. This is because if $x,y\in C\cap X_{0}$
   \begin{equation*}
       F'(x)-F'(y)=F(x)-F(y)-(F(x)-F(y))\llcorner B+e(x)-e(y),
   \end{equation*}
   $F(x)-F(y)$ is supported in $\bigcup_{i\in I_{C}}B_{i}^{C}\subseteq\bigcup_{i\in I^{*}_{C}}B_{i}^{C}$ and
   \begin{equation*}
       \support((F(x)-F(y))\llcorner B)+e(x)-e(y))\subseteq B\subseteq U_{r_{C}}\subseteq\bigcup_{i\in I^{*}_{C}}B_{i}^{C}.
   \end{equation*}
   In addition, by construction $F'(x)\llcorner T_{C}=0$ for every $x\in X_{0}\cap C$ and $B_{i}^{C}\cap T_{C}=\emptyset$ for every $i\in J_{C}$. We will use these constructions to extend $F'$ inductively, skeleton by skeleton.

   \textbf{Inductive property.} Given $1\leq k\leq p$ and a $k$-cell $E$ of $X$, there exists a real number $r_{E}$ such that
   \begin{equation*}
       L+k\delta\leq r_{E}\leq L+(k+1)\delta
   \end{equation*}
   and an $L+(k+2)\delta$-admissible family of the form
   \begin{equation*}
       \{B_{i}^{E}\}_{i\in I_{E}^{*}}=\{U_{r_{E}}\}\cup\{B_{i}^{E}:i\in J_{E}\}
   \end{equation*}
   where
   \begin{equation*}
       J_{E}=\{i\in I_{E}:B_{i}^{E}\cap U_{r_{E}}=\emptyset\}.
   \end{equation*}
   such that 
   \begin{enumerate}[label=(F\arabic*)]
    \item \label{ite : 8.1} $F'|_{E}$ is localized in $\{B_{i}^{E}\}_{i\in I_{E}^{*}}$.
    \item \label{ite : 8.2} $F'(x)\llcorner (\mathbb{D}^{3}\setminus U_{r_{E}})=F(x)\llcorner(\mathbb{D}^{3}\setminus U_{r_{E}})$. 
    \item \label{ite : 8.3} $F'(x)\llcorner T_{C}=0$ for every $x\in E$ and every $C\in\faces_{p}(X)$ containing $E$.
    \item \label{ite : 8.4} $B_{i}^{E}\cap T_{C}=\emptyset$ for every $i\in J_{E}$ and every $C\in\faces_{p}(X)$ containing $E$.
    \item \label{ite : 8.5} For every $x\in E$,
    \begin{equation*}
        \mass(F'(x)\llcorner\interior(\mathbb{D}^{3}))+\mass(F'(x)\llcorner B)\leq M_{E}+k+1
    \end{equation*}
    where
    \begin{equation*}
        M_{E}=\max_{x\in E\cap X_{0}}\{\mass(F(x)\llcorner\interior(\mathbb{D}^{3}))\}.
    \end{equation*}
    
   \end{enumerate}

    We show by induction that $F'$ can be extended verifying the previous inductive property.

    \textbf{Base case.} Let $k=1$. Let $E\in\faces_{1}(X)$. Define
    \begin{equation*}
        r_{E}=\max\{r_{C}:C\in\faces_{p}(X),C\supseteq E\}.
    \end{equation*}
    
    Denote $U_{E}=U_{r_{E}}$ and $C_{E}=C_{r_{E}}$. Observe that $L+\delta\leq r_{E}\leq L+2\delta$ and that if $C_{0}\in\faces_{p}(X)$ is such that $r_{E}=r_{C_{0}}$, then $C_{E}$ does not intersect $B_{i}^{C_{0}}$ for every $i\in I_{C_{0}}$ and hence neither intersects $B_{i}^{E}$ for any $i\in I_{E}$ (by monotonously $\delta$-localization). In addition, $C_{E}\cap\mathcal{S}=\emptyset$. Define $J_{E}$ and $\{B_{i}^{E}\}_{i\in I_{E}^{*}}$ as in the inductive property, so that 
    \begin{equation*}
        \bigcup_{i\in I_{E}}B_{i}^{E}\subseteq\bigcup_{i\in I_{E}^{*}}B_{i}^{E}.
    \end{equation*}
    Let $x$ and $y$ be the vertices of $E$. Identify $E$ with the interval $[0,1]$ where $x\cong 0$ and $y\cong 1$ and define
    \begin{equation*} F'(t)=F(t)\llcorner(\mathbb{D}^{3}\setminus U_{E})+F_{E}(t)
    \end{equation*}
    where we define $F_{E}$ as follows. As $\mass(F'(x)\llcorner U_{E})$ and $\mass(F'(y)\llcorner U_{E})$ have the same parity, we can homotop both $F'(x)\llcorner U_{E}$ and $F'(y)\llcorner U_{E}$ to $e(x)$ defined as $0$ if $\mass(F'(x)\llcorner U_{E})$ is even and $e$ otherwise. To be precise, for $t\in[0,\frac{1}{2}]$ we construct $F_{E}(t)$ by moving one point $q\in\support(F'(x)\llcorner U_{E})$ towards $e$ at a time. The trajectories $\gamma_{q}$ of such points are curves in $\mathbb{D}^{3}$ which do not intersect
    \begin{equation*}
    \bigcup_{\substack{C\in\faces_{p}(X) \\ C\supseteq E}}T_{C} 
    \end{equation*} and which are supported in $\mathbb{S}^{2}=\partial\mathbb{D}^{3}$ for $q\in\mathbb{S}^{2}$. For $t=[\frac{1}{2},1]$ we proceed in the same way but starting from $F_{E}(1)=F'(y)\llcorner U_{E}$. By construction, it holds 
        \begin{equation*}
            \mass(F'(t)\llcorner\interior(\mathbb{D}^{3}))+\mass(F'(t)\llcorner B)\leq M_{E}+2.
        \end{equation*}
        This yields \ref{ite : 8.1} to \ref{ite : 8.5}.

    
    \textbf{Inductive step.} Let $E$ be a $(k+1)$-face of $X$. Choose $r_{E}\in [L+(k+1)\delta,L+(k+2)\delta]$ so that $C_{r_{E}}\cap B_{i}^{E}=\emptyset$ and $C_{r_{E}}\cap\mathcal{S}=\emptyset$ for every $i\in I_{E}$. Define $J_{E}$ and $\{B_{i}^{E}\}_{i\in I_{E}^{*}}$ as in the inductive property. Denote $U_{E}=U_{r_{E}}$, $B_{E}=B_{r_{E}}$ and $C_{E}=C_{r_{E}}$. Notice that $F'|_{\partial E}$ is localized in $\{B_{i}^{E}\}_{i\in I_{E}^{*}}$ because of the monotonously $\delta$-localization of $F$ and the fact that
    \begin{equation*}
        r_{E}\geq L+(k+1)\delta\geq r_{E'}
    \end{equation*}
    for every $k$-face $E'$ of $E$. The same holds for $F|_{E}$ and together with the fact that $C_{E}\cap\mathcal{S}=\emptyset$ yields 
    \begin{equation*}
        F(x)\llcorner C_{E}=0
    \end{equation*}
    for every $x\in E$. In particular, $F(x)\llcorner (\mathbb{D}^{3}\setminus U_{E})$ is a continuous function on $E$, and the same holds for $F'(x)\llcorner U_{E}$ on $\partial E$. We will extend $F'|_{\partial E}$ to $E$ so that properties \ref{ite : 8.1} to \ref{ite : 8.5} hold. We will identify $E$ with the cone over $\partial E$ and we will use coordinates $(x,t)\in\partial E\times[0,1]$ where $(x,0)$ corresponds to $x$ and $(x,1)$ to the center of $E$ for each $x\in\partial E$. We define
    \begin{equation*}
        F'(x,t)=F(x,t)\llcorner(\mathbb{D}^{3}\setminus U_{E})+F_{E}(x,t)
    \end{equation*}
    with $F_{E}$ supported in $U_{E}$ and $F_{E}(x,0)=F'(x)\llcorner U_{E}$ for all $x\in\partial E$. The previous will imply that $F'(x,t)$ extends $F'|_{\partial E}$. Thus, our goal is to homotop $F_{E}^{0}:\partial E\to\mathcal{Z}_{0}(U_{E})$ given by $F_{E}^{0}(x)=F'(x)\llcorner U_{E}$ to a constant map while keeping properties \ref{ite : 8.1} to \ref{ite : 8.5}.

    Denote $S^{1}_{E}=\partial B_{E}$. As the number of points in $F_{E}^{0}(x)\llcorner (B_{E}\setminus B)$ is a priori unbounded for $x\in \partial E$, as explained before we can not just apply a homotopy $H_{E}:U_{E}\times[0,1]\to U_{E}$ contracting $U_{E}$ to the point $e$ to the family $F_{E}^{0}$ because that could introduce too many points to $B\cup\interior(\mathbb{D}^{3})$. Instead, we will contract $F_{E}^{0}$ in three steps.
    \begin{enumerate} 
        \item \label{ite : 9.1} The first one will homotop it to a family $F_{E}^{1}:\partial E\to\mathcal{Z}_{0}(\mathbb{D}^{3})$ supported in $\overline{U}$ ($U=U_{L}$) such that $F_{E}^{1}(x)\llcorner\partial B$ is supported in a discrete set $\{y_{1},...,y_{N}\}$ for every $x\in\partial E$ (the discrete set is independent of $x\in \partial E$). The mentioned homotopy $F_{E}^{0}:\partial E\times[0,1]\to\mathcal{Z}_{0}(B_{E})$ will keep the property that
        \begin{equation*}
            \mass(F_{E}^{0}(x,t)\llcorner\interior(\mathbb{D}^{3}))+\mass(F_{E}^{0}(x,t)\llcorner\interior(B))\leq M_{E}+k+1
        \end{equation*}
        where $F_{E}^{0}(x,0)=F_{E}^{0}(x)$ and $F_{E}^{0}(x,1)=F_{E}^{1}(x)$ for $x\in\partial E$.
        \item \label{ite : 9.2} The second one will homotop $F_{E}^{1}$ to a family $F_{E}^{2}$ also supported in $\overline{U}$ such that $F_{E}^{2}(x)\llcorner\partial B$ is supported in a set $\{q\}$ of $1$ element. The corresponding homotopy $F_{E}^{1}:\partial E\times[0,1]\to\overline{U}$ will verify
        \begin{equation*}
            \mass(F_{E}^{1}(x,t)\llcorner\interior(\mathbb{D}^{3}))+\mass(F_{E}^{1}(x,t)\llcorner\interior(B))\leq M_{E}+k+2.
        \end{equation*}
        for every $t\in[0,1]$ and in addition
        \begin{equation*}
            \mass(F_{E}^{2}(x)\llcorner\interior(\mathbb{D}^{3}))+\mass(F_{E}^{2}(x)\llcorner\interior(B))\leq M_{E}+k+1
        \end{equation*}
        where $F^{1}_{E}(x,0)=F^{1}_{E}(x)$ and $F^{1}_{E}(x,1)=F^{2}_{E}(x)$.
        \item \label{ite : 9.3} The third homotopy will just contract $F_{E}^{2}$ radially towards $e$ in $\overline{U}$ (mapping points of $\mathbb{S}^{2}$ into $\mathbb{S}^{2}$), increasing the mass bound from $M_{E}+k+1$ to at most $M_{E}+k+2$.
    \end{enumerate}

    We proceed to describe how to construct the first two homotopies, being them induced by deformation retractions of subspaces of $U_{E}$. For (\ref{ite : 9.1}), let $T_{1},...,T_{N}$ be an enumeration of $\{T_{C}:C\in\faces_{p}(X),C\supseteq E\}$ according to an increasing ordering of the corresponding $\theta_{C}$. Let $T'_{1},...,T'_{N}$ be intermediate $T_{\theta,\eta}$ with $\theta'_{i}\in (\theta_{i},\theta_{i+1})$ modulo $2\pi$ and $\eta\leq\eta_{i}$ for every $1\leq i\leq N$. By inductive hypothesis, $F_{E}^{0}(x)=F'(x)\llcorner U_{E}$ is supported in $U_{E}^{*}:=U_{E}\setminus\cup_{i=1}^{N}T_{i}$ for $x\in\partial E$. 
    Consider a homotopy $H_{E}^{1}:U_{E}^{*}\times[0,1]\to U_{E}^{*}$ which is the identity in $U_{\frac{L}{2}}$ and acts as follows in $U_{r_{E}}\setminus U_{\frac{L}{2}}$. We can write
    \begin{equation*}
        U_{r_{E}}\setminus U_{\frac{L}{2}}=\bigcup_{r\in[\frac{L}{2},r_{E}]}C_{r}
    \end{equation*}
    as a disjoint union (each $C_{r}$ corresponds to a $z$-slice of $U_{r_{E}}\setminus U_{\frac{L}{2}}$). For $r\in[L,r_{E}]$, the homotopy $H_{E}^{1}$ restricted to $C_{r}$ retracts $C_{r}\setminus\cup_{i=1}^{N}T_{i}$ (which is a circle minus $N$ straight short segments emanating from boundary points) to the smaller concentric circle $C_{r}'$ of radius $\radius(C_{r})-\eta$ union the $N$ segments $T'_{i}\cap C_{r}$ (each of them of length $\eta$). The homotopy is chosen so that points in $\partial C_{r}$ remain in $\partial C_{r}$ and at the end they are mapped to $\partial C_{r}\cap\cup_{i=1}^{N}T'_{i}$ (which is a finite set of points). For $r\in[\frac{L}{2},L]$, we interpolate between the $H_{E}^{1}$ previously defined and the identity. Observe that the points in $H_{E}^{1}(F_{E}^{0}(x),1)\llcorner (B_{E}\setminus B)$ lie in the $N$ geodesic segments $T_{i}'\cap\mathbb{S}^{2}$. Then our next step is to retract $U_{E}$ onto $U$ radially towards $e$. Hence, we obtain a family $F_{E}^{1}$ which is supported in $\overline{U}$ with the property that $F_{E}^{1}\llcorner\partial B$ is supported in $\{y_{1},...,y_{N}\}$ where $y_{i}=T_{i}'\cap\partial B$. A more rigorous description of these homotopies will be provided when the result is extended to $n\geq 4$.

    To obtain $F_{E}^{2}(x,t)$, we need to apply $N-1$ homotopies $\overline{H}_{E}^{1},...,\overline{H}_{E}^{N-1}$ to $F_{E}^{1}$ so that if $\overline{F}_{E}^{i}$ is the map obtained after applying the  first $i$, $\overline{F}_{E}^{i}(x)\llcorner\partial B$ is supported in $\{y_{i+1},...,y_{N}\}$. We describe $\overline{H}_{E}^{i}$ in the circle $C_{L}$. We consider a curve $\gamma$ connecting $y_{i}$ with $y_{i+1}$ in $C_{L}$ without touching any of the $T_{i}$. We take a tiny ``tubular neighborhood'' $N_{\gamma}$ of $\gamma$ in $C_{L}$ which also avoids the $T_{i}$ and only intersects $\{y_{1},...,y_{N}\}$ at $\{y_{i},y_{i+1}\}$. It is possible to retract $C_{L}$ onto $C_{L}\setminus N_{\gamma}$ so that the trajectory of $y_{i}$ is given by $\gamma$. Then we can extend this to a homotopy $\overline{H}_{E}^{i}$ in $\overline{U}$ which is not the identity only in a tiny neighborhood $\tilde{N}_{\gamma}\supseteq N_{\gamma}$ of $\gamma$ in $\overline{U}$ (and in particular, away from the $T_{i}$). We can think of the deformation induced by $\overline{H}_{E}^{i}$ as that of digging a narrow tunnel in $U$ along the curve $\gamma$ starting at the point $y_{i}$ and ending at $y_{i+1}$. By applying to the family $F_{2}^{E}$ the composition of the $\overline{H}_{E}^{i}$ we obtain a map with the properties in (\ref{ite : 9.2}) above.
   
\end{proof}

Now our goal is to extend Proposition \ref{Avoid ball D^3} to $\mathbb{D}^{n}$, $n\geq 4$. We will need the following two preliminary lemmas.

\begin{lemma}\label{Lemma skeleton}
    Let $\mathbb{S}^{n-1}\subseteq\mathbb{R}^{n}$ be the unit sphere. Let $P^{d}\subseteq \mathbb{S}^{n-1}$, $1\leq d\leq n-1$, be an embedded $d$-dimensional polyhedron in $\mathbb{S}^{n-1}$ whose cells are geodesically convex in $\mathbb{S}^{n-1}$. Suppose that a certain hyperplane $\Pi\subseteq\mathbb{R}^{n}$ intersects $P$. Then $\Pi$ intersects the $k$-skeleton of $P$ for every $1\leq k\leq d$.
\end{lemma}

\begin{proof}
    We proceed by induction in $d$. If $d=1$ it is automatic. Suppose it is true for $d-1$ and consider a $d$-dimensional polyhedron $P\subseteq\mathbb{S}^{n-1}$ and a hyperplane $\Pi\subseteq\mathbb{R}^{n}$ which intersects $P$. We want to show first that $\Pi\cap P_{d-1}\neq\emptyset$. Suppose the contrary. Pick $Q$ a top dimensional cell of $P$ such that $\Pi\cap Q\neq\emptyset$ and a point $q\in \Pi\cap Q$. It must be $q\in\interior(Q)$ as $\partial Q\subseteq P_{d-1}$. Denote $H^{+}, H^{-}$ the two open half spaces of $\mathbb{R}^{n}$ determined by $\Pi$. We have
    \begin{equation*}
        \partial Q=\partial Q\cap H^{+}\cup\partial Q\cap H^{-}
    \end{equation*}
    as a disjoint union of open sets, so by connectedness it must be $\partial Q\subseteq H^{+}$ or $\partial C\subseteq H^{-}$. Assume the first is true without loss of generality. Then as $H^{+}\cap \mathbb{S}^{n-1}$ is geodesically convex and contains the vertices of the cell $Q$, it must contain all of $Q$. But this is absurd because $q\notin H^{+}$ (as it is in $\Pi$). Therefore, $\Pi\cap P_{d-1}\neq\emptyset$ and by inductive hypothesis $\Pi$ intersects $P_{k}$ for every $1\leq k\leq d-1$ as desired.
\end{proof}

\begin{lemma}\label{Lemma delta hyperplane}
    For each $n\in\mathbb{N}$, there exists $\delta_{n}>0$ such that the following holds. Given a finite collection $\{B_{i}\}_{i\in I}$ of balls in $\mathbb{S}^{n-1}\subseteq\mathbb{R}^{n}$ with
    \begin{equation*}
        \sum_{i\in I}\diameter(B_{i})<\delta_{n}
    \end{equation*}
    there exists a plane $\Pi$ through the origin such that $\Pi\cap B_{i}=\emptyset$ for every $i\in I$.
\end{lemma}

\begin{proof}
    For each $i\in I$, let $Q_{i}$ be an embedded convex $(n-1)$-simplex in $\mathbb{S}^{n-1}$ containing $B_{i}$ whose diameter is at most $d(n)\diameter(B_{i})$ for a certain dimensional constant $d(n)$. Let $Q_{(1)}=\bigcup_{i\in I}(Q_{i})_{1}$ be the union of the $1$-skeletons of the $Q_{i}$. Observe that if a hyperplane $\Pi\subseteq\mathbb{R}^{n}$ intersects $\bigcup_{i\in I}B_{i}$ then it intersects $Q_{i}$ for some $i\in I$ and hence by Lemma \ref{Lemma delta hyperplane} it also cuts its $1$-skeleton, which implies that it intersects $Q_{(1)}$. Therefore, it suffices to find a hyperplane $\Pi\subseteq\mathbb{R}^{n}$ which does not intersect $Q_{(1)}$. Denote $n(\Pi,Q_{(1)})$ the number of points in $\Pi\cap Q_{(1)}$. By Crofton's formula, there is a universal constant $C(n)>0$ such that if $\mu$ denotes the rotation-invariant metric on the space $\mathcal{P}$ of planes through the origin then
    \begin{equation*}
        \int_{\mathcal{P}}n(\Pi,Q_{(1)})d\mu(\Pi)=C(n)\length(Q_{(1)}).
    \end{equation*}
    Therefore, if every hyperplane $\Pi\in\mathcal{P}$ intersected $Q_{(1)}$, as $n(\Pi,Q_{(1)})$ is integer valued we would have
    \begin{equation*}
        \mu(\mathcal{P})=\int_{\mathcal{P}}1d\mu(\Pi)\leq\int_{\mathcal{P}}n(\Pi,Q_{(1)})d\mu(\Pi)=C(n)\length(Q_{(1)})
    \end{equation*}
    and hence $\length(Q_{(1)})\geq\frac{\mu(\mathcal{P})}{C(n)}$. But
    \begin{equation*}
        \length(Q_{(1)})=\sum_{i\in I}\length((Q_{i})_{1})\leq \sum_{i\in I}e(n)\diameter(Q_{i})\leq e(n)d(n)\sum_{i\in I}\diameter(B_{i})
    \end{equation*}
    where $e(n)= {n\choose 2}$ denotes the number of $1$-faces of an $n$-dimensional simplex. Thus taking $\delta_{n}=\frac{\mu(\mathcal{P})}{C(n)d(n)e(n)}$ we get the desired result. 
\end{proof}

\begin{proposition}\label{Avoid ball D^n}
    There exists $\delta_{0}=\delta_{0}(n,p)>0$ such that the following is true. Let $F:X^{p}\to\mathcal{Z}_{0}(\mathbb{D}^{n})$ be a $\delta$-localized family, $\delta<\delta_{0}$. Let $0<L<\delta$ be a very small real number and let $B$ be the open ball in $\mathbb{S}^{n-1}$ centered at $e$ of radius $L$. There exists a continuous family $F':X^{p}\to\mathcal{Z}_{0}(\mathbb{D}^{n})$ with the following properties:
    \begin{enumerate}
        \item $F'$ is $c(n,p)\delta$-localized and therefore $c(n,p)\delta\max_{x\in X_{0}}\mass(F'(x))$-fine for some constant $c(n,p)$ depending only on $n$ and $p$.
        \item $F'(x)\llcorner(\mathbb{D}^{n}\setminus B)=F(x)\llcorner (\mathbb{D}^{n}\setminus B)$ for every $x\in X^{p}_{0}$.
        \item If $\tilde{B}=B_{c(n,p)\delta}(e)$ then $F'(x)\llcorner(\mathbb{D}^{n}\setminus\tilde{B})=F(x)\llcorner(\mathbb{D}^{n}\setminus\tilde{B})$ for every $x\in X^{p}$.
        \item Given a cell $C$, denote
        \begin{equation*}
            \mass_{C}=\max_{x\in C\cap X_{0}}\{\mass(F(x)\llcorner\interior(\mathbb{D}^{n}))\}.
        \end{equation*}
        Then if $x\in C$ and $\dim(C)=k$ it holds
        \begin{equation*}
        \mass(F'(x)\llcorner\interior(\mathbb{D}^{n}))+\mass(F'(x)\llcorner B)\leq \mass_{C}+k+1.
        \end{equation*}
        Therefore if
        \begin{equation*}
            \mass_{0}=\max_{x\in X_{0}}\{\mass(F(x)\llcorner\interior(\mathbb{D}^{n}))\}
        \end{equation*}
        it holds
        \begin{equation*}
            \mass(F'(x)\llcorner\interior(\mathbb{D}^{n}))+\mass(F'(x)\llcorner B)\leq \mass_{0}+p+1
        \end{equation*}
        for every $x\in X$.

    \end{enumerate}
\end{proposition}

We will need the following definitions.

\begin{definition}
    Let $\mathbb{R}^{n}_{-}=\{x\in\mathbb{R}^{n}:x_{n}\leq 0\}$ and $e=(0,0,...,0,-1)\in \mathbb{D}^{n}$.
\end{definition}

\begin{definition}
    Given $r\in[0,\frac{\pi}{2})$ and $\varepsilon\in(0,\frac{1}{2})$, denote
    \begin{align*}
        B_{r} & =B^{\mathbb{S}^{n-1}}_{r}(e),\\
        U_{r} & =\{(1-s)q:q\in B_{r},s\in[0,r]\},\\
        S^{n-2}_{r} & =\{x\in\mathbb{S}^{n-1}:x_{n}  =-\cos(r)\},\\
        S^{n-1}_{r} &= \{(1-s)q:q\in S^{n-2}_{r},s\in[0,r]\}\cup\{(1-r)q:q\in B_{r}\},\\
        A^{n-1}_{r} &=\{x\in\mathbb{S}^{n-1}:-\cos(r)\leq x_{n}\leq 0\},\\
        A^{n}_{r,\varepsilon} & =\{(1-s)q:q\in A^{n-1}_{r}, s\in[0,\varepsilon]\}
    \end{align*}
    where $B_{r}^{\mathbb{S}^{n-1}}(e)$ denotes the ball in $\mathbb{S}^{n-1}$ of radius $r$ centered at $e$.
\end{definition}

\begin{remark}
    Observe that
    \begin{align*}
        \partial B_{r} & =S^{n-2}_{r},\\
        \partial U_{r} & = S^{n-1}_{r}\cup B_{r},\\
        A^{n-1}_{r} & =A^{n}_{r,\varepsilon}\cap\mathbb{S}^{n-1}.
    \end{align*}
    In addition, $B=B_{L}$ and we denote $U=U_{L}$.
\end{remark}

\begin{remark}
    For every $0<r<\frac{\pi}{2}$ the map $\phi:\mathbb{S}^{n-2}\times[-\cos(r),0]\to A^{n-1}_{r}$ given by
    \begin{equation*}
        \phi(u,t)=(\sqrt{1-t^{2}}u,t)
    \end{equation*}
    is a diffeomorphism. Furthermore, for each $\varepsilon\in(0,\frac{1}{2})$ we can extend the previous to a map $\phi:\mathbb{S}^{n-2}\times[-\cos(r),0]\times[0,\varepsilon]\to A^{n}_{r,\varepsilon}$ (regarding $\mathbb{S}^{n-2}\times[-\cos(r),0]\subseteq\mathbb{S}^{n-2}\times[-\cos(r),0]\times[0,\varepsilon]$ via the embedding $(u,t)\mapsto(u,t,0)$) in the following way
    \begin{equation*}
        \phi(u,t,s)=(1-s)\phi(u,t)=(1-s)(\sqrt{1-t^{2}}u,t).
    \end{equation*}
    Hence we can use $(u,t,s)\in\mathbb{S}^{n-2}\times[-\cos(r),0]\times[0,\varepsilon]$ as a coordinate system for $A^{n}_{r,\varepsilon}$. Notice that if $q=\phi(u,t,s)$ the coordinate $s$ measures the distance from $q$ to $\mathbb{S}^{n-1}$ and if $s=0$, $t$ is the $x_{n}$-coordinate of $q$ and $u$ represents its location in the parallel $\mathbb{S}^{n-1}\cap\{x_{n}=t\}$.
\end{remark}

\begin{definition}
    Given a hyperplane $\Pi\subseteq\mathbb{R}^{n-1}$ through the origin and a number $\varepsilon\in(0,\frac{1}{2})$, let
    \begin{equation*}
        s_{\Pi}=\{\phi(u,t)\in A^{n-1}_{\frac{L}{2}}:u\in\Pi\}
    \end{equation*}
    and
    \begin{equation*}
        T_{\Pi,\varepsilon}=\{\phi(u,t,s)\in A^{n}_{\frac{L}{2},\varepsilon}:u\in\Pi\}
    \end{equation*}
    being $s_{\Pi}=T_{\Pi,\varepsilon}\cap\mathbb{S}^{n-1}$. 
\end{definition}

\begin{remark}
    Notice that when $n=3$, $s_{\Pi}$ is determined by two antipodal points in $\mathbb{S}^{1}$ and hence can be associated with a certain angle $\theta\in[0,\pi)$; being a generalization of the $s_{\theta}$ defined before (which corresponded to a certain angle $\theta\in[0,2\pi)$). This also yields an analogy between $T_{\Pi,\varepsilon}$ and $T_{\theta,\varepsilon}$.
\end{remark}

\begin{remark}\label{Rk Pi}
    Observe that $s_{\Pi}\subseteq(\Pi\times\mathbb{R})\cap\mathbb{S}^{n-1}$ and  $T_{\Pi,\varepsilon}\subseteq(\Pi\times\mathbb{R})\cap\mathbb{D}^{n}$, moreover
    \begin{align*}
        s_{\Pi} & =\{x\in(\Pi\times\mathbb{R})\cap\mathbb{S}^{n-1}:x_{n}\in[-\cos(\frac{L}{2}),0]\}\\
        T_{\Pi,\varepsilon} & =\{x=\phi(u,t,s)\in(\Pi\times\mathbb{R})\cap\mathbb{D}^{n}:t\in[-\cos(\frac{L}{2}),0],s\in[0,\varepsilon]\}\}.
    \end{align*}
\end{remark}

\begin{proof}[Proof of Proposition \ref{Avoid ball D^n}]
The strategy of the proof will be the same as for Proposition \ref{Avoid ball D^3}. We start by constructing $F'|_{X_{0}}$ in the same way as for $n=3$, then we define a certain $T_{C}$ for each top dimensional cell $C$ of $X$ and then we extend $F'$ skeleton to skeleton inductively with an inductive hypothesis similar to that stated before.

Given $C\in\faces_{p}(X)$, we choose $r_{C}\in [L+K\delta,L+(K+1)\delta]$ ($K=K(n)$ to be determined later) such that $S^{n-1}_{r_{C}}\cap B_{i}^{C}=\emptyset$ for every $i\in I_{C}$ and $S^{n-1}_{r_{C}}\cap\mathcal{S}=\emptyset$ (being $\{B_{i}^{C}\}_{i\in I_{C}}$ the $\delta$-admissible family associated to $C$ and $\mathcal{S}=\bigcup_{x\in X_{0}}\support(F'(x))$). We denote
\begin{equation*}
    J_{C}=\{i\in I_{C}:B^{C}_{i}\cap U_{r_{C}}=\emptyset\}.
\end{equation*}
We claim that if we choose $K$ suitably, there exists a hyperplane $\Pi\subseteq\mathbb{R}^{n-1}$ such that $s_{\Pi}\cap B_{i}^{C}=\emptyset$ for every $i\in J_{C}$ and also $s_{\Pi}\cap\mathcal{S}=\emptyset$.  Consider the map $\proj:A^{n-1}_{r_{C}}\to S^{n-2}_{r_{C}}$ given in $\phi$-coordinates by $ \proj(u,t)=(u,-\cos(r_{C}))$. Notice that $\proj$ is $1$-Lipschitz and that $B_{i}^{C}\cap\mathbb{R}^{n}_{-}\subseteq S^{n-1}_{r_{C}}$ for every $i\in J_{C}$. Therefore, if $\tilde{B}_{i}^{C}=\proj(B_{i}^{C}\cap\mathbb{R}^{n}_{-})$,
$\{\tilde{B}_{i}^{C}:\in J_{C}\}$ is a family of balls in $S^{n-2}_{r_{C}}$ with sum of radius less than $\delta$. As $S^{n-2}_{r_{C}}$ is an $(n-2)$-sphere of radius $\sin(r_{C})$, by Proposition \ref{Lemma delta hyperplane}, if
\begin{equation*}
    \frac{\delta}{\sin(r_{C})}<\delta_{n-2}
\end{equation*}
then there will exist a hyperplane $\Pi_{C}$ in $\mathbb{R}^{n-1}$ such that $\tilde{\Pi}_{C}\cap\tilde{B}_{i}^{C}=\emptyset$ for every $i\in J_{C}$, where $\tilde{\Pi}_{C}=\{\phi(u,-\cos(r_{C})):u\in\Pi_{C}\}$. But as $L+K\delta\leq r_{C}\leq L+(K+1)\delta$,
\begin{equation}\label{Eq delta localized}
    \frac{\delta}{\sin(r_{C})}\leq\frac{\delta}{\sin(L+K\delta)}<\frac{\delta}{\sin(K\delta)}=\frac{1}{K}\frac{K\delta}{\sin(K\delta)}.
\end{equation}
Define $K=K(n)=\frac{2}{\delta_{n-2}}$. Using that $\lim_{x\to 0}\frac{x}{\sin(x)}=1$, let $\eta_{0}>0$ be such that $0<x<\eta_{0}$ implies $0<\frac{x}{\sin(x)}<2$. Set $\overline{\delta}_{n}=\frac{\eta_{0}}{K}$. Hence if $0<\delta<\overline{\delta}_{n}$, $0<K\delta<\eta_{0}$ and by (\ref{Eq delta localized})
\begin{equation*}
    \frac{\delta}{\sin(r_{C})}<\frac{2}{K}=\delta_{n-2}.
\end{equation*}
Therefore if our original family is $\delta$-localized for $\delta\leq\overline{\delta}_{n}$, we can conclude that for each $C\in\faces_{p}(X)$ there exists a hyperplane $\Pi_{C}$ in $\mathbb{R}^{n-1}$ for which $s_{C}=s_{\Pi_{C}}$ does not intersect any $B_{i}^{C}$ for $i\in J_{C}$ or $\mathcal{S}$. By compactness, we can choose $\varepsilon_{C}>0$ sufficiently small so that $T_{\Pi_{C},\varepsilon_{C}}\cap B_{i}^{C}=\emptyset$ for every $i\in J_{C}$ and $T_{\Pi_{C},\varepsilon_{C}}\cap\mathcal{S}=\emptyset$. Taking
\begin{equation*}
    \varepsilon=\min\{\varepsilon_{C}:C\in\faces_{p}(X)\}
\end{equation*}
and defining $T_{C}=T_{\Pi_{C},\varepsilon}$ we can see that $T_{C}\cap B_{i}^{C}=\emptyset$ for every $i\in J_{C}$ and $T_{C}\cap\mathcal{S}=\emptyset$.

\textbf{Inductive property.} Given $1\leq k\leq p$ and a $k$-cell $E$ of $X$, there exists a real number $r_{E}$ such that
   \begin{equation*}
       L+(K+k)\delta\leq r_{E}\leq L+(K+k+1)\delta
   \end{equation*}
   and an $L+(K+k+2)\delta$-admissible family of the form
   \begin{equation*}
       \{B_{i}^{E}\}_{i\in I_{E}^{*}}=\{U_{r_{E}}\}\cup\{B_{i}^{E}:i\in J_{E}\}
   \end{equation*}
   where
   \begin{equation*}
       J_{E}=\{i\in I_{E}:B_{i}^{E}\cap U_{r_{E}}=\emptyset\}.
   \end{equation*}
   such that
   \begin{enumerate}[label=(G\arabic*)]
    \item \label{ite : 10.1} $F'|_{E}$ is localized in $\{B_{i}^{E}\}_{i\in I_{E}^{*}}$.
    \item \label{ite : 10.2} $F'(x)\llcorner(\mathbb{D}^{n}\setminus U_{r_{E}})=F(x)\llcorner(\mathbb{D}^{n}\setminus U_{r_{E}})$ for every $x\in E$.
    \item \label{ite : 10.3} $F'(x)\llcorner T_{C}=0$ for every $x\in E$ and every $C\in\faces_{p}(X)$ containing $E$.
    \item \label{ite : 10.4} $B_{i}^{E}\cap T_{C}=\emptyset$ for every $i\in J_{E}$ and every $C\in\faces_{p}(X)$ containing $E$.
    \item \label{ite : 10.5} For every $x\in E$
    \begin{equation*}
        \mass(F'(x)\llcorner\interior(\mathbb{D}^{n}))+\mass(F'(x)\llcorner B)\leq M_{E}+k+1
    \end{equation*}
    where
    \begin{equation*}
        M_{E}=\max_{x\in E\cap X_{0}}\{\mass(F(x)\llcorner\interior(\mathbb{D}^{n}))\}.
    \end{equation*}
    
   \end{enumerate}

    We show by induction that $F'$ can be extended verifying the previous inductive property.

    \textbf{Base case.} Let $k=1$ and let $E$ be a $1$-face of $X$. Same as for $n=3$, we set
    \begin{equation*}
        r_{E}=\max\{r_{C}:C\in\faces_{p}(X),C\supseteq E\}.
    \end{equation*}
    It follows $L+K\delta\leq r_{E}\leq L+(K+1)\delta$ and $S^{n-1}_{r_{E}}\cap B^{E}_{i}=\emptyset$ for all $i\in I_{E}$. Define $J_{E}$ and $\{B^{E}_{i}:i\in I_{E}^{*}\}$ as in the inductive property. Let $x$ and $y$ be the vertices of $E$. Identify $E$ with the interval $[0,1]$ where $x\cong 0$ and $y\cong 1$ and set
    \begin{equation*} F'(t)=F(t)\llcorner(\mathbb{D}^{n}\setminus U_{r_{E}})+F_{E}(t)
    \end{equation*}
    where we define $F_{E}$ exactly in the same way as for $n=3$. 
    

    \textbf{Inductive step.} Let $E$ be a $(k+1)$-face of $X$. Choose $r_{E}\in [L+(K+k)\delta,L+(K+k+1)\delta]$ so that $S^{n-1}_{r_{E}}\cap B_{i}^{E}=\emptyset$ for every $i\in I_{E}$ and $S^{n-1}_{r_{E}}\cap\mathcal{S}=\emptyset$. Define $J_{E}$ and $\{B_{i}^{E}\}_{i\in I_{E}^{*}}$ as in the inductive property. Denote
    \begin{align*}
        U_{E} & =U_{r_{E}},\\
        B_{E} & = B_{r_{E}}.
    \end{align*}
    We proceed as for $n=3$. The main difference is how to homotop the continuous family $F_{E}^{0}:\partial E\to\mathcal{Z}_{0}(U_{E})$ defined as $F_{E}^{0}(x)=F'(x)\llcorner U_{E}$ to a family $F_{E}^{1}$ for which $F_{E}^{1}\llcorner\partial B$ is supported in a discrete set which is independent of $x\in\partial E$ (recall that $B=B_{L}$ so $\partial B=S^{n-2}_{L}$). 

    Let us introduce some notation. Denote
    \begin{equation*}
        U_{E}^{*}=U_{E}\setminus\bigcup_{\substack{C\in\faces_{p}(X) \\ C\supseteq E}}T_{C}
    \end{equation*}
    and
    \begin{equation*}
        B_{E}^{*}=B_{E}\setminus\bigcup_{\substack{C\in\faces_{p}(X) \\ C\supseteq E}}s_{C}
    \end{equation*}
    (recall that $B_{E}=U_{E}\cap \mathbb{S}^{n-1}$). Define also
    \begin{equation*}
        \mathbb{S}^{n-2}_{E}=\mathbb{S}^{n-2}\setminus\bigcup_{\substack{C\in\faces_{p}(X) \\ C\supseteq E}}\Pi_{C}.
    \end{equation*}
    Each connected component of $\mathbb{S}^{n-2}_{E}$ is contractible. Let us pick points $y_{1},...,y_{N}\in\mathbb{S}^{n-2}_{E}$ one for each connected component and define a Lipschitz homotopy $R_{E}:\mathbb{S}^{n-2}_{E}\times[0,1]\to\mathbb{S}_{E}$ such that $R_{E}(x,0)=x$ and $R_{E}(x,1)=y_{j}$ if $x$ belongs to the same connected component as $y_{j}$. This yields a homotopy $H_{E}:U_{E}^{*}\times[0,1]\to U_{E}^{*}$ which is the identity outside of $A^{n}_{\frac{L}{2},\varepsilon}$ and inside it has the following expression in coordinates $(u,t,s)\in\mathbb{S}^{n-1}\times[-\cos(\frac{L}{2}),-\cos(r_{E})]\times[0,\varepsilon]$:

    \[
    H_{E}((u,t,s),t')=\begin{cases}
        (R_{E}(u,(1-\frac{s}{\varepsilon})t'),t,s) & \text{if} -\cos(L)\leq t\leq -\cos(r_{E})\\
        (R_{E}(u,(1-\frac{s}{\varepsilon})t'h(t)),t,s) & \text{if} -\cos(\frac{L}{2})\leq t\leq -\cos(L)
    \end{cases}
    \]
    where $h(t)$ is the linear function such that $h(-\cos(\frac{L}{2}))=0$ and $h(-\cos(L))=1$. One can check that $H_{E}$ is continuous because for $s=\varepsilon$ and $t=-\cos(\frac{L}{2})$ it holds $H_{E}((u,t,s),t')=(u,t,s)$. In addition, $\Image(H_{E})\subseteq U_{E}^{*}$ as the $u$-coordinate stays in $\mathbb{S}^{n-2}_{E}$, $H_{E}$ maps $(U_{E}\cap\interior(\mathbb{D}^{n}))\cup B$ into itself and the same holds for  $B_{E}^{*}\setminus B$. But moreover, we can see that
    \begin{equation*}
        H_{E}((B_{E}^{*}\setminus B)\times\{1\})=\{\phi(y_{j},t,0):1\leq j\leq N,-\cos(L)\leq t\leq -\cos(r_{E})\}
    \end{equation*}
    Let us define $F_{E}^{1}(x)$ to be the map obtained from $H_{E}(F_{E}^{0}(x),1)$ by applying the retraction $U_{E}^{*}\to\overline{U}$ given in coordinates by $(u,t,s)\mapsto (u,g(t),s)$ with
    \[
    g(t)=\begin{cases}
        t & \text{if } t\leq -\cos(L)\\
        -\cos(L) & \text{otherwise}.
    \end{cases}
    \]
    The underlying homotopy $F_{E}^{0}(x,t)$ from $F_{E}^{0}$ to $F_{E}^{1}$ verifies properties \ref{ite : 10.1} to \ref{ite : 10.4} above with 
    \begin{equation*}
    \mass(F_{E}^{0}(x,t)\llcorner\interior(\mathbb{D}^{n}))+\mass(F_{E}^{0}(x,t)\llcorner B)\leq M_{E}+k+1
    \end{equation*}
    for every $x\in\partial E$ (recall that $k=\dim(E)-1$ so this is stronger than \ref{ite : 10.5}). Additionally, $F_{E}^{1}(x)\llcorner\partial B$ is supported in $\{\phi(y_{j},-\cos(L),0):1\leq j\leq N\}$ for every $x\in\partial E$.

    Next we want to homotop $F_{E}^{1}$ into a family $F_{E}^{2}$ with the additional property that $F_{E}^{1}(x)\llcorner\partial B$ is supported in a fixed point for every $x\in\partial E$ (which will be chosen to be $\phi(y_{N},-\cos(L),0)$). We proceed as for $n=3$, defining $N-1$ homotopies $\overline{H}_{E}^{1},...,\overline{H}_{E}^{N-1}$ in $\overline{U}$ such that if $\overline{F}_{E}^{l}(x)$ is the map obtained after applying the first $l$ of them, then $\overline{F}_{E}^{l}(x)\llcorner\partial B$ is supported in $\{\phi(y_{j},-\cos(L),0):j\geq l+1\}$.

    We construct each $\overline{H}_{E}^{j}$ as follows. Denote $\overline{y}_{j}=(y_{j},-\cos(L),0)$. We consider an embedded curve $\gamma$ in $S^{n-1}_{L}=\partial U\setminus B$ connecting $\overline{y}_{j}$ with $\overline{y}_{j+1}$ and take a small tubular neighborhood $N_{\gamma}$ of $\gamma$ in $\overline{U}$ which does not contain any $\overline{y}_{l}$ for $l\notin\{j,j+1\}$. Consider a diffeomorphism between $N_{\gamma}$ and $\mathbb{D}^{n-1}\times[0,1]$ and choose coordinates $(u,s)\in\mathbb{D}^{n-1}\times[0,1]$. We define $\overline{H}_{E}^{j}$ in $N_{\gamma}$ in a way so that points with $\vert u\vert\leq\frac{1}{2}$ are retracted to $\mathbb{D}^{n}\times\{1\}$ and in $\frac{1}{2}\leq|u|\leq 1$ we interpolate between that map and the identity as we approach $\partial\mathbb{D}^{n}\times[0,1]$ so that we can extend $\overline{H}_{E}^{j}$ by the identity in $\overline{U}\setminus N_{\gamma}$. To be precise, we consider the following homotopy
    \[
    \overline{H}_{E}^{j}((u,s),t)=\begin{cases}
        (u,t) & \text{if } |u|\leq\frac{1}{2} \text{ and }s\leq t\\
        (u,a(|u|)t+(1-a(|u|))s) & \text{if }\frac{1}{2}\leq|u|\leq 1 \text{ and }s\leq t\\
        (u,s) & \text{if } s\geq t
        
    \end{cases}
    \]
    where $a:\mathbb{R}\to\mathbb{R}$ is a linear function with $a(\frac{1}{2})=1$ and $a(1)=0$. Then we can see that $\overline{H}_{E}^{j}$ is the identity in $\partial\mathbb{D}^{n-1}\times[0,1]\cup\mathbb{D}^{n-1}\times[1]$ providing a well defined homotopy in $\overline{U}$. In addition, $\overline{H}_{E}^{j}((0,0),t)=(0,t)$ which means that the trajectory of $\overline{y}_{j}$ along $\overline{H}_{E}^{j}$ is precisely $\gamma$ as desired. Thus the number of points in $\interior(\mathbb{D}^{n})\cup B$ gets increased by at most $1$ during the homotopy $\overline{H}_{E}^{j}(\overline{F}_{j-1}^{E}(x),t)$ and moreover at $t=1$ we have
    \begin{equation*}
        \mass(\overline{F}_{E}^{j}(x)\llcorner(\interior(\mathbb{D}^{n})\cup B))\leq\mass(\overline{F}_{E}^{j-1}(x)\llcorner(\interior(\mathbb{D}^{n})\cup B))
    \end{equation*}
    as the only point which may have entered that region during the homotopy is $\overline{y}_{j}$ which at the end of the homotopy overlaps with $\overline{y}_{j+1}$. We set $F_{E}^{2}=\overline{F}_{E}^{N-1}$, which by construction has the property that
    \begin{equation*}
        \mass(F_{E}^{2}(x)\llcorner(\interior(\mathbb{D}^{n})\cup B))\leq M_{E}+k+1
    \end{equation*}
    and verifies \ref{ite : 10.1} to \ref{ite : 10.4}. If we apply a radial contraction of $\overline{U}$ towards $e$  (specifically $(u,s,t)\mapsto(u,s,(1-t')t+t'(-1))$ for $t'\in[0,1]$), we obtain a contraction of $F_{E}^{2}$ through families verifying \ref{ite : 10.1} to \ref{ite : 10.5}, which allows to extend $F'$ to all of $E$.

\end{proof}

\subsection{Parametric isoperimetric inequality in compact manifolds}\label{Section Isoperimetric in manifolds}

The goal of this section is to prove the Parametric Isoperimetric Inequality for families of $0$-cycles supported in a Riemannian polyhedron (which yields the result in compact piecewise smooth Riemannian manifolds).

\begin{theorem}\label{Isoperimetric inequality 2}
    Let $P$ be an $m$-dimensional Riemannian polyhedron and let $P'$ be a subpolyhedron. There exists a constant $C(P)$ depending only on $P$ such that the following holds. Let $F:X^{p}\to\mathcal{Z}_{0}(P,P')$ be a contractible family of $0$-cycles on $P$ and denote 
    \begin{equation*}
            \mass_{0}=\sup\{\mass(F(x)):x\in X\}.
        \end{equation*}
    Then there exists a continuous family $G:X^{p}\to\mathcal{I}_{1}(P,P')$ such that
    \begin{enumerate}
        \item $\partial G(x)=F(x)$.
        \item $\mass(G(x))\leq C(P)(\mass_{0}p^{-\frac{1}{m}}+p^{\frac{m-1}{m}})$.
    \end{enumerate}
\end{theorem}

We start with a contractible family $F:X^{p}\to\mathcal{Z}_{0}(P,\partial P')$ supported in a Riemannian polyhedron $P$. We fix a smooth triangulation of $P$. In order to obtain a short filling $G$ of $F$, we will first start with a possibly very long filling $G:X^{p}\to\mathcal{Z}_{1}(M,\partial M)$, whose existence is guaranteed by the fact that $F$ is contractible. Then we will proceed to deform $G$ cell by cell, so that for each $m$-dimensional cell $Q$ of $M$ we homotop $G|_{Q}$ to the short filling $G_{Q}$ constructed in Theorem \ref{Isoperimetric inequality 1} (here we are using that $(Q,g)$ is bilipschitz homeomorphic to the Euclidean disk $\mathbb{D}^{n}$). The next proposition shows how to do such deformation. After doing that on each top dimensional cell $Q$ of $M$, $G$ could still have very large mass due to its part supported in the codimension-$1$ skeleton $P_{m-1}$ of $M$. Then what we do is to apply the same type of deformation at each $(m-1)$-cell of $M_{m-1}$, reducing the problem to the $(m-2)$ skeleton of $M$. And then we proceed inductively. This is explained after proving \ref{Prop Q}, completing the proof of Theorem \ref{Isoperimetric inequality 2}.

\begin{proposition}\label{Prop Q}
    Let $(P,g)$ be an $m$-dimensional Riemannian polyhedron. Let $G:X^{p}\to\mathcal{I}_{1}(P)$ be a continuous $\delta$-localized family of $1$-chains. Denote $F(x)=\partial G(x)$. Let $Q$ be a closed top dimensional cell of $P$ and let
    \begin{equation*}
        \mass_{Q}=\sup\{\mass(F(x)\llcorner\interior(Q)):x\in X\}.
    \end{equation*}
    There exists a homotopy $G:X^{p}\times[0,1]\to\mathcal{I}_{1}(P)$ with $\partial G(x,t)=F(x)$ for every $t\in[0,1]$ such that if $\overline{G}(x)=G(x,1)$ then
    \begin{enumerate}
    \item $G(x,t)-G(x)$ is supported in $Q$ for every $t\in[0,1]$.
    \item There exist continuous families $\overline{G}_{Q}:X\to\mathcal{I}_{1}(Q)$ and $\overline{G}_{Q^{C}}:X\to\mathcal{I}_{1}(M\setminus\interior{(Q)})$ such that
    \begin{equation*}
        \overline{G}(x)=\overline{G}_{Q}(x)+\overline{G}_{Q^{C}}(x),
    \end{equation*}
    
    \begin{equation*}
        \mass(\overline{G}_{Q}(x))\leq C(Q)(\mass_{Q}p^{-\frac{1}{m}}+p^{\frac{m-1}{m}})+C'(Q)(2\mass_{Q}+p+1)\delta'
    \end{equation*}
    and
    \begin{equation*}
        \mass(\partial \overline{G}_{Q}(x)\llcorner\partial Q)\leq 4(\mass_{Q}+p+1). 
    \end{equation*}
    
\end{enumerate}
In the previous, $\delta'=c(n,p)\delta$ and if $\rho_{Q}:(Q,\partial Q)\to(\mathbb{D}^{n},\partial\mathbb{D}^{n})$ is the optimal bilipschitz map (in terms of minimizing $\Lip(\rho_{Q})$), the constant $C(Q)$ depends linearly on $\Lip(\rho_{Q})$ and $C'(Q)=\Lip(\rho_{Q})\Lip(\rho_{Q}^{-1})$. 
\end{proposition}

\begin{proof}[Proof of Proposition \ref{Prop Q}]

Let $m=\dim(P)$. We start by applying Lemma \ref{Lemma G_Q} to write $G(x)=G_{Q}(x)+G_{Q^{C}}(x)$ where both $G_{Q}$ and $G_{Q^{C}}$ are continuous families of absolute $1$-cycles and they are supported in $Q$ and $\interior(Q)^{C}$ respectively. Let $F_{Q}(x)=\partial G_{Q}(x)$. In the following we identify $(Q,\partial Q)$ with $(\mathbb{D}^{m},\partial\mathbb{D}^{m})$ by a bilipschitz map with optimal Lipschitz constant. By Proposition \ref{Avoid ball D^n}, there exists a small ball $B\subseteq\partial Q$, a ball $\tilde{B}$ in $Q$ with $\tilde{B}\cap \partial Q\supseteq B$ and $\radius(\tilde{B})\leq\delta'=c(n,p)\delta$ and a family $F_{Q}':X^{p}\to\mathcal{Z}_{0}(Q)$ which coincides with $F_{Q}$ outside of $\tilde{B}$ such that
\begin{equation*}
    \mass(F'_{Q}(x)\llcorner Q^{B})\leq \mass_{Q}+p+1
\end{equation*}
where $Q^{B}=\interior(Q)\cup B$. As $F_{Q}'-F_{Q}$ is supported in $\tilde{B}$ which is convex, by taking cones centered at a point in $\tilde{B}$ we can construct a family $R_{Q}:X^{p}\to\mathcal{I}_{1}(\tilde{B})$ such that $\partial R_{Q}(x)=F'_{Q}(x)-F_{Q}(x)$. Let $S_{Q}(x)=G_{Q}(x)+R_{Q}(x)$. Observe that
\begin{align*}
    \partial S_{Q}(x) & =\partial G_{Q}(x)+\partial R_{Q}(x)\\
    & = F_{Q}(x)+(F'_{Q}(x)-F_{Q}(x))\\
    & = F'_{Q}(x)
\end{align*}
and hence $\mass(\partial S_{Q}(x)\llcorner Q^{B})\leq \mass_{Q}+p+1$. Our goal is to replace the family $S_{Q}$ by a short family (one whose mass is bounded by $C(Q)(\mass_{Q}p^{-\frac{1}{m}}+p^{\frac{m-1}{m}})$) with boundary $F'_{Q}$ using the construction in Theorem \ref{Isoperimetric inequality 1}. To be precise, we will construct a homotopy $S_{Q}:X^{p}\times[0,1]\to\mathcal{I}_{1}(Q)$ with $S_{Q}(x,0)=S_{Q}(x)$ such that
\begin{enumerate}
    \item $\partial S_{Q}(x,t)=F'_{Q}(x)$ for every $t\in[0,1]$.
    \item $\mass(S_{Q}(x,1)\llcorner\interior(Q))\leq C(Q)(\mass_{Q}p^{-\frac{1}{n}}+p^{\frac{n-1}{n}})$, in fact $\overline{S}_{Q}(x)=S_{Q}(x,1)\llcorner\interior(Q)$ is equal to the family constructed in Theorem \ref{Isoperimetric inequality 1} (in particular, it is continuous) and hence it also has the property that $\mass(\partial\overline{S}_{Q}(x)\llcorner\partial Q)\leq 2(\mass_{Q}+p+1)$.
\end{enumerate}
We will also construct a homotopy $R_{Q}:X^{p}\times[0,1]\to\mathcal{I}_{1}(Q)$ such that
\begin{enumerate}
    \item $\partial R_{Q}(x,t)=F_{Q}'(x)-F_{Q}(x)$ for every $t\in[0,1]$.
    \item Let $\overline{R}_{Q}(x)=R_{Q}(x,1)\llcorner\interior Q$. Then $\mass(\overline{R}_{Q}(x))\leq C'(Q)(2\mass_{Q}+p+1)\delta'$ and $\mass(\partial \overline{R}_{Q}(x)\llcorner\partial Q)\leq 2\mass_{Q}+p+1$.
\end{enumerate}
Hence the homotopy $G_{Q}(x,t)=S_{Q}(x,t)+R_{Q}(x,t)$ will allow to replace $G_{Q}(x)\llcorner\interior(Q)$ by the absolute family $\overline{G}_{Q}(x)=G_{Q}(x,1)\llcorner\interior(Q)=\overline{S}_{Q}(x)+\overline{R}_{Q}(x)$ which has controlled mass and boundary mass. Observe that $\partial G_{Q}(x,t)=F_{Q}(x)=\partial G_{Q}(x)$ hence in fact $G_{Q}$ extends to a homotopy $G(x,t)=G_{Q}(x,t)+G_{Q^{C}}(x)$ with $\partial G(x,t)=F(x)$ for every $t\in[0,1]$. This homotopy will provide the desired $\overline{G}(x)=G(x,1)$.

We introduce the following notation. Given two points $z,w\in Q$ let $\Xi(z,w)$ be the segment connecting them. Given a $0$-chain $x$ in $Q$ and a Lipschitz map $T:Q\to Q$ define
\begin{equation*}
    \Xi(x,T(x))=\sum_{z\in\support(x)}\Xi(z,T(z)).
\end{equation*}
Let $P_{0}$ be the unique point for which the tangent cone to $\partial Q$ centered at $P_{0}$ intersects $\partial Q$ at $\partial B$. For each $z\in Q$, let $H(z)$ be the point in the line through $P_{0}$ and $z$ and in $\partial Q$ which is furthest from $P_{0}$. Let $H(z,t)$ be a constant speed parametrization of the segment $\Xi(z,H(z))$ being $H(z,0)=z$ and $H(z,1)=H(z)$. Observe that $H:Q\times[0,1]\to Q$ is a Lipschitz homotopy which is the identity in $\partial Q\setminus B$. For $t\in[0,\frac{1}{2}]$ define
\begin{equation*}
    S_{Q}(x,t)=H_{2t}(S_{Q}(x))+\Xi(H_{2t}(F'_{Q}(x)),F'_{Q}(x))
\end{equation*}
and for $t\in[\frac{1}{2},1]$ set
\begin{equation*}
    S_{Q}(x,t)=\Psi_{2t-1}(\tilde{S}_{Q}(x))+\Xi(\Psi_{2t-1}(F'_{Q}(x)),F'_{Q}(x))
\end{equation*}
where $\tilde{S}_{Q}(x)=S_{Q}(x,\frac{1}{2})$ and $\{\Psi_{t}:0\leq t\leq 1\}$ is an homotopy from $\Psi_{0}=id$ to the PL map $\Psi_{1}=\Psi$ from Lemma \ref{Deformation Lemma 2}. It is clear that $\partial S_{Q}(x,t)=F_{Q}'(x)$ for every $t\in[0,1]$. In addition, if $\overline{S}_{Q}(x)=S_{Q}(x,1)\llcorner\interior (Q)$ as defined above then $\overline{S}_{Q}$ is equal to the map $G$ constructed in Theorem \ref{Isoperimetric inequality 1} with $r=p^{-\frac{1}{m}}$ and $F=F_{Q}'$ (in particular, it is a continuous family of absolute $1$-chains), hence
\begin{equation*}
    \mass(\overline{S}_{Q}(x))\leq C(Q)(\mass_{Q}p^{-\frac{1}{m}}+p^{\frac{m-1}{m}})
\end{equation*}
and 
\begin{equation*}
    \mass(\partial \overline{S}_{Q}(x)\llcorner\partial Q)\leq 2(\mass_{Q}+p+1). 
\end{equation*}

Now we proceed to define $R_{Q}(x,t)$. Let $q$ be the center of $Q$. Consider the map $I:\tilde{B}\to \tilde{B}\cap\partial Q$ which maps each point $z\in\tilde{B}$ to the point $I(z)$ in the line through $q$ and $z$ and in $\partial \tilde{B}$ which is furthest from $q$ (and hence it is in $\tilde{B}\cap\partial Q$). Let $I_{t}(z)$ be a constant speed parametrization of the segment from $z$ to $I(z)$. This yields a Lipschitz homotopy $I:\tilde{B}\times[0,1]\to\tilde{B}$ with $I(z,0)=z$ and $I(z,1)=I(z)$ for each $z\in\tilde{B}$. This is analog to the map $H$ previously defined but replacing $Q$ by $\tilde{B}$ and $P_{0}$ by $q$. We proceed to replace $R_{Q}(x)$ by the collection of segments $\Xi(\partial R_{Q}(x),I(\partial R_{Q}(x)))$ in the following way
\begin{equation*}
    R_{Q}(x,t)=I_{t}(R_{Q}(x))+\Xi(I_{t}(F'_{Q}(x)-F_{Q}(x)),F'_{Q}(x)-F_{Q}(x)).
\end{equation*}

Observe that the number of segments in $\Xi(I_{t}(F'_{Q}(x)-F_{Q}(x)),F'_{Q}(x)-F_{Q}(x))$ is equal to $\mass((F'_{Q}(x)-F_{Q}(x))\llcorner\interior(Q))$ because points in $\partial\tilde{B}\cap\partial Q$ are fixed by $I$ and $(F'_{Q}(x)-F_{Q}(x))\llcorner(\partial\tilde{B}\setminus\partial Q)=0$. As $\mass(F_{Q}(x)\llcorner\interior(Q))\leq \mass_{Q}$ and $\mass(F'_{Q}(x)\llcorner\interior{Q})\leq \mass_{Q}+p+1$, we deduce that
\begin{equation*}
    \mass\big (\Xi(I_{1}(F'_{Q}(x)-F_{Q}(x)),F'_{Q}(x)-F_{Q}(x))\big)\leq(2\mass_{Q}+p+1)\delta'.
\end{equation*}
From our definition, it also follows that
\begin{align*}
    \mass(\partial\Xi(I_{1}(F'_{Q}(x)-F_{Q}(x)),F'_{Q}(x)-F_{Q}(x))\llcorner\partial Q) & \leq\mass((F_{Q}(x)-F_{Q}'(x))\llcorner\interior(Q))\\
    & \leq 2\mass_{Q}+p+1.
\end{align*}
As $\overline{R}_{Q}(x)=R_{Q}(x,1)\llcorner\interior{Q}=\Xi(I_{1}(F'_{Q}(x)-F_{Q}(x)),F'_{Q}(x)-F_{Q}(x))$ the homotopy $R_{Q}(x,t)$ has the desired properties.
\end{proof}

\begin{proof}[Proof of Theorem \ref{Isoperimetric inequality 2}]
    By induction on $m=\dim(P)$. If $m=1$, let $C(P)$ be the length of the $1$-dimensional polyhedron $P$. Using \cite{GL22}[Proposition~3.5], we can construct a family $G:X\to\mathcal{I}_{1}(P)$ such that $\partial G(x)=F(x)$ (more details on how to do this are provided in the proof of the inductive step) and observe that as we are taking coefficients in $\mathbb{Z}_{2}$ it holds $\mass(G(x))\leq C(P)$, verifying the required estimate.
    Suppose the result is true for $m-1$ and let $P$ be an $m$-dimensional Riemannian polyhedron. By \cite{GL22}[Proposition~3.5], for every $\eta>0$ there exists a continuous family $\tilde{G}:X\to\mathcal{I}_{1}(M)$ such that 
    \begin{equation*}
        \mathcal{F}(\partial \tilde{G}(x)-F(x))<\eta.
    \end{equation*}
    Take $\eta<<p^{\frac{m-1}{m}}$ and use Theorem \ref{Thm continuous delta approx chains} to approximate $\tilde{G}$ by a $\delta$-localized family $G':X^{p}\to\mathcal{I}_{1}(M)$ such that $\mathcal{F}(G'(x),\tilde{G}(x))\leq\eta$ and if $F'=\partial G'$ then
    \begin{equation*}
        \sup_{x\in X}\{\mass(F'(x))\}\leq\mass_{0}.
    \end{equation*}
    In the previous, we used that $F'$ is a family of $0$-cycles. Notice that $\delta$ can also be taken to be arbitrarily small. As 
    \begin{align*}
        \mathcal{F}(F(x),F'(x)) & \leq\mathcal{F}(F(x),\partial\tilde{G}(x))+\mathcal{F}(\partial\tilde{G}(x),\partial G'(x))\\
        & \leq \eta+\mathcal{F}(\tilde{G}(x),G'(x))\\
        & \leq 2\eta,
    \end{align*}
    by Theorem \ref{Prop filling small continuous families} we can construct a family $G'':X^{p}\to\mathcal{I}_{1}(M)$ such that $\partial G''(x)=F(x)-F'(x)$ and $\mass(G''(x))\leq \varepsilon$ for a small $\varepsilon>0$ provided $\eta$ is small enough. Therefore if $G_{1}(x)=G'(x)+G''(x)$ then $\partial G_{1}(x)=F(x)$, $G''(x)$ has very small mass, $G'$ and $F'=\partial G'$ are $\delta$-localized and $F'$ has mass bounded by $\mass_{0}$. Therefore, it is enough to prove the theorem for the family $F'$ by deforming its $\delta$-localized filling $G'$ cell by cell of $P$.
    
    Let $\{Q_{j}\}_{1\leq j\leq J}$ be the set of $m$-cells of $P$. Applying Proposition \ref{Prop Q} to each $Q_{j}$, we can homotop $G'$ to a family $\overline{G}$ with $\partial\overline{G}(x)=F'(x)$ such that if $\overline{G}_{j}(x)=\overline{G}(x)\llcorner\interior(Q_{j})$ then
    \begin{equation*}
        \mass(\overline{G}_{j}(x))\leq C(Q_{j})(\mass_{j}p^{-\frac{1}{m}}+p^{\frac{m-1}{m}})
    \end{equation*}
    (plus a small error in terms of $\delta$ which will converge to $0$ as we choose $\delta\to 0$) and
    \begin{equation*}
        \mass(\partial\overline{G}_{j}(x)\llcorner\partial Q)\leq 4(\mass_{j}+p+1)
    \end{equation*}
    where
    \begin{equation*}
        \mass_{j}=\sup\{\mass(F'(x)\llcorner\interior(Q_{j})):x\in X\}.
    \end{equation*}
    Let $\hat{G}(x)=\overline{G}(x)-\sum_{j=1}^{J}\overline{G}_{j}(x)$ and $\hat{F}=\partial\hat{G}$. Let $P'$ denote the $(m-1)$-skeleton of $P$. Observe that as $\support(\hat{G}(x))\subseteq P'$, $\hat{F}$ is a contractible family of $0$-cycles supported in $P'$ verifying
    \begin{equation*}
        \mass(\hat{F}(x))\leq \tilde{\mass}_{0}=4\mass_{0}+4J(p+1).
    \end{equation*}
    By inductive hypothesis, there exists a continuous family $G_{0}:X\to\mathcal{I}_{1}(P')$ such that $\partial G_{0}(x)=\hat{F}(x)$ and
    \begin{align*}
        \mass(G_{0}(x)) & \leq C(P')(\tilde{\mass}_{0}p^{-\frac{1}{m-1}}+p^{\frac{m-2}{m-1}})\\
        &\leq C(P')(\tilde{\mass}_{0}p^{-\frac{1}{m}}+p^{\frac{m-1}{m}})\\
        & \leq 9J C(P')(\mass_{0}p^{-\frac{1}{m}}+p^{\frac{m-1}{m}}).
    \end{align*}
    If we set
    \begin{equation*}
        G(x)=G_{0}(x)+G''(x)+\sum_{j=1}^{J}\overline{G}_{j}(x)
    \end{equation*}
    and
\begin{equation*}
        C(P)=9JC(P')+\sum_{j=1}^{J}C(Q_{j}).
    \end{equation*}
    we obtain the desired result.
\end{proof}

\section{Weyl Law for the Volume Spectrum}\label{Section Weyl law}

In this section we prove the Weyl Law for the Volume Spectrum for $1$-cycles (Theorem \ref{Thm Weyl law}), using the Parametric Coarea Inequality and the Parametric Isoperimetric Inequality (Theorem \ref{Thm Parametric Coarea}). Fix an $n$-dimensional compact Riemannian manifold $(M^{n},g)$, possibly with boundary (in fact, this proof works for almost $1$-Lipschitz triangulable piecewise smooth Riemannian manifolds, see Definition \ref{Def almost 1-Lip triangulable}). We know that the Weyl law holds for contractible domains of $\mathbb{R}^{n}$. The previous means that there exists a universal constant $\alpha(n)$ such that if $U$ is a compact contractible domain in $\mathbb{R}^{n}$ with Lipschitz boundary (or Lipschitz domain according to the terminology in \cite{LMN}) then
\begin{equation*}
    \lim_{p\to\infty}\omega_{p}^{1}(U,\partial U)p^{-\frac{n-1}{n}}=\alpha(n)\Vol(U)^{\frac{1}{n}}.
\end{equation*}
In \cite{LMN} it was also showed that for a compact $n$-manifold $(M,g)$,
\begin{equation*}
    \alpha(n)\Vol(M,g)^{\frac{1}{n}}\leq\liminf_{p\to\infty}\omega^{1}_{p}(M,g)p^{-\frac{n-1}{n}}
\end{equation*}
so our goal will be to prove that
\begin{equation*}
    \limsup_{p\to\infty}\omega_{p}^{1}(M,g)p^{-\frac{n-1}{n}}\leq\alpha(n)\Vol(M,g)^{\frac{1}{n}}.
\end{equation*}
With that purpose, we want to construct for each $p\in\mathbb{N}$ an ``almost optimal'' (at least asymptotically up to an error of the order $o(p^{\frac{n-1}{n}})$) $p$-sweepout of $(M,g)$, to obtain good upper bounds for $\omega_{p}^{1}(M,g)$.

Fix $\varepsilon>0$. Using \cite{Bowditch}[Theorem 1.2 and 1.3], consider a $(1+\frac{\varepsilon}{2})$-bilipschitz smooth triangulation of $M$ via simplices $\{Q_{i}\}_{1\leq i\leq N}$ (see Definition \ref{Def smooth triangulation}). This implies that each $Q_{i}$ is $(1+\frac{\varepsilon}{2})$-bilipschitz diffeomorphic to a compact contractible domain $U_{i}$ in $\mathbb{R}^{n}$ (more precisely, each $U_{i}$ is a linear $n$-simplex). Denote $\Phi_{i}:Q_{i}\to U_{i}$ such bilipschitz diffeomorphisms. We can translate the $U_{i}$'s if necessary and connect each $U_{i}$, $i>1$ with $U_{1}$ via a tiny tube $T_{i}$ obtaining a compact contractible domain $U=U_{1}\cup\bigcup_{i=2}^{N}U_{i}\cup T_{i}$ with Lipschitz boundary such that $\Vol(U)<(1+\varepsilon)^{n}\Vol(M,g)$.

Now fix $p\in\mathbb{N}$. By definition of $p$-width, we can find a $p(n-1)$-dimensional cubical complex $X(p)$ and a $p$-sweepout $\tilde{F}^{p}:X(p)\to\mathcal{I}_{1}(U,\partial U)$ of $U$ such that 
\begin{equation*}
    \sup_{x\in X(p)}\mass(\tilde{F}^{p}(x))\leq (1+\varepsilon)\omega^{1}_{p}(U,\partial U).
\end{equation*}
Denote $F^{p}_{i}:X(p)\to\mathcal{I}_{1}(Q_{i},\partial Q_{i})$ the map $F^{p}_{i}(x)=\Phi_{i}^{-1}(\tilde{F}^{p}(x)\llcorner U_{i})$ and $F^{p}(x)=\sum_{i=1}^{N}F^{p}_{i}(x)$. Notice that $\mass(F^{p}_{i}(x))\leq (1+\varepsilon)\mass(\tilde{F}^{p}(x)\llcorner Q_{i})$ and therefore
\begin{equation}\label{Equation Fp wp}
    \mass(F^{p}(x))\leq (1+\varepsilon)\mass(\tilde{F}^{p}(x))\leq (1+\varepsilon)^{2}\omega_{p}^{1}(U,\partial U).
\end{equation}
However, $F^{p}$ is not a family of relative cycles in $(M,\partial M)$ because it may have an arbitrarily large boundary supported in $\bigcup_{i=1}^{N}\partial Q_{i}$ which is not contained in $\partial M$. Denote $V_{m}=\bigcup_{i=1}^{m}Q_{i}$. To fix the previous, we are going to construct inductively a short $p$-sweepout $G_{m}$ of $(V_{m},\partial V_{m})$ which will finally provide a short $p$-sweepout $G=G_{N}$ of $(M,\partial M)$. To be precise, we want $G_{m}$ to obey the following mass bounds
\begin{enumerate}
    \item $\mass(G^{p}_{m}(x))\leq (1+\gamma_{p})\mass(F^{p}(x)\llcorner V_{m})+C(Q_{1},...,Q_{m})p^{1+\alpha-\frac{1}{n-1}}$
    \item $\mass(\partial G^{p}_{m}(x))\leq C(Q_{1},,,.Q_{m})p^{1+\alpha}$
\end{enumerate}
where $\alpha$ and the $\gamma_{p}$ come from the Parametric Coarea inequality, and it holds $\lim_{p\to\infty}\gamma_{p}=0$.

In order to define the $G_{m}$, we will need to apply Theorem \ref{Thm Parametric Coarea} to each $F^{p}_{i}$ (this can be done because the theorem holds for simplices as they are PL submanifolds of $\mathbb{R}^{n}$, see \cite{StaCoarea}). For that purpose we fix $0<\alpha<\frac{1}{n-1}-\frac{1}{n}$. We obtain a new $p$-sweepout $\overline{F
}^{p}_{i}:X(p)\to\mathcal{Z}_{1}(Q_{i},\partial Q_{i})$ which is induced by a family of absolute chains $\overline{F}^{p}_{i}:X(p)\to\mathcal{I}_{1}(Q_{i})$ with $\support(\partial\overline{F}^{p}_{i}(x))\subseteq\partial Q_{i}$ such that
\begin{enumerate}
    \item $\mass(\overline{F}_{i}^{p}(x))\leq(1+\gamma_{p})\mass(F^{p}_{i}(x))+C(\partial Q_{i})p^{\frac{n-2}{n-1}+\alpha}$.
    \item $\mass(\partial\overline{F}^{p}_{i}(x))\leq C(\partial Q_{i})p^{1+\alpha}$.
\end{enumerate}

Next we proceed to construct the $G^{p}_{m}$ inductively. For $m=1$, we set $G^{p}_{1}=\overline{F}^{p}_{1}$ and $C(Q_{1})=C(\partial Q_{1})$. Suppose we already constructed $G^{p}_{m}$ verifying
\begin{enumerate}
    \item $\mass(G^{p}_{m}(x))\leq (1+\gamma_{p})\mass(F^{p}(x)\llcorner V_{m})+C(Q_{1},...,Q_{m})p^{1+\alpha-\frac{1}{n-1}}$
    \item $\mass(\partial G^{p}_{m}(x))\leq C(Q_{1},,,.Q_{m})p^{1+\alpha}$
\end{enumerate}
Denote $A_{m}=\partial V_{m}\cap\partial Q_{m+1}$. We want to glue the families $G_{m}^{p}$ and $\overline{F}_{m+1}^{p}$ along $A_{m}$ to get a $p$-sweepout of $V_{m}\cup Q_{m+1}=A_{m+1}$. We know that $\partial G_{m}^{p}(x)\llcorner A_{m}$ and $\partial\tilde{F}^{p}_{m+1}(x)\llcorner A_{m}$ are $p$-sweepouts of $A_{m}$ by $0$-cycles. Despite they may not coincide (preventing us from gluing $G_{m}^{p}$ and $F_{m+1}^{p}$ directly), their difference $\eta^{p}_{m}:X(p)\to\mathcal{Z}_{0}(A_{m},\partial A_{m})$ is a contractible family of $0$-cycles in $(A_{m},\partial A_{m})$ with
\begin{equation*}
    \mass(\eta^{p}_{m}(x))\leq \tilde{C}(Q_{1},...,Q_{m+1})p^{1+\alpha}.
\end{equation*}
By Theorem \ref{Isoperimetric inequality 2}, we can find a continuous filling $\tau^{p}_{m}:X(p)\to\mathcal{I}_{1}(A_{m})$ such that
\begin{equation*}
    \mass(\tau^{p}_{m}(x))\leq C(A_{m})(\tilde{C}(Q_{1},...,Q_{m+1})p^{1+\alpha-\frac{1}{n-1}}+p^{\frac{n-2}{n-1}})\leq \overline{C}(Q_{1},...,Q_{m+1})p^{1+\alpha-\frac{1}{n-1}}.
\end{equation*}
Thus if we set $G^{p}_{m+1}:=G^{p}_{m}+\overline{F}^{p}_{m+1}+\tau^{p}_{m}$, we obtain a $p$-sweepout of $(V_{m+1},\partial V_{m+1})$ with the desired bounds for the mass and the mass of the boundary.

Let $G^{p}=G^{p}_{N}$, by construction we know that it is a $p$-sweepout of $(M,\partial M)$ and
\begin{equation*}
    \mass(G^{p}(x))\leq (1+\gamma_{p})\mass(F^{p}(x))+C(Q_{1},...,Q_{N})p^{1+\alpha-\frac{1}{n-1}}.
\end{equation*}
Thus
\begin{equation*}
   \omega_{p}^{1}(M,g)\leq (1+\gamma_{p})\sup_{x\in X(p)}\mass(F^{p}(x))+C(Q_{1},...,Q_{N})p^{1+\alpha-\frac{1}{n-1}}.
\end{equation*}
Using the fact that $\alpha<\frac{1}{n-1}-\frac{1}{n}$, the previous implies that
\begin{align*}
    \limsup_{p\to\infty}\omega_{p}^{1}(M,g)p^{-\frac{n-1}{n}} & \leq\limsup_{p\to\infty}(1+\gamma_{p})\big(\sup_{x\in X(p)}\{\mass(F^{p}(x))\}\big )p^{-\frac{n-1}{n}}+C(Q_{1},...,Q_{m})p^{\alpha+\frac{1}{n}-\frac{1}{n-1}}\\
    & =\limsup_{p\to\infty}\big(\sup_{x\in X(p)}\mass(F^{p}(x))\}\big )p^{-\frac{n-1}{n}}\\
    & \leq (1+\varepsilon)^{2}\limsup_{p\to\infty}\omega_{1}^{p}(U)p^{-\frac{n-1}{n}}\\
    & =(1+\varepsilon)^{2}\alpha(n)\Vol(U)^{\frac{1}{n}}\\
    & \leq (1+\varepsilon)^{n+2}\alpha(n)\Vol(M,g)^{\frac{1}{n}}
\end{align*}
by (\ref{Equation Fp wp}). As that holds for every $\varepsilon>0$, we deduce
\begin{equation*}
    \limsup_{p\to\infty}\omega_{p}^{1}(M,g)p^{-\frac{n-1}{n}}\leq\alpha(n)\Vol(M,g)^{\frac{1}{n}}
\end{equation*}
which completes the proof of the Weyl Law.

\bibliography{Bibliography}

@article{FA62,
    author="F. Almgren",
    title="The homotopy groups of the integral cycle groups",
    journal="Topology",
    year="1962",
    pages="257-299"
}

@article{FA65,
    author="F. Almgren",
    title="The theory of varifolds",
    journal="Mimeographed notes, Princeton",
    year="1965"
}

@article{Fleming66,
    author ="W. H. Fleming" ,
    title = "Flat chains over a finite coefficient group",
    journal = "Transactions of the American Mathematical Society",
    volume="121",
    number="1",
    year="1966",
    pages="160-186"
}

@book{FedererGMT,
    author = "H. Federer",
    title = "Geometric Measure Theory",
    series= "Die Grundlehren der mathematischen Wissenschaften, Band 153",
    publisher = "Springer-Verlag New York Inc.",
    year = "1969"
}

@article{StaCoarea,
    author="B. Staffa",
    title="Parametric Coarea Inequality for $1$-cycles",
    journal="preprint",
    year="2024",
    note={}
}

@article{LMN,
    author="Y. Liokumovich and F. C. Marques and A. Neves",
    title="Weyl law for the volume spectrum",
    journal="Annals of mathematics",
    volume="187",
    year="2018",
    pages="933-961"
}

@article{LS,
    author="Yevgeny Liokumovich and Bruno Staffa",
    title="Generic density of geodesic nets",
    journal="Selecta Mathematica - New Series",
    year="2024",
    volume="30",
    number="14",
    note="\url{https://link.springer.com/article/10.1007/s00029-023-00901-7}"
}

@article{LiSta,
    author= "Xinze Li and Bruno Staffa",
    title="On the equidistribution of closed geodesics and geodesic nets",
    journal="Transactions of the American Mathematical Society",
    year="2023",
    volume="376",
    number="12",
    pages="8825-8855",
    note="\url{https://www.ams.org/journals/tran/2023-376-12/S0002-9947-2023-09028-6/S0002-9947-2023-09028-6.pdf}"
}

@article{GL22,
    author="L. Guth and Y. Liokumovich",
    title="Parametric inequalites and {Weyl} law for the volume spectrum",
    journal="Geometry and Topology",
    year="2025",
    volume="29",
    number="2",
    pages="863-902",
    note={\url{https://msp.org/gt/2025/29-2/gt-v29-n2-p06-s.pdf}}
}

@article{GuthMinMax,
    author="L. Guth",
    title="Minimax problems related to cup powers and {Steenrod} squares",
    journal="Geometric and Functional
Analysis",
    volume="18",
    year="2009",
    pages="1917-1987"
}

@article{GuthWidthVolume,
    author ="L. Guth",
    title = "The Width-Volume Inequality",
    journal = "GAFA Geom. funct. anal.",
    volume="17",
    year = "2007",
    pages="1139-1179"
}

@article{JP73,
    author="J. T. Pitts",
    title="Regularity and singularity of one dimensional stationary integral varifolds on manifolds arising from variational methods in the large",
    journal="in Symposia Mathematica",
    volume="XIV",
    year="1973",
    pages="465-472"
}

@book{JP81,
    author = {J. T. Pitts},
    doi = {doi:10.1515/9781400856459},
    url = {https://doi.org/10.1515/9781400856459},
    title = {Existence and Regularity of Minimal Surfaces on {Riemannian} Manifolds},
    year = {1981},
    publisher = {Princeton University Press},
    ISBN = {9781400856459}
}

@article{SchoenSimon,
  title={Regularity of stable minimal hypersurfaces},
  author={R. M. Schoen and L. Simon},
  journal={Communications on Pure and Applied Mathematics},
  year={1981},
  volume={34},
  pages={741-797},
  url={https://api.semanticscholar.org/CorpusID:124924186}
}

@article{MNS,
author="F. C. Marques and A. Neves and A. Song",
title="Equidistribution of minimal hypersurfaces for generic metrics",
journal="Inventiones mathematicae",
volume="216",
year="2019",
pages="421-443"
}

@article{MN,
author = {F. Marques and A. Neves},
year = {2017},
month = {08},
pages = {},
title = {Existence of infinitely many minimal hypersurfaces in positive {Ricci} curvature},
volume = {209},
journal = {Inventiones mathematicae},
doi = {10.1007/s00222-017-0716-6}
}

@article{Willmore,
author = {F. Marques and A. Neves},
year = {2014},
month = {},
pages = {683-782},
title = {Min-Max theory and the {Willmore} conjecture},
volume = {179},
journal = {Annals of Mathematics},
doi = {https://doi.org/10.4007/annals.2014.179.2.6}
}

@article{Song,
author = {A. Song},
year = {2023},
month = {},
pages = {859-895},
title = {Existence of infinitely many minimal hypersurfaces in closed manifolds},
volume = {197},
journal = {Annals of Mathematics},
doi = {https://doi.org/10.4007/annals.2023.197.3.1}
}

@article{ChoMan,
author={O. Chodosh and C. Mantoulidis},
title={The p-widths of a surface},
journal={Publications Mathématiques de l'IHÉS},
year="2023",
volume="137",
pages="245-342",
note={\url{https://doi.org/10.1007/s10240-023-00141-7}}
}

@article{ChoMan20,
author = {Otis Chodosh and Christos Mantoulidis},
title = {{Minimal surfaces and the Allen--Cahn equation on 3-manifolds: index, multiplicity, and curvature estimates}},
volume = {191},
journal = {Annals of Mathematics},
number = {1},
publisher = {Department of Mathematics of Princeton University},
pages = {213 -- 328},
year = {2020},
doi = {10.4007/annals.2020.191.1.4},
URL = {https://doi.org/10.4007/annals.2020.191.1.4}
}

@article{IMN,
author="K. Irie and F. C. Marques and A. Neves",
title="Density of minimal hypersurfaces for generic metrics",
journal="Annals of Mathematics",
volume="187",
year="2018",
pages="963-972"
}

@article{Gromov02,
  title={Isoperimetry of waists and concentration of maps},
  author={M. Gromov},
  journal={Geometric and Functional Analysis},
  year={2003},
  volume={13},
  pages={178-215}
}

@article{Gromov09,
  title={Singularities, Expanders and Topology of Maps. Part 1: Homology Versus Volume in the Spaces of Cycles},
  author={M. Gromov},
  journal={Geometric and Functional Analysis},
  year={2009},
  volume={19},
  pages={743-841}
}

@article{Gromov86,
  title={Dimension, non-linear spectra and width},
  author={M. Gromov},
  journal={Geometric Aspects of Functional Analysis},
  year={1986/87},
  volume={1317},
  pages={132-184}
}

@article{Bowditch,
  title={Bilipschitz Triangulations of Riemannian manifolds},
  author={Brian H. Bowditch},
  journal="preprint",
  year={2020},
  note={\url{https://homepages.warwick.ac.uk/~masgak/papers/triangulations.pdf}}
}

@article{ZhouMult1,
author = {Xin Zhou},
title = {{On the Multiplicity One Conjecture in min-max theory}},
volume = {192},
journal = {Annals of Mathematics},
number = {3},
publisher = {Department of Mathematics of Princeton University},
pages = {767 -- 820},
year = {2020},
doi = {10.4007/annals.2020.192.3.3},
URL = {https://doi.org/10.4007/annals.2020.192.3.3}
}

@article{MNMult1,
title = {Morse index of multiplicity one min-max minimal hypersurfaces},
journal = {Advances in Mathematics},
volume = {378},
pages = {107527},
year = {2021},
issn = {0001-8708},
doi = {https://doi.org/10.1016/j.aim.2020.107527},
url = {https://www.sciencedirect.com/science/article/pii/S0001870820305557},
author = {Fernando C. Marques and André Neves}
}

@article{Akopyan,
    author="A. V. Akopyan",
    title="A {PL}-analogue for {Nash}-{Kuiper} Theorem",
    journal="preprint",
    year="2007",
    note={\url{https://users.mccme.ru/akopyan/papers.html}}
}

@article{Minemyer,
author = {Minemyer, B.},
title = {Isometric embeddings of polyhedra into {Euclidean} space},
journal = {Journal of Topology and Analysis},
volume = {07},
number = {04},
pages = {677-692},
year = {2015},
doi = {10.1142/S179352531550020X},
URL = {https://doi.org/10.1142/S179352531550020X},
eprint = {https://doi.org/10.1142/S179352531550020X}
}
\bibliographystyle{amsplain}

\newcommand{\Addresses}{{
  \bigskip
  \footnotesize

  \textsc{Department of Mathematics, Rice University, Houston, TX 77005, USA}\par\nopagebreak
  \textit{Email address}: \texttt{bruno.staffa@rice.edu}
  }}

\Addresses

\end{document}